\documentclass[11pt]{amsart}

\usepackage{amsthm}
\usepackage{amsfonts}
\usepackage{amsmath}
\usepackage{amssymb}
\usepackage{enumitem}
\usepackage{fancyhdr}
\usepackage{tikz}
\usepackage{tikz-cd}
\usepackage{hyperref}

\usetikzlibrary{arrows,decorations.markings}

\usepackage[numbers]{natbib}

\usepackage{mathtools}

\setlist[description]{leftmargin=\parindent,labelindent=\parindent}

\usepackage{graphicx}
\graphicspath{ {images/} }

\newtheorem{theorem}{Theorem}
\newtheorem{lemma}[theorem]{Lemma}
\newtheorem{proposition}[theorem]{Proposition}
\newtheorem{corollary}[theorem]{Corollary}

\numberwithin{theorem}{section}
\numberwithin{equation}{section}

\theoremstyle{definition}
\newtheorem{definition}[theorem]{Definition}

\newtheorem{remark}[theorem]{Remark}
\newtheorem{example}[theorem]{Example}
\newtheorem{assumption}[theorem]{Assumption}
\newtheorem{convention}[theorem]{Convention}

\title[Open enumerative mirror symmetry for lines in the mirror quintic]{Open enumerative mirror symmetry for lines in the mirror quintic}
\author{Sebastian Haney}
\address{Department of Mathematics, Columbia University, 2990 Broadway, New York, NY 10027}

\DeclareMathOperator\id{id}

\DeclareMathOperator\Log{Log}
\DeclareMathOperator\hol{hol}
\DeclareMathOperator\Int{Int}
\DeclareMathOperator\pt{pt}
\DeclareMathOperator\Coh{Coh}
\DeclareMathOperator\ch{ch}

\DeclareMathOperator\evi{evi}
\DeclareMathOperator\evb{evb}

\DeclareMathOperator\out{out}

\DeclareMathOperator\Wein{Wein}
\DeclareMathOperator\Ob{Ob}
\DeclareMathOperator\Spec{Spec}
\DeclareMathOperator\Gr{Gr}
\DeclareMathOperator\even{even}
\DeclareMathOperator\odd{odd}

\DeclareMathOperator\GGM{GGM}
\DeclareMathOperator\HKR{HKR}
\DeclareMathOperator\td{td}
\DeclareMathOperator\Der{Der}
\DeclareMathOperator\ord{ord}
\DeclareMathOperator\vG{vG}

\DeclareMathOperator\imm{im}

\newcommand{\Limm}{L_{\mathrm{im}}}
\newcommand{\tLimm}{\widetilde{L}_{\mathrm{im}}}
\newcommand{\Lvg}{L_{\mathrm{im}}^5}
\newcommand{\tLvg}{\widetilde{L}_{\mathrm{im}}^5}
\newcommand{\Lsing}{L_{\mathrm{sing}}}
\newcommand{\tLsing}{\widetilde{L}_{\mathrm{sing}}}
\newcommand{\Lvgsing}{L_{\mathrm{sing}}^{5}}
\newcommand{\tLvgsing}{\widetilde{L}_{\mathrm{sing}}^{5}}
\newcommand{\tLvgsmooth}{\widetilde{L}_{\mathrm{sm}}^5}
\newcommand{\pisyz}{\pi_{\mathrm{SYZ}}}
\newcommand{\Wvg}{W^{5}}
\newcommand{\tWvg}{\widetilde{W}^{5}}

\newcommand{\Larc}{L_{\mathrm{arc}}}

\allowdisplaybreaks
\begin{document}

\begin{abstract}
Mirror symmetry gives predictions for the genus zero Gromov--Witten invariants of a closed Calabi--Yau variety in terms of period integrals on a mirror family of Calabi--Yau varieties. We deduce an analogous mirror theorem for the open Gromov--Witten invariants of certain Lagrangian submanifolds of the quintic threefold from homological mirror symmetry, assuming the existence of a negative cyclic open-closed map. These Lagrangians can be thought of as SYZ mirrors to lines, and their open Gromov--Witten (OGW) invariants coincide with relative period integrals on the mirror quintic calculated by Walcher. Their OGW invariants are irrational numbers contained in $\mathbb{Q}(\sqrt{-3})$, and admit an expression similar to the Ooguri--Vafa multiple cover formula involving special values of a Dirichlet $L$-function. We achieve these results by studying the Floer theory of a different immersed Lagrangian in the quintic that supports a one-dimensional family of objects in the Fukaya category mirror to coherent sheaves supported on lines in the mirror quintic. The field in which the OGW invariants lie arises as the invariant trace field of (the smooth locus of) a closely related hyperbolic Lagrangian submanifold with conical singularities in the quintic. These results explain some of the predictions on the existence of hyperbolic Lagrangian submanifolds in the quintic put forward by Jockers--Morrison--Walcher.
\end{abstract}

\maketitle
\tableofcontents
\section{Introduction}
The first mathematical manifestation of mirror symmetry was a prediction for the genus zero Gromov--Witten invariants of the quintic threefold $X$ in terms of period integrals on a mirror family $\mathcal{X}^{\vee}$ of Calabi--Yau threefolds~\cite{COGP91}. Morrison~\cite{Mor93} later observed that these enumerative predictions can be reformulated as an isomorphism between variations of Hodge structures (VHS) associated to $X$ and $\mathcal{X}^{\vee}$. One can construct a VHS from $\mathcal{X}^{\vee}$ using classical Hodge theory, and there is a candidate mirror VHS obtained by equipping the quantum cohomology $QH^*(X)$ with the (small) quantum connection. The genus zero Gromov--Witten invariants of the quintic were calculated by Givental~\cite{Giv95} and Lian--Liu--Yau~\cite{LLY97}, and shown to agree with the values predicted by~\cite{COGP91}.

Kontsevich~\cite{Kon94} proposed homological mirror symmetry as a conceptual framework for explaining enumerative predictions from mirror symmetry. Ganatra--Perutz--Sheridan~\cite{GPS15} showed that homological mirror symmetry for Calabi--Yau varieties implies Hodge-theoretic mirror symmetry in the sense of~\cite{Mor93}, assuming the existence of a \textit{negative cyclic open-closed map} on the Fukaya category. In this setting, Hodge-theoretic mirror symmetry is roughly recovered by comparing the natural (weak proper) Calabi--Yau structure on the Fukaya category with the Calabi--Yau structure on the derived category of the mirror.

It has long been expected~\cite{Wit95, AV, Wal07, MW} that the (genus zero) \textit{open Gromov--Witten} invariants, which count pseudoholomorphic disks with Lagrangian boundary, can also be recovered from mirror symmetry. Using classical mirror principles, Walcher~\cite{Wal07} gave a prediction for the open Gromov--Witten invariants of the `real quintic' (the set of real points in a quintic threefold defined over $\mathbb{R}$), which was later verified by Pandharipande--Solomon--Walcher~\cite{PSW08} using equivariant localization and the real structure. As in the closed case, the predictions for open Gromov--Witten invariants can be rephrased as the prediction that there is an isomorphism between two extensions of variations of mixed Hodge structure (VMHS)~\cite{HW22}. The extension of VMHSs in the $B$-model arises from a classical construction from a smooth family of homologically trivial algebraic cycles (see e.g.~\cite{Voi}), while the candidate mirror $A$-model extension of VMHSs is governed by the open Gromov--Witten invariants of a nullhomologous Lagrangian brane. This extension should essentially come from the relative quantum cohomology constructed by Solomon--Tukachinsky~\cite{ST23}.

The purpose of this paper is to prove a mirror theorem, in a specific example, involving the open Gromov--Witten invariants of an immersed Lagrangian brane in a closed Calabi--Yau threefold, and to illustrate how such results can be recovered from homological symmetry. Our results require that the Fukaya category satisfies some widely expected structural properties related to the existence of a cyclic open-closed map. These assumptions are needed in the proof of homological mirror symmetry~\cite{She15, SS21, GHHPS24} and to extract closed Gromov--Witten invariants from the Fukaya category~\cite{GPS15}.

On the $A$-side, we construct an immersed Lagrangian that supports an infinite family of objects in the Fukaya category mirror to (non-isolated) lines in the mirror quintic. Our main result recovers the open Gromov--Witten potential of one of these objects. The open Gromov--Witten potential is thought of as a formal power series in the Novikov variable $Q$, and the coefficients of this power series are the open Gromov--Witten invariants. In this paper, we denote by $\omega$ a primitive third root of $1$, which we take to be $\omega = e^{\frac{2\pi i}{3}}$ for definiteness.
\begin{theorem}\label{mainprelim}
Suppose that Assumptions~\ref{generationass},~\ref{traceass}, and~\ref{connectionsass} are all satisfied by the quintic threefold. Then there is an embedded Lagrangian submanifold $\tLvgsmooth$ and two rank $1$ local systems over the Novikov ring, denoted $\nabla^{\vG}_{\omega^2}$ and $\nabla^{\vG}_{\omega}$, which give rise to unobstructed Lagrangian branes $(\tLvgsmooth,\nabla^{\omega^i}_{\Lambda})$ for $i = 1,2$ when equipped with suitable bounding cochains. The open Gromov--Witten potentials, in the sense of Fukaya~\cite{Fuk11}, of these two branes satisfy
\begin{align}
\Psi(\tLvgsmooth,\nabla_{\Lambda}^{\omega^2})-\Psi(\tLvgsmooth,\nabla_{\Lambda}^{\omega}) = \sum_{d=1}^{\infty}n_d\sum_{k=1}^{\infty}\frac{\chi(k)}{k^2} Q^{kd} \label{bpspotential}
\end{align}
where $\chi\colon\mathbb{Z}\to\lbrace -1,0,1\rbrace$ is the nontrivial Dirichlet character of order $3$ given by
\begin{align*}
\chi(k)\coloneqq\frac{1}{\sqrt{-3}}(\omega^k-\omega^{2k}) = \begin{cases}
0 & k\equiv 0 \mod 3 \\
1 & k\equiv 1 \mod 3 \\
2 & k\equiv -1 \mod 3
\end{cases}
\end{align*}
and $n_d\in\sqrt{-3}\mathbb{Z}[\frac{1}{3}]$. The coefficients $n_d$ in~\eqref{bpspotential} are two times the corresponding numbers in~\cite[\S{6.1}]{Wal12}. In particular, the open Gromov--Witten invariants, which are the coefficients $\widetilde{n}_d$ in the expansion
\begin{align*}
\Psi(\tLvgsmooth,\nabla_{\Lambda}^{\omega^2})-\Psi(\tLvgsmooth,\nabla_{\Lambda}^{\omega})= \sum_{d=1}^{\infty}\widetilde{n}_d Q^d
\end{align*}
of~\eqref{bpspotential}, lie in the field $\mathbb{Q}(\sqrt{-3})$.
\end{theorem}
The Lagrangian submanifold $\tLvgsmooth$ can intuitively be thought of as the lift of a tropical curve, similar the examples in mirror quintic threefolds constructed in~\cite{MR20}. The irrationality of the open Gromov--Witten invariants is determined by the holonomy representations of the local systems $\nabla_{\Lambda}^{\omega^i}$, which are defined over the Novikov ring with coefficients in $\mathbb{Q}(\sqrt{-3})$. The first few values of $\widetilde{n}_d$ are
\begin{align*}
\widetilde{n}_1 &= 280000\sqrt{-3} \\
\widetilde{n}_2 &=\frac{22296200000}{3}\sqrt{-3} \\
\widetilde{n}_3 &= \frac{10031895589000000}{27}\sqrt{-3} \\
\widetilde{n}_4 &= \frac{660275805871745000000}{27}\sqrt{-3} \,.
\end{align*}
Notice that
\begin{align*}
\sum_{k=1}^{\infty}\frac{\chi(k)}{k^2}
\end{align*}
can be written as $L(2;\chi)$, where $L(s;\chi)$ is the Dirichlet $L$-function associated to $\chi$. This is also a special value of the D-logarithm introduced by Walcher~\cite{Wal12}. One will notice a similarity between this formula and the Ooguri--Vafa multiple cover formula (see~\cite{OV00, AV, Liu}). As explained more thoroughly in Remark~\ref{hyperbolicconj}, the result of Theorem~\ref{mainprelim}, and the auxiliary results appearing in its proof, can be regarded as confirming a prediction on the existence of \textit{hyperbolic Lagrangians} in the quintic threefold due to Jockers--Morrison--Walcher based partially on the results of~\cite{Wal12}.

The open Gromov--Witten invariants of Theorem~\ref{mainprelim} are obtained from homological mirror symmetry for the quintic threefold. This is achieved by studying the Floer theory of a \textit{different} Lagrangian immersion $\tLvg$ (cf. Theorem~\ref{main1}) in a way that we summarize in the rest of this introduction. We are also able to compute the open Gromov--Witten invariants of $\tLvg$ using this framework (cf. Corollary~\ref{oldmain} below). To the author's knowledge, these are the first examples of an enumerative invariant to be recovered from homological mirror symmetry before being calculated using classical techniques (e.g. Atiyah--Bott localization).

\subsection{Background on lines in the mirror quintic}
Let $\Delta$ denote the unit disk in $\mathbb{C}$ and let $\Delta^*\coloneqq\Delta\setminus\lbrace 0\rbrace$ denote the punctured unit disk. To define the mirror quintic family, we first consider the so-called Dwork family quintics with fibers given by
\begin{align}\label{dwork}
X_z^5\coloneqq\left\lbrace\prod_{j=1}^5 x_j - \frac{z^{1/5}}{5}\sum_{j=1}^5 x_j^5 = 0\right\rbrace\subset\mathbb{CP}^4 \, ; \; z\in\Delta^* \, .
\end{align}
There is an action of $(\mathbb{Z}/5)^3$ on each $X_z^5$, which is inherited from the natural action of $(\mathbb{Z}/5)^5$ on $\mathbb{CP}^4$ by restricting to the subgroup that preserves the monomial $\prod_{j=1}^5 x_j$, and taking the quotient by the diagonal subgroup. The mirror quintic family $\pi\colon\mathcal{X}^{5,\vee}\to\Delta^*$ has fibers $\pi^{-1}(z)\coloneqq X_z^{5,\vee}$ which are crepant resolutions of the orbifold quotients $X_z^5/(\mathbb{Z}/5)^3$.

As was first observed by van Geemen, cf.~\cite{AK91}, the Dwork quintics $X_z$ all contain infinite families of lines. This can be seen by considering the special lines of the form
\begin{align}\label{vangeemeneqn}
\widetilde{C}_z^{\omega}\coloneqq\left\lbrace x_1+\omega x_2+\omega^2 x_3 = 0 \, , \, x_4 = \frac{a}{3}(x_1+x_2+x_3) \, , \, x_5 = \frac{b}{3}(x_1+x_2+x_3) \right\rbrace\subset X_{z}^5
\end{align}
where $a,b,\omega\in\mathbb{C}$ are constants subject to the relations
\begin{align}\label{vangeemeneqn2}
1+\omega+\omega^2 &= 0 \\
a^5+b^5 &= 27 \nonumber \\
ab &= 6z^{1/5} \, . \nonumber
\end{align}
Each of the lines $\widetilde{C}_z^{\omega}$ is called a \textit{van Geemen line}. Fix one such line, and note that the orbit of $\widetilde{C}_z^{\omega}$ under the group $(\mathbb{Z}/5)^3$ contains $125$ elements. Moreover, the orbit of $\widetilde{C}_z^{\omega}$ under the action of the symmetric group $S_5$, which permutes the coordinates on $\mathbb{CP}^4$, has order $40$, since the subgroup given by permuting the first three coordinates preserves such a line. These facts imply that there are at least $5000$ lines on $X_z$, and since this exceeds the virtual number of lines, $2875$, there must actually be infinitely many. 

The lines on a generic Dwork quintic $X_z^5$ were described explicitly in~\cite{COVV12}. In addition to $375$ isolated lines, there are two families of lines, each of which can be identified with a $125$-fold cover of a genus six curve, meaning that they are both genus $626$ curves. It follows from the results of~\cite{COVV12} that the families of lines in the mirror quintic can be identified with two genus six curves. Denote by $C_z^{\omega}$ the line in the mirror quintic $X_z^{\vee}$ corresponding to any of the lines $\widetilde{C}_z^{\omega}$, for some $a,b\in\mathbb{C}$ as in~\eqref{vangeemeneqn2}. Notice that all such choices of $a,b\in\mathbb{C}$ yield the same line. Abusing terminology, we also call the lines $C_z^{\omega}$ in the mirror quintic van Geemen lines.
\begin{convention}
Note that the mirror quintic family can be viewed as a scheme over $\mathbb{C}((z))$, where $z$ here denotes a formal variable. We denote this scheme by $X^{5,\vee}$. Similarly, the van Geemen lines determine lines in $X^{5,\vee}$, which we denote by $C^{\omega}$ and $C^{\omega^2}$.
\end{convention}

\subsection{Tropical Lagrangian submanifolds in the quintic}
The Lagrangian $\tLvgsmooth$ of Theorem~\ref{mainprelim} is constructed as the lift of a tropical curve, which comes from considering the tropicalization of a van Geemen line. To explain the meaning of this statement, consider the cycles $\widetilde{C}_0^{\omega}$ in the toric boundary $X_0^5\coloneqq\left\lbrace\prod_{j=1}^5 x_j = 0\right\rbrace$ of $\mathbb{CP}^4$ obtained from~\eqref{vangeemeneqn} by setting $z=0$. It follows from~\eqref{vangeemeneqn2} that either $a = 0$ or $b = 0$, meaning that $\widetilde{C}_z^{\omega}$ either lies in one of the hyperplanes $\lbrace x_4 = 0\rbrace$ or $\lbrace x_5 = 0\rbrace$. For concreteness, we will consider the case where $b=0$, so we can think of $\widetilde{C}_z^{\omega}$, for some fixed $a\in\mathbb{C}$, as a cycle in $\mathbb{CP}^3\cong\lbrace x_5 = 0\rbrace$. An easy computation shows that $\widetilde{C}_z^{\omega}$ intersects the toric boundary of $\mathbb{CP}^3$ in exactly four points, meaning that its restriction to the big torus $(\mathbb{C}^*)^3\subset\mathbb{CP}^3\cong\lbrace x_5 = 0\rbrace$ is a rational curve with four punctures.

The tropicalization of a curve $C\subset(\mathbb{C}^*)^n$ can be thought of as the image $\Log_t(C)$, in the limit $t\to\infty$, under the map $\Log_t\colon(\mathbb{C}^*)^n\to\mathbb{R}^n$ given by
\begin{align}\label{syzfibr}
\Log_t(x_1,\ldots,x_n) = (\log_t|x_1|,\ldots,\log_t|x_n|) \, .
\end{align}
The tropicalization will be a $1$-dimensional polyhedral complex whose $1$-dimensional faces are affine line segments in $\mathbb{R}^n$ and whose $0$-dimensional faces satisfy a  balancing condition that dictates how the $1$-faces meet at the $0$-faces. A direct calculation shows that the tropicalization of $\widetilde{C}_z^{\omega}\cap(\mathbb{C}^*)^3$ is a \textit{$4$-valent} tropical curve with a single vertex. The tropicalization is independent of the choice of $a$, since all such choices have the same norm.

There is a general theory for constructing Lagrangian lifts of \textit{smooth} tropical subvarieties of $\mathbb{R}^3$ that should correspond, under homological mirror symmetry, to (structure sheaves on) the corresponding complex subvarieties of $(\mathbb{C}^*)^3$; see inter alia~\cite{Mik19, Hic20, Hic21, Mat20}. As shown in~\cite[Theorem 1.1]{Han24a}, the natural analogue of these constructions applied to a $4$-valent tropical curve arising from the very affine van Geemen lines will only produce a singular Lagrangian submanifold $\Lsing\subset T^*T^3$ in the cotangent bundle of a $3$-torus. It is shown in loc. cit. that $\Lsing$ has one conical singular point modeled on the \textit{Harvey--Lawson cone} of~\cite{HL}, and that the smooth part of $\Lsing$ is diffeomorphic to the complement in $S^3$ of the \textit{minimally-twisted five-component chain link} drawn in Figure~\ref{chainlink}.

The starting point for our construction is to identify a particular Lagrangian torus $L_{\tau,0}$ in a Dwork quintic $X_{\tau}^5$, which supports a family of objects mirror to points on a quintic. By this we mean that the objects of the Fukaya category $\mathcal{F}(X_{\tau})$ obtained by equipping $L_{\tau,0}$ with a rank $1$ local system are all mirror to skyscraper sheaves on the mirror quintic. The Lagrangian torus in question is obtained from a moment fiber in toric boundary of $\mathbb{CP}^4$ under symplectic parallel transport. One would expect the Lagrangian tori constructed this way to be smooth SYZ fibers of an SYZ fibration on the quintic (cf.~\cite{Rua01, Sol20}).

We then consider a $125$-fold cover of $\Lsing$ contained inside a Weinstein neighborhood of this torus. Intuitively, taking this cover corresponds, under mirror symmetry, to taking a quotient by the action of $(\mathbb{Z}/5)^3$ in the construction of the mirror quintic. We obtain a compact singular Lagrangian $\tLvgsing$ by attaching suitable copies of the Harvey--Lawson cone liing near the intersection of the toric boundary with $X_{\tau}^5$ to this (non-compact) singular Lagrangian. The approach we use to perform this gluing construction is similar to, but much easier than, the construction of tropical Lagrangians in mirror quintics in~\cite{MR20}.

The Lagrangian submanifold $\tLvgsmooth$ of the quintic is obtained from $\tLvgsing$ using the (a topological version of) the construction of~\cite{Joy04}. This roughly replaces the a neighborhood of each cone point of such a singular Lagrangian with a copy of an Aganagic--Vafa brane~\cite{AV}. Thought of in these terms, it is unclear how to calculate the open Gromov--Witten invariants of $\tLvgsmooth$ directly: it has very few apparent symmetries given its construction by soft topological methods, and there are currently few available tools for counting disks with Lagrangian boundary tropically in this context. This might make one hopeful that the open Gromov--Witten invariants can be extracted from homological mirror symmetry using the framework of~\cite{Hug24b}, but there is no reason to expect the Floer theory of $\tLvgsmooth$ to tractable. Instead of studying $\tLvgsmooth$ directly as an object of the Fukaya category, we will study the Floer theory of a closely related immersed Lagrangian denoted $\tLvg$.

An immersed Lagrangian lift, denoted $\Limm$, of a curve with the same combinatorial type as the tropicalization of $\widetilde{C}_z^{\omega}\cap(\mathbb{C}^*)^3$ was constructed in~\cite{Han24a} by `doubling' the singular Lagrangian $\Lsing$. Moreover, it was shown~\cite[Theorem 1.2]{Han24a} that equipping $\Limm$ with suitable local systems yields objects of the wrapped Fukaya category of $T^*T^3$ mirror to lines in $\mathbb{CP}^3$. This was achieved by computing the Floer cohomology of such objects with the $0$-section in $T^*T^3$, which amounts to computing the support of the mirror sheaf. This immersed Lagrangian appears as a local piece of the immersed Lagrangian $\tLvg$. Near the singular points of $\tLsing$, we can also think of $\tLvg$ as being obtained by replacing the conical pieces of $\Lsing$ with `doubled' immersed Lagrangian submanifolds~\cite[\S{5.3}]{Han24a} that miss the toric boundary of $\mathbb{CP}^4$. Consequently, we can think of objects of the Fukaya category supported on $\tLvg$ as objects of the \textit{relative Fukaya category}, as is usually considered in proofs homological mirror symmetry for closed Calabi--Yau manifolds~\cite{She15, SS21, GHHPS24}. This fact makes the Floer theory of $\tLvg$ significantly easier to work with than the Floer theory of $\tLvgsmooth$.

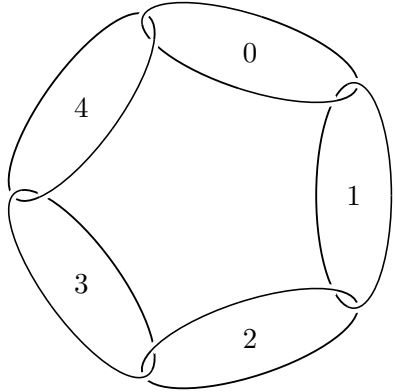
\begin{figure}
\begin{tikzpicture}[scale = 1.0]
\draw(2,0)[line width = 0.25 mm] ellipse (0.5 and 1.5);
\draw[rotate = 72, line width = 0.25 mm] (2,0) ellipse (0.5 and 1.5);
\draw[rotate = 144, line width = 0.25 mm] (2,0) ellipse (0.5 and 1.5);
\draw[rotate = 216, line width = 0.25 mm] (2,0) ellipse (0.5 and 1.5);
\draw[rotate = 288, line width = 0.25 mm] (2,0) ellipse (0.5 and 1.5);

\draw (2.5,0)[white, double = black, ultra thick] arc (0:100:0.5 and 1.5);
\draw (2.5,0)[white, double= black, ultra thick] arc (0:-100:0.5 and 1.5);
\draw[white, double = black, ultra thick, rotate = 72] (2.5,0) arc (0:100:0.5 and 1.5);
\draw[white, double = black, ultra thick, rotate = 144] (1.5,0) arc (180:240:0.5 and 1.5);
\draw[white, double = black, ultra thick, rotate = 72] (1.5,0) arc (180:240:0.5 and 1.5);

\draw[white, double = black, ultra thick, rotate = 144] (1.5,0) arc (180:120:0.5 and 1.5);

\draw[white, double = black, ultra thick, rotate = 216] (2.5,0) arc (0:-100:0.5 and 1.5);

\draw[white, double = black, ultra thick, rotate = 216] (2.5,0) arc (0:100:0.5 and 1.5);

\draw[white, double = black, ultra thick, rotate = 288] (1.5,0) arc (180:240:0.5 and 1.5);
\draw[white, double = black, ultra thick, rotate = 288] (1.5,0) arc (180:120:0.5 and 1.5);

\node[] at (2,0) {1};
\node[] at (0.62,1.90) {0};
\node[] at (-1.62,1.17) {4};
\node[] at (-1.62,-1.17) {3};
\node[] at (0.62,-1.90) {2};

\end{tikzpicture}
\caption{The minimally-twisted five-component chain link complement, and a labeling of its components.}\label{chainlink}
\end{figure}
\begin{remark}\label{hyperbolicconj}
The smooth part of $\tLvgsing$, denoted $\widetilde{L}'$ is a cover of the minimally-twisted five-component chain link complement $L'$, and hence it is a hyperbolic $3$-manifold. The hyperbolicity of $L'$ determines a faithful representation $\pi_1(L')\to SL(2,\mathbb{C})$. Given any subgroup $\Gamma$ of $SL(2,\mathbb{C})$, one can construct an extension of $\mathbb{Q}$ called the \textit{invariant trace field} of $\Gamma$ by adjoining the traces of certain element of $\Gamma$ to $\mathbb{Q}$ (see~\cite[Chapter 3]{MR03} for a precise definition). This turns out to depend only on the commensurability class of $\Gamma$. From this we obtain a field-valued invariant of hyperbolic $3$-manifolds which is preserved under taking finite degree covers.

The invariant trace field of $L'$ is $\mathbb{Q}(\sqrt{-3})$, which can be seen from an ideal triangulation~\cite[\S{2.6}]{DT03}, and hence it is also the invariant trace field of $\widetilde{L}'$. This is also the field in which the open Gromov--Witten invariants appearing in Theorem~\ref{mainprelim} lie. Such a relation between invariant trace fields and open Gromov--Witten invariants has been long-expected in string theory based on computations of Jockers--Morrison--Walcher. In this way, one can thin of the field $\mathbb{Q}(\sqrt{-3})$ as a \textit{period ring} for the branes $(\tLvgsmooth,\nabla^{\Lambda}_{\omega^i}$, which corresponds to the role that this field plays on the $B$-side. By~\cite[Corollary 1.2]{Han24b}, there is a general relation on the $A$-side between the field of definition of brane structures on a Lagrangian submanifold and the field of definition of its open Gromov--Witten invariants.
\end{remark}

\begin{remark}
It is remains unclear whether there should also be a general relationship between topological invariants of a Lagrangian submanifold, such as the hyperbolic volume or invariant trace field, and the values of its open Gromov--Witten invariants. In particular, we note that $\tLvgsmooth$ is not a hyperbolic $3$-manifold. The results of \S\ref{embeddedtrop} show that $\tLvgsmooth$ is obtained by Dehn filling (a cover of) the Lagrangian lift of a tropical curve given in~\eqref{tropsmoothing}. It follows that $\tLvgsmooth$ contains an essential torus, which can be thought of in this context as a conormal fiber over the finite edge (cf. Figure~\ref{smthfig}) of this tropical curve. We also remark that one of the original motivations for the predictions of Jockers--Morrison--Walcher comes from their announced computation of a constant first approximated by Laporte--Walcher~\cite{LW}, which they conjectured to be the volume of a hyperbolic $3$-manifold mirror to the van Geemen lines. The hyperbolic volume of $\widetilde{L}'$ is $1250\operatorname{Im}\operatorname{Li}_2(-\omega)$, where $\operatorname{Im}\operatorname{Li}_2(-\omega)$ denotes the imaginary part of the dilogarithm. We compute this volume using the fact that the hyperbolic volume of the minimally-twisted five=component chain link complement is $10\operatorname{Im}\operatorname{Li}_2(-\omega)$, since it can be written explicitly as the union of $10$ regular ideal tetrahedra. This disagrees with the predicted volume $130\operatorname{Im}\operatorname{Li}_2(-\omega)$, which we obtain by rewriting~\cite[(17)]{LW} in terms of the dilogarithm. We also point out that one should not expect to obtain an \textit{embedded} hyperbolic Lagrangian submanifold from $\tLvgsing$, per the discussion in~\cite[\S{1}]{Han24a} on Dehn fillings of the chain link complement.
\end{remark}

Given its close relation to the Lagrangian torus in $X_{\tau}^5$ supporting mirrors to points, we can describe mirror sheaves, with respect to the mirror equivalence of~\cite{She15} to objects supported on $L_{\mathrm{im}}^5$. We use a Morse--Bott model for the Fukaya category following~\cite{Fuk17}, but with some modifications that allow for immersed Lagrangians equipped with local systems and gradings~\cite{Sei00}. A description of the Fukaya category we use is given in Appendix~\ref{immersedfloerappendix}. The results of~\cite{She15} still apply provided our model for the Fukaya category satisfies certain technical assumptions, namely Assumption~\ref{generationass}, which are expected to hold for any reasonable version of the Fukaya category. Most crucially, these assumptions say that the Fukaya category should have all of the algebraic structure needed to mimic the proof of Abouzaid's generation criterion~\cite{Abo10}.
\begin{theorem}[Paraphrasing of Theorem~\ref{mirrorsheaf}]\label{main1}
Under Assumption~\ref{generationass}, consider a generic non-isolated line $C$ in the mirror quintic. Then there is an unobstructed rank one $\mathbb{C}$-local system $\nabla_C$ on $\tLvg$ such that the resulting object $(\tLvg,\nabla_C)$ of the Fukaya category is mirror, under the equivalence of~\cite{She15}, to a rank $2$ vector bundle supported on $C$.
\end{theorem}
We can, in particular, take $C$ to be a van Geemen line $C^{\omega}$, and the corresponding local system is denoted $\nabla^{\vG}_{\omega}$ as above. Here the mirror quintic is understood as a variety defined over the Novikov field. It is clear from their definition that the van Geemen lines can be naturally identified with lines in this variety. To prove Theorem~\ref{main1}, we will calculate its Floer cohomology with the torus $L_{\tau,0}$. This computation uses~\cite[Theorem 1.2]{Han24a} in an essential way. By computing $HF^*(\tLvg,\nabla_C)$, we will show that this object is mirror to the pushforward of a vector bundle. This determines the \textit{algebraic second Chern class} of the mirror sheaf, which is crucial for deducing Theorem~\ref{mainprelim} from Theorem~\ref{main1}. 

To obtain Theorem~\ref{mainprelim} from Theorem~\ref{main1}, we will relate the $A_{\infty}$-algebra of one of the Lagrangian branes $(\tLvg,\nabla)$ to the $A_{\infty}$-algebra of $\tLvgsmooth$ equipped with a suitable brane structure. This is carried out in \S\ref{lsmthchern}, and involves studying the behavior of the Fukaya $A_{\infty}$-algebra under clean Lagrangian surgery in a particularly simple case.

We will also be able to compute the open Gromov--Witten invariants of this immersed Lagrangian from Theorem~\ref{main1} given the result of Theorem~\ref{main2} below.
\begin{corollary}\label{oldmain}
Let $\nabla^{\vG}_{\omega^i}$, for $i = 1,2$, denote the local systems corresponding to the van Geemen lines as given in Example~\ref{vangeemenlocalsystem}. Then the difference of open Gromov--Witten potentials 
\begin{align*}
\Psi(\tLvg,\nabla^{\vG}_{\omega^2})
-\Psi(\tLvg,\nabla^{\vG}_{\omega})
\end{align*}
satisfy formulas of the same form as in Theorem~\ref{mainprelim}, except that the constants $n_d$ and $\widetilde{n}_d$ are four times those appearing in~\cite[\S{6.1}]{Wal12}.
\end{corollary}
\subsection{Open Gromov--Witten invariants and homological mirror symmetry} The statement of Theorem~\ref{mainprelim} involves the difference of open Gromov--Witten potentials for two Lagrangian branes in the quintic supported on the same embedded Lagrangian submanifold $\tLvgsmooth$. Since this Lagrangian is not nullhomologous, the open Gromov--Witten potential defined for graded Lagrangian submanifolds of Calabi--Yau threefolds by Fukaya~\cite{Fuk11} is not independent of the almost complex structure used to define it. We can, however, show that the difference of open Gromov--Witten potentials is independent of the almost complex structure by viewing it as the open Gromov--Witten potential of an \textit{immersed Lagrangian submanifold} given by two copies of $\tLvgsmooth$.

This requires a definition of the open Gromov--Witten potential can be defined for Lagrangian immersions. The construction of the open Gromov--Witten in~\cite{Fuk11} does not immediately generalize to the immersed case, as the Fukaya $A_{\infty}$-algebra of a clean Lagrangian immersion does not possess cyclic symmetry~\cite{Fuk10}. We will avoid this issue by following the strategy of~\cite{Han24b}, where it was shown that the open Gromov--Witten invariants can be defined using the Calabi--Yau structure obtained from a \textit{cyclic open closed map}. We recall that this is a map from the cyclic homology of the Fukaya category to the quantum cohomology which intertwines the  relevant $S^1$-actions. Incidentally, the (negative) cyclic open-closed map is also needed to extract genus zero Gromov--Witten invariants from the Fukaya category, as shown by Ganatra--Perutz--Sheridan~\cite{GPS15}. Hence, we will assume (cf. Assumptions~\ref{connectionsass} and~\ref{traceass}) that suitable versions of the cyclic open-closed map exist.

Using results of Cho and Lee~\cite{CL}, we can construct a homotopy cyclic $\infty$-inner product $\psi$ on $\mathcal{A}$ from the trace of Assumption~\ref{traceass}. More specifically, this is an $A_{\infty}$-bimodule homomorphism $\psi\colon\mathcal{A}_{\Delta}\to\mathcal{A}^{\vee}$ from the diagonal bimodule to the linear dual satisfying some additional symmetries. By closely following the arguments in~\cite{Han24b}, we can use this to construct the open Gromov--Witten potential of a clean Lagrangian immersion.
\begin{theorem}\label{immersedogw}
Let $\mathcal{A}$ denote the curved Fukaya $A_{\infty}$-algebra of a graded clean Lagrangian immersion $L$ in a Calabi--Yau threefold equipped with a  rank $1$ local system, and suppose that it satisfies Assumption~\ref{traceass}. Then there is a well-defined open Gromov--Witten potential $\Psi\colon\mathcal{M}(L)\to\Lambda$, where $\mathcal{M}(L)$ denotes the space of gauge-equivalence classes of bounding cochains on $L$, given by
\begin{align*}
\Psi(b)\coloneqq\mathfrak{m}_{-1}+\sum_{N=0}^{\infty}\sum_{p+q+k = N}\frac{1}{N+1}\psi_{p,q}(b^{\otimes p}\otimes\underline{\mathfrak{m}_k(b^{\otimes k})}\otimes b^{\otimes q})(b) \, 
\end{align*}
where $\mathfrak{m}_{-1}$ is a count of pseudoholomorphic disks with boundary on $L$ and no boundary marked points. This is invariant of the almost complex structure on $M$ up to a wall-crossing term given by counting closed pseudoholomorphic curves intersecting $L$ in a point.
\end{theorem}
\begin{remark}
Assuming that $L$ is graded and $3$-dimensional allows us to work over the Novikov ring, rather than an extension thereof as in~\cite{ST21, ST23}. There should be no essential difficulty in defining the open Gromov--Witten potential of a general Lagrangian immersion using the techniques of this paper by considering a larger coefficient ring. It would be interesting to formulate the notion of a \textit{point-like} bounding cochain in this situation. Such a definition might make it possible to obtain meaningful counts of disks with boundary and corners on a Lagrangian immersion.
\end{remark}

One can compensate for the wall-crossing term mentioned in Theorem~\ref{immersedogw} when $L$ is nullhomologous using a choice of bounding $4$-chain, obtaining a potential which is truly invariant under changes of almost complex structure. In this setting, Solomon--Tukachinsky~\cite{ST23} construct a connection on the relative cohomology $QH^*(X,L)$, essentially constructing an extension of VHSs, but without specifying an integral structure. The connection on $QH^*(X,L)$ is defined using the open Gromov--Witten invariants, and so one would expect open enumerative mirror symmetry to be implied by an isomorphism between $QH^*(X,L)$ and an extension of VHSs defined in the $B$-model, as described in~\cite{Car87}.

Let $X$ be a Calabi--Yau $3$-fold, and let $X^{\vee}$ be a mirror family of Calabi--Yau threefolds, thought of here as a variety over a Novikov field. We further assume that $X^{\vee}$ is \textit{maximally unipotent}, in the sense of~\cite[\S{1.1}]{GPS15}, a suitable analogue over the Novikov field of the notion of a large complex structure limit point~\cite{CK99}. In~\cite{GPS15} Ganatra--Perutz--Sheridan show how to recover the enumerative predictions of~\cite{COGP91} from homological mirror symmetry using variations of \textit{semi-infinite} Hodge structures (VSHS) defined at the categorical level. A VSHS can be thought of as a variation of Hodge structure without a choice of integral local system. If $X^{\vee}$ is smooth and compact, one can endow its negative cyclic homology $HC_{\bullet}^{-}(D^b\Coh(X^{\vee}))$ with the structure of a VSHS. Assuming the existence of a weak proper Calabi--Yau structure on the Fukaya category (cf. Assumption~\ref{traceass}), and using the fact that $X^{\vee}$ is maximally unipotent, one can also construct a VSHS associated to $HC_{\bullet}^{-}(\mathcal{F}(X))$ on the $A$-side. Homological mirror symmetry determines an isomorhpism between these two VSHSs. Ganatra--Perutz--Sheridan use the negative cyclic open-closed map (cf. Assumption~\ref{connectionsass}) and an appropriate version of the HKR isomorphism (cf.~\cite{Tu24}) to show that these VSHSs are isomorphic to the ones relevant to homological mirror symmetry, which suffices to deduce the genus zero Gromov--Witten invariants of $X$ from homological mirror symmetry.

To prove Theorem~\ref{mainprelim}, we will establish a partial analogue of the main Theorem of~\cite{GPS15} in settings like the one considered in Theorem~\ref{main1} for \textit{extensions} of VSHS. We do so under similar assumptions about the existence and properties of the cyclic open-closed map.
\begin{remark}
Hugtenburg~\cite{Hug24b} proposes a construction of extensions of VSHSs at the categorical level, as well as a framework for showing that this extension is isomorphic to relative quantum cohomology. This relies on the definition of the cyclic open-closed map in~\cite{Hug24a}, which was defined on the negative cyclic homology of a Fukaya $A_{\infty}$-algebra, as opposed to the Fukaya category, in a way the relies on cyclic symmetry. The same issues that arise when one attempts to construct a cyclically symmetric pairing on the $A_{\infty}$-algebra of an immersed Lagrangian would also arise when one attempts to construct such a pairing on the Fukaya category, which would pose a challenge in attempting to extend the chain-level arguments of~\cite{Hug24b} to the Fukaya category of a Calabi--Yau manifold. Similar difficulties would arise when attempting to adapt the results of Solomon--Tukachinsky~\cite{ST21, ST23} to the present setting.

Instead of comparing extensions of VSHS at the categorical level, our approach transfers this difficulty to a comparison of $B$-model VSHSs, which admit a description in terms of normal functions that is particularly helpful in the setting of Theorem~\ref{main1}. This means that our arguments can be carried out at the homological level, albeit in significantly less general settings than one could hope to address with the techniques of~\cite{Hug24b}.
\end{remark}

For the purpose of relating the open Gromov--Witten invariants to \textit{mirror symmetry}, we find it convenient to formulate the extension of VSHSs associated to a Lagrangian submanifold slightly differently than in~\cite{ST23}. The relative period integrals are determined by a normal function, which can be thought an element of an Ext group in the category of VSHSs. This normal function is itself determined by a choice of bounding chain for a nullhomologous algebraic cycle. Therefore, we choose a formulation of the $A$-model extension that privileges the normal function and a bounding smooth singular chain for an immersed Lagrangian submanifold. To carry out this approach, we work under some restrictive assumptions on the Hodge numbers  of $X$, which are satisfied by the quintic.

The following theorem encapsulates the results of \S\ref{comparisonthms}, which show how the open Gromov--Witten invariants are recovered from homological mirror symmetry in sufficiently simple geometric situations. We assume that $X$ and $X^{\vee}$ are a mirror pair of Calabi--Yau $3$-folds, where $X^{\vee}$ is defined over a certain Novikov ring. Alternatively, one can think of $X^{\vee}$ as a family $\mathcal{X}^{\vee}$ of Calabi--Yau threefolds over the punctured unit disk. We assume that $X$ is simply connected and that its even degree Hodge numbers coincide with those of the quintic. These conditions are summarized in Assumptions~\ref{BHodgenumberass} and~\ref{AHodgenumberass}. In this setting, we consider a pair of Lagrangian branes $L_0,L_1\in\mathcal{F}(X)$ with the same open-closed image, and let $\mathcal{L}_0,\mathcal{L}_1\in D_{dg}^b\Coh(X^{\vee})$ denote the mirror objects. These should satisfy Assumption~\ref{functass}, meaning that $\mathcal{L}_i$ is required be the pushforward of a vector bundle on an algebraic curve $C^i$, and that $C^0$ and $C^1$ are homologous. Theorem~\ref{main1} verifies this assumption for the pair of branes $(\tLvg,\nabla^{\vG}_{\omega^2})$ and $(\tLvg,\nabla^{\vG}_{\omega})$. Since a nontrivial algebraic curve is never nullhomologous, it is natural to phrase our results in terms of a pair of Lagrangian branes on the $A$-side, rather than a single Lagrangian brane.

\begin{theorem}\label{main2}
Suppose that $X$ and $X^{\vee}$ are a mirror pair of Calabi--Yau $3$-folds subject to Assumptions~\ref{BHodgenumberass} and~\ref{AHodgenumberass}, and that $L_0$ and $L_1$ are a pair of immersed Lagrangian branes subject to Assumption~\ref{functass}. Then there is a normal function in the quantum cohomology of $X$ determined by the open Gromov--Witten invariants of the immersion $L_0\cup L_1\subset X$. This normal function coincides, up to changing coordinates by the mirror map, with a normal function determined by $mC$, where $C$ is the algebraic cycle $C_0-C_1$.
\end{theorem}
\begin{remark}[Open enumerative mirror symmetry for the real quintic]
Consider the set of real points in a Dwork quintic threefold $X_{\tau}^5$ for a small real constant $\tau$, i.e. the \textit{real quintic}. This is an embedded Lagrangian submanifold diffeomorphic to $\mathbb{RP}^3$, which supports two objects of the Fukaya category corresponding to its two unitary local systems. Provided that one can show that these objects supported on the real quintic are mirror to the expected $B$-model objects, which can be thought of as structure sheaves of conics in the mirror quintic~\cite{MW}, it should follow from Theorem~\ref{main2} that the main theorem of~\cite{PSW08} can be recovered from homological mirror symmetry.
\end{remark}
This result does not establish a relationship between period integrals and and open Gromov--Witten invariants in the maximum conceivable generality, but it does apply in the setting of Theorem~\ref{main1}, and it is also expected to apply to the real quintic and its mirror conics~\cite{Wal07, PSW08, MW}. The strategy of the proof is to define a candidate normal function in quantum cohomology by hand, using the simple form of the quantum connection on the quintic to check that we can indeed define a Hodge structure this way. For this purpose, we use an expression for the quantum connection in terms of a basis on quantum cohomology specified by Schwarz--Vologodsky--Walcher~\cite{SVW17}. The determination of this basis is where a \textit{splitting} of the Hodge filtration (cf.~\cite{Mor93} or~\cite[Def. 2.9 and \S{2.2}]{GPS15}) is used in our argument. Our construction of the open Gromov--Witten potential, combined with assumptions about the cyclic open-closed map, shows that this normal function lifts to a normal function in the negative cyclic homology of the Fukaya category. Under homological mirror symmetry, this corresponds to a normal function in the derived category of the mirror, and hence yields a normal function for the $B$-model VSHS associated to $X^{\vee}$.

We can compare this extension of VSHSs to the one derived from the cycle $mC$ by appealing to~\cite{Cal04}, which relates the Chern character to the HKR isomorphism used in~\cite{GPS15} to compare closed $B$-model Hodge structures. This, combined with the calculations of~\cite{Wal12} then implies Theorem~\ref{mainprelim}, since the extensions of $B$-model Hodge structures are both essentially determined by the algebraic second Chern class of an object in the derived category. For the pushforward of a vector bundle on a curve, the algebraic second Chern class is just an integer multiple of the support. The need to determine the second Chern class is the reason that merely computing the support of the mirror sheaf to $(\tLvg,\nabla^{\vG}_{\omega})$ would be insufficient to prove Theorem~\ref{mainprelim}.

\subsection*{Acknowledgments} I thank my advisor, Mohammed Abouzaid, for first suggesting this problem to me, for comments on a draft of this paper, and for his guidance and encouragement. I owe a special debt of gratitude to Nick Sheridan, Jake Solomon, and Johannes Walcher for several important suggestions and questions that decisively influenced this work. I am also grateful to Andrew Hanlon and Nick Sheridan for pointing out an error in an earlier version of this paper. Finally, I thank Denis Auroux, Shaoyun Bai, Soham Chanda, Mark McLean, Paul Seidel, and Mohan Swaminathan for edifying discussions related to various aspects of this paper. This project was partially supported through Abouzaid's NSF grant DMS-2103805 and by the Simons Collaboration on Homological Mirror Symmetry.

\section{Mirrors to points in the mirror quintic}\label{mirrorstopointssect} Since all quintic threefolds are symplectomorphic, we will usually work in a Dwork quintic threefold $X_{\tau}^5$, where $\tau$ is a small real number. In particular, complex conjugation on $\mathbb{CP}^4$ restricts to an anti-symplectic involution on $X_{\tau}^5$, which will allow us to apply the results of~\cite{Sol20} to prove that certain Lagrangian submanifolds are unobstructed.

Studying tropical Lagrangian submanifolds in the quintic using homological mirror symmetry requires some knowledge of how the mirror functor interacts with the putative SYZ fibration. This section draws a partial connection between these two pictures of mirror symmetry, as we show that certain smooth Lagrangian tori in the quintic (which are traditionally thought of as nonsingular SYZ fibers), correspond to skyscraper sheaves under the mirror functor of~\cite{She15}.

For any Dwork quintic $X_{\tau}^5$, recall from~\cite[Example 1.8]{Sol20} and~\cite[\S{3.1}]{Rua01} that there is a natural family of Lagrangian tori in $X_{\tau}^5$, which are constructed as follows. The central fiber $X_0^5$ in the Dwork pencil is the union of coordinate hyperplanes in $\mathbb{CP}^4$. Each of these hyperplanes can be identified with $\mathbb{CP}^3$, so the smooth locus of $X_0^5$ consists of the union of $5$ disjoint copies of $(\mathbb{C}^*)^3$. There is a moment map $\Log\coloneqq\Log_e\colon(\mathbb{C}^*)^3\to\mathbb{R}^3$ which is just the SYZ fibration discussed in the introduction. We obtain Lagrangian tori in $X_{\tau}^5$, for sufficiently small $\tau$, by deforming these moment fibers under symplectic parallel transport.
\begin{definition}
Consider the meromorphic function
\begin{align*}
s\coloneqq\frac{\prod_{j=1}^5 x_j}{\sum_{j=1}^5 x_j^5}
\end{align*}
on $\mathbb{CP}^4$. Let $f$ denote the real part of this function, and consider its gradient $\nabla f$ with respect to the metric on $\mathbb{CP}^4$ induced by the Fubini--Study form and the integrable complex structure. If $L_{0,q}\in(\mathbb{C}^*)^3$ denotes a smooth moment fiber in one of the components of $X_0$ over a point $q\in\mathbb{R}^3$, let $L_{\tau,q}$ denote its image in $X_{\tau}^5$ under the gradient flow of $f$.
\end{definition}
The Lagrangian torus $L_{\tau,q}$ will live near one of the coordinate hyperplanes of $\mathbb{CP}^4$, meaning that there are actually five Lagrangian tori obtained this way. We do not include this in the notation since the choice of coordinate hyperplane will usually be irrelevant or otherwise clear from context.
\begin{remark}
We emphasize that our main results, and in particular Theorem~\ref{main1}, do not assume the existence of an SYZ fibration on the quintic. Instead, we only need to consider certain smooth Lagrangian tori, and in particular we do not need to consider the singular fibers or discriminant locus of an SYZ fibration.
\end{remark}
\begin{remark}\label{graphtorus}
As shown in~\cite[\S{3.2}]{Rua01}, the normalized gradient of $f$ is given by
\begin{align*}
\frac{\nabla f}{|\nabla f|^2} = \mathrm{Re}\left(\frac{x_1^5+x_2^5+x_3^5+1}{x_1x_2x_3}\frac{\partial}{\partial x_5}\right)
\end{align*}
when restricted to the chart $\lbrace x_4 = 1 \, , \; x_5 = 0\rbrace$. In particular, we can think of the Lagrangian tori $L_{\tau,q}$, for small $\tau$, as the graph of a holomorphic function restricted to a moment fiber on $(\mathbb{C}^*)^3$.
\end{remark}
Since these Lagrangian tori in $X_{\tau}^5$ come in a family (as $q$ varies), they will not all be exact in the complement of a divisor, and thus they cannot all be objects of the relative Fukaya category as considered by~\cite{PS24}. For this reason, it is more convenient for us to construct mirrors to points by identifying an appropriate family of $A$-branes supported on a single Lagrangian torus, which we will arrange to be exact.

We define the torus $L_{\tau,0}$ is to be the result of parallel transporting the SYZ fiber $L_{0,0}$ over the origin in a copy of $(\mathbb{C}^*)^3$. The fibration is given in~\eqref{syzfibr}. In one of the hyperplanes $\lbrace x_5 = 0\rbrace\subset X_0^5$, the torus $L_{0,0}$ is given by
\begin{align}\label{l00def}
\lbrace|x_1| = |x_2| = |x_3| = |x_4| = 1\colon x_1x_2x_3x_4 = 1\rbrace\subset\lbrace x_1x_2x_3x_4 = 1\rbrace\cong(\mathbb{C}^*)^3\subset\mathbb{CP}^3
\end{align}
meaning that it is the unique Lagrangian torus in $X_0^5$ preserved by all permutations of the first four coordinates on $\mathbb{CP}^4$. It follows from Remark~\ref{graphtorus}, that $L_{\tau,0}$ possesses the same symmetries.
\begin{definition}\label{toricdivisor}
Let $D_{\tau}\subset X_{\tau}^5$ denote the divisor given by the intersection
\begin{align}
D_{\tau}\coloneqq X_{\tau}^5\cap X_0^5
\end{align}
of $X_{\tau}^5$ with the coordinate hyperplanes in $\mathbb{CP}^4$.
\end{definition}
By construction, $L_{\tau,0}$ lies in the complement of $D_{\tau}$, so we will be able to think of it as an object of the relative Fukaya category~\cite{She15}. Since $X^5_{\tau}\setminus D_{\tau}$ is an affine variety, it carries a Liouville structure induced by pulling back the standard symplectic form on $(\mathbb{C}^*)^4$. We fix a primitive for this symplectic form which makes $L_{\tau,0}\subset X^5_{\tau}\setminus D_{\tau}$ an exact Lagrangian submanifold. The split-generators used in the proof of homological mirror symmetry of~\cite{She15} are represented by immersed Lagrangian spheres, and thus they are exact with respect to any choice of primitive.
\begin{proposition}\label{mirrors2points}
For $\tau\in\mathbb{R}$, the Lagrangian branes obtained by equipping $L_{\tau,0}$ with rank one complex local systems correspond to skyscraper sheaves in the derived category $D^b\Coh(\mathcal{X}^{5,\vee})$ of the mirror quintic with respect to the mirror functor of~\cite{She15}.
\end{proposition}
A previous version of this article claimed an analogue of this result using the mirror functor of~\cite{GHHPS24}, but the proof of homological mirror symmetry as written in op. cit. does not apply to the quintic threefold and its mirror.
\begin{proof}
First note that complex conjugation on $X^5_{\tau}$ acts on $L_{\tau,0}$ and acts on $H^1(L_{\tau,0};\mathbb{C})$ as $-\id$ (cf.~\cite[Example 1.8]{Sol20}). Thus, by~\cite{Sol20}, equipping $L_{\tau,0}$ with any rank one $\mathbb{C}$-local system yields an unobstructed Lagrangian brane (with bounding cochain $0$).

The split-generators for the Fukaya category described in~\cite{She15} are naturally thought of as Lagrangian branes in the Fermat quintic threefold $X^5_{\infty}$. Letting $M^5$ denote the pair of pants
\begin{align*}
M^5\coloneqq\lbrace z_1+z_2+z_3+z_4+z_5 = 0\rbrace\subset\mathbb{CP}^4\setminus\bigcup_{j=1}^5\lbrace z_j = 0\rbrace
\end{align*}
and noting that $X_{\infty}^5\setminus D_{\infty}$ is the affine Fermat quintic
\begin{align*}
X^5_{\infty}\setminus D_{\infty} = \lbrace x_1^5+x_2^5+x_3^5+x_4^5+x_5^5 = 0\rbrace\subset\mathbb{CP}^4\setminus\bigcup_{j=1}^5\lbrace x_j = 0\rbrace
\end{align*}
we recall that there is a covering map
\begin{align*}
X^5_{\infty}\setminus D_{\infty}&\to M^5 \\
x_j &\mapsto x_j^5
\end{align*}
Sheridan's split-generators~\cite{She15} for $\mathcal{F}(X^5_{\infty})$ are obtained by taking lifts of an immersed Lagrangian sphere $S^3\looparrowright M^5$ constructed in~\cite{She11}. The projection of $M^5$ to the moment fiber of $(\mathbb{C}^*)^4$ defines a \textit{coamoeba} whose boundary is a polyhedral $3$-sphere with self-intersections. Sheridan's immersed sphere is, intuitively, characterized as a lift of this sphere to the pair of pants. The result of the proposition will follow from comparing the Lagrangian torus $L_{\tau,0}$ to a boundary facet of the coamoeba.

Since all smooth quintic threefolds are symplectomorphic, we can think of the Lagrangian tori $L_{\tau,0}\subset X^5_{\tau}$ as Lagrangian submanifolds of $X^5_{\infty}$. It will be convenient to specify a symplectomorphism between the very affine quintic threefolds $X_{\tau}^5\setminus D_{\tau}$ and $X_{\infty}^5\setminus D_{\infty}$. Representing the big torus in $(\mathbb{C}^*)^4\subset\mathbb{CP}^4$ by $\lbrace x_1x_2x_3x_4x_5 = \xi\rbrace$, for some generic constant $\xi\in\mathbb{C}^*$ of small norm, we can represent the very affine Dwork quintic as the vanishing locus of a polynomial as follows
\begin{align}\label{veryaffinecoords}
X_{\tau}^5\setminus D_{\tau} = X_{\tau}^5\cap\left\lbrace\prod_{j=1}^5 x_j = \xi\right\rbrace = \left\lbrace \xi-\frac{\tau^{1/5}}{5}\sum_{j=1}^5 x_j^5 = 0\right\rbrace\subset\mathbb{CP}^4\setminus\bigcup_{j=1}^5\lbrace x_j = 0\rbrace \,.
\end{align}
Scaling the constant term of~\eqref{veryaffinecoords} by small positive real constants determines a symplectomorphism between $X_{\tau}^5\setminus D_{\tau}$ and $X_{\infty}^5\setminus D_{\infty}$ which carries $L_{\tau,0}$ to a Lagrangian torus in the affine Fermat quintic which we can take to be given by the set of points $\lbrace |x_1| = |x_2| = |x_3| = \epsilon\rbrace$, for a constant $\epsilon\in\mathbb{C}^*$ of small norm, as can be seen from Remark~\ref{graphtorus}.

The image of $L_{\tau,0}$ under this covering map is a Lagrangian torus in the pair of pants whose argument projection is a boundary facet of the coamoeba, using the description in~\cite[Proposition 2.1]{She11}, and this corresponds to the structure sheaf of a smooth point on the mirror variety $\lbrace z_1z_2z_3z_4z_5 = 0\rbrace\subset\mathbb{C}^5$. Consequently $L_{\tau,0}$, thought of as an object of $\mathcal{F}(X_{\tau}^5)$, is mirror to the structure sheaf of a point by the description of the mirror equivalence in~\cite{She15} as a deformation of a mirror functor discussed in~\cite[Theorem 7.4]{She11}.
\end{proof}
An easy computation in a Weinstein neighborhood shows that if one equips $L_{\tau,0}$ with two different rank $1$ local systems, then the Floer cohomology of the resulting pair of Lagrangian branes (in the quintic) vanishes. This Floer cohomology group corresponds to the Ext group between two skyscraper sheaves under mirror symmetry, implying that no two distinct local systems on $L_{\tau,0}$ yield mirrors to the same point of the mirror quintic.

\section{Tropical Lagrangians in the quintic threefold}\label{tropicallagrsect} In this section, we will construct the immersed Lagrangian of Theorem~\ref{main1}. A brief outline of its construction is as follows.

A Weinstein neighborhood of the Lagrangian torus of Proposition~\ref{mirrors2points} yields an open neighborhood in the Dwork quintic $X_{\tau}^5$ in which we can include copies of tropical Lagrangians. Using this chart, we can identify a copy of (a cover of) the tropical Lagrangian of~\cite[Theorem 1.1]{Han24a} in the quintic. By attaching certain Lagrangian cones to this noncompact Lagrangian in $X_{\tau}^5$, we obtain a Lagrangian denoted $\tLvgsing$ which has isolated conical singularities, all of which are modeled on the Harvey--Lawson cone, see Definition~\ref{hlcone} below. The immersed Lagrangian $\tLvg\to X_{\tau}^5$ lies in a small Weinstein neighborhood of $\tLvgsing$, and is obtained by `doubling' the singular Lagrangian as in~\cite[\S{5.3}]{Han24a}. This immersed Lagrangian can also be thought of as the result of attaching immersed Lagrangian handles to (a cover of) the Lagrangian immersion $\Limm$ constructed in~\cite[Theorem 1.2]{Han24a}. This section draws heavily from~\cite{Han24a}, but we have attempted to keep our exposition in the current paper mostly self-contained.
\begin{remark}
The notation for Lagrangian immersions in this paper differs slightly from the notation of~\cite[Theorem 1.2]{Han24a}. In loc. cit., we denoted by $\tLimm$ the domain of the immersed Lagrangian lift of a $4$-valent tropical vertex. The image of this immersion was denoted $\Limm$. In this paper, we will use the symbol $\tLimm$ to refer to a particular \textit{cover} of $\Limm$ instead.
\end{remark}

\subsection{Singular tropical Lagrangians near an SYZ fiber} We begin by recalling some facts about the singular Lagrangian $\Lsing$ in $T^*T^3$ from~\cite{Han24a}. Let $Q$ be a $3$-dimensional integral affine space, meaning that it comes equipped with a choice of full-rank lattice. This induces a local system $T^*_{\mathbb{Z}}Q$ of integral $1$-forms, from which we can form $T^*Q/T^*_{\mathbb{Z}}Q$. We can naturally identify this space with $T^*T^3$, and under this identification the projection $\pisyz\colon T^*Q/T^*_{\mathbb{Z}}Q\to Q$ corresponds to the projection from $T^*T^3$ to the cotangent fiber, which we also denote by $\pisyz$. Denote by $(q_1,q_2,q_3)$ the coordinates on $Q$, and by $(\theta_1,\theta_2,\theta_3)$ the dual coordinates on the fiber of $T^*Q$, which descend to coordinates in the $T^3$-fibers of $T^*Q/T^*_{\mathbb{Z}}Q$.

It will sometimes be helpful to identify $T^*T^3\cong T^*Q/T^*_{\mathbb{Z}}Q$ with $(\mathbb{C}^*)^3$ using the identification
\begin{align*}
(q_j,\theta_j)\mapsto x_j\coloneqq \exp(q_j+2\pi i\theta_j) \, .
\end{align*}
Under this identification, the map $\pisyz$ corresponds to the map $\Log\coloneqq\Log_e$ defined in the introduction.

\begin{definition}
A tropical curve $W\subset Q$ is an embedded graph with edges $\lbrace W_s\rbrace$ equipped with positive integers weights $\lbrace w_s\in\mathbb{Z}_{>0}\rbrace$ which satisfy the following conditions:
\begin{itemize}
\item[•] each edge $W_s$ is contained in an integral affine subspace of $Q$;
\item[•] for every vertex $p$ of $W$, the edges $W_1,\ldots,W_m$ adjacent to $p$, with corresponding weights $w_1,\ldots,w_m$ and tangent vectors $v_1,\ldots,v_m$ satisfy the balancing condition
\begin{align*}
\sum_{i=1}^m w_iv_i = 0\, .
\end{align*}
\end{itemize}
\end{definition}
The tropical curve $V\subset Q$ we will consider most often in this paper is given by the union of the four $1$-dimensional cones in $Q$, each of which has weight one
\begin{align*}
V_1 &= \left\lbrace(q_1,0,0)\in Q\colon q_1 > 0\right\rbrace \\
V_2 &= \left\lbrace(0,q_2,0) \in Q\colon q_2 > 0\right\rbrace \\
V_3 &= \left\lbrace(0,0,q_3) \in Q\colon q_3 > 0\right\rbrace \\
V_4 &= \left\lbrace(q_1,q_2,q_3) \in Q\colon q_1 = q_2 = q_3 < 0\right\rbrace \, .
\end{align*}
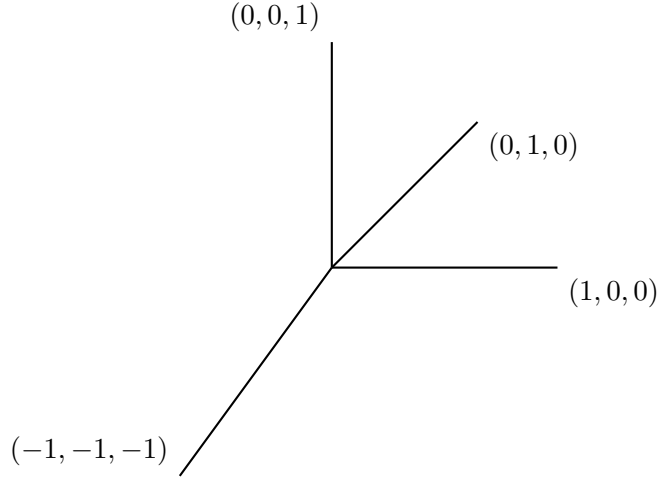
\begin{figure}
\begin{tikzpicture}
\begin{scope}[scale = 0.75]
  \draw[thick] (0,0,0) -- (4,0,0) node[anchor = north west]{$(1,0,0)$};
  \draw[thick] (0,0,0) -- (0,4,0) node[anchor = south east]{$(0,0,1)$};
  \draw[thick] (0,0,0) -- (0,-1,7) node[anchor = south east]{$(-1,-1,-1)$};
  \draw[thick] (0,0,0) -- (4.2,4.2,4.2) node[anchor = north west]{$(0,1,0)$};
\end{scope}
\end{tikzpicture}
\caption{The tropical curve $V$.}\label{tropicalcurve}
\end{figure}
We note that the $4$-valent curve considered in~\cite{Han24a} had edges pointing in the opposite directions, but this change makes no difference to the results proved there.

If $W\subset Q$ is an integral affine subspace, recall that the \textit{periodized conormal} to $W$ in $T^*Q/T^*_{\mathbb{Z}}Q$, denoted $N^*W/N_{\mathbb{Z}}^*W$, is given by taking the quotient of the conormal bundle by the sublattice of integral covectors.
\begin{remark}
In $(\mathbb{C}^*)^3$, the periodized conormals to the legs of $V$ can be written as
\begin{align}
N^*V_i/N^*_{\mathbb{Z}}V_i = 
\left\lbrace |u_j| = |u_k|\text{ and }u_j\in[1,\infty)\right\rbrace \label{conormal123}
\end{align}
for $\lbrace i,j,k\rbrace = \lbrace 1,2,3\rbrace$, and
\begin{align}
N^*V_i/N^*_{\mathbb{Z}}V_i = \left\lbrace |u_1| = |u_2| = |u_3| \text{ and }u_1u_2u_3\in(0,1]\right\rbrace \, . \label{conormal4}
\end{align}
\end{remark}
\begin{definition}
A Lagrangian submanifold $L\subset T^*Q/T^*_{\mathbb{Z}}Q$ is said to be a Lagrangian lift of a tropical curve $W$ if it agrees with the periodized conormals to the edges of $W$ away from subset of the form $\pisyz^{-1}(B)\subset T^*Q/T^*_{\mathbb{Z}}Q$, where $B\subset Q$ is a small open neighborhood of the vertices of $W$.
\end{definition}
We constructed a singular lift of $V$ in~\cite{Han24a}, as we recall below.
\begin{theorem}[{\cite[Theorem 4.1]{Han24a}}]\label{singularsyzlift}
The tropical curve $V$ admits a Lagrangian lift $\Lsing$ with a single singular point modeled on the Harvey--Lawson cone, and the complement $L'$ of this singular point is diffeomorphic to the minimally-twisted five-component chain link complement. In particular, there is a small open ball $B\subset Q$, which we can take to be arbitrarily small, centered at the origin such that $\Lsing\setminus\pisyz^{-1}(B)$ coincides with the periodized conormals to the $1$-dimensional cones of $V$.
\end{theorem}
The $3$-manifold $L'$ has an ideal triangulation constructed in~\cite{DT03}, and is represented schematically in Figure~\ref{idealtriangulation}. This triangulation consists of $10$ hyperbolic ideal tetrahedra, but it is helpful to arrange these tetrahedra $L'$ as the union of two ideal cubes in the construction of $\Lsing$.
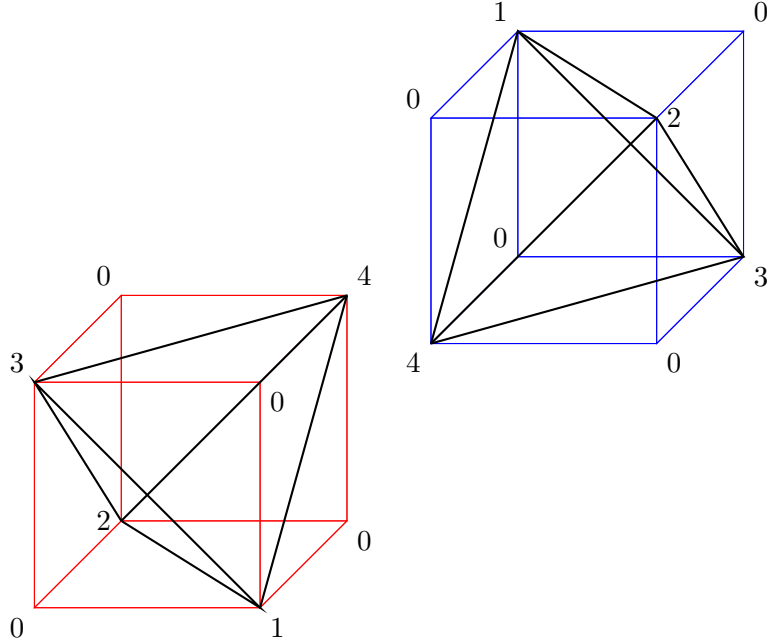
\begin{figure}
\begin{tikzpicture}
\begin{scope}[xshift =  75, yshift = 50, scale = 3]
\draw[blue] (0,0,0) -- (1,0,0) -- (1,1,0) -- (0,1,0) -- (0,0,0);
\draw[blue] (0,0,0) -- (0,1,0) -- (0,1,1) -- (0,0,1) -- (0,0,0);
\draw[blue] (0,0,0) -- (0,0,1) -- (1,0,1) -- (1,0,0) -- (0,0,0);
\draw[blue] (1,0,0) -- (1,1,0) -- (1,1,1) -- (1,0,1) -- (1,0,0);
\draw[blue] (0,1,0) -- (1,1,0) -- (1,1,1) -- (0,1,1) -- (0,1,0);
\draw[blue] (0,0,1) -- (1,0,1) -- (1,1,1) -- (0,1,1) -- (0,0,1);

\draw[thick, black] (0,0,1) -- (0,1,0) -- (1,0,0) -- (0,0,1);
\draw[thick, black] (0,0,1) -- (1,1,1);
\draw[thick, black] (0,1,0) -- (1,1,1);
\draw[thick, black] (1,0,0) -- (1,1,1);
\node[anchor = west] at (1,1,1) {$2$};
\node[anchor = north east] at (0,0,1) {$4$};
\node[anchor = south east] at (0,1,0) {$1$};
\node[anchor = north west] at (1,0,0) {$3$};

\node[anchor = south east] at (0,0,0) {$0$};
\node[anchor = south west] at (1,1,0) {$0$};
\node[anchor = north west] at (1,0,1) {$0$};
\node[anchor = south east] at (0,1,1) {$0$};
\end{scope}

\begin{scope}[xshift = -75, yshift = -50, scale = 3]
\draw[red] (0,0,0) -- (1,0,0) -- (1,1,0) -- (0,1,0) -- (0,0,0);
\draw[red] (0,0,0) -- (0,1,0) -- (0,1,1) -- (0,0,1) -- (0,0,0);
\draw[red] (0,0,0) -- (0,0,1) -- (1,0,1) -- (1,0,0) -- (0,0,0);
\draw[red] (1,0,0) -- (1,1,0) -- (1,1,1) -- (1,0,1) -- (1,0,0);
\draw[red] (0,1,0) -- (1,1,0) -- (1,1,1) -- (0,1,1) -- (0,1,0);
\draw[red] (0,0,1) -- (1,0,1) -- (1,1,1) -- (0,1,1) -- (0,0,1);

\draw[thick, black] (0,0,0) -- (1,0,1) -- (0,1,1) -- (0,0,0);
\draw[thick, black] (0,0,0) -- (1,1,0);
\draw[thick, black] (1,0,1) -- (1,1,0);
\draw[thick, black] (0,1,1) -- (1,1,0);
\node[anchor = east] at (0,0,0) {$2$};
\node[anchor = south west] at (1,1,0) {$4$};
\node[anchor = north west] at (1,0,1) {$1$};
\node[anchor = south east] at (0,1,1) {$3$};

\node[anchor = north west] at (1,1,1) {$0$};
\node[anchor = north east] at (0,0,1) {$0$};
\node[anchor = south east] at (0,1,0) {$0$};
\node[anchor = north west] at (1,0,0) {$0$};
\end{scope}
\end{tikzpicture}
\caption{A decomposition of $L'$ into two ideal cubes, each of which is written as the union of five ideal tetrahedra.}\label{idealtriangulation}
\end{figure}

\begin{remark}\label{spinstructure}
Using this ideal triangulation, we specified a spin structure on $L'$ from a smoothing of the $2$-dimensional complex dual to the triangulation. More precisely, from the dual complex, one can produce a branched surface in a non-canonical way. The particular smooth branched surface which determines our choice of spin structure essentially comes from an ordering of the ideal vertices in the triangulation.  For details on the construction of the smoothing, and the procedure used to produce a spin structure from this data, see the discussion near~\cite[Lemma 4.2]{Han24a}.
\end{remark}

\begin{remark}\label{weinvertex}
Given an arbitrary tubular neighborhood of the zero-section in $T^*T^3$, we can assume that $\pisyz^{-1}(B)$ is contained in this neighborhood by taking $B$ to be sufficiently small.
\end{remark}
We will now describe the singular point of $\Lsing$ in more detail.
\begin{definition}\label{hlcone}
The Harvey--Lawson cone is defined to be the subset
\begin{align*}
C_{HL}\coloneqq\left\lbrace |y_1| = |y_2| = |y_3|\text{ and }y_1y_2y_3\in\mathbb{R}_{\geq0}\right\rbrace\subset\mathbb{C}^3 \, .
\end{align*}
Let $\Lambda_{HL}\coloneqq C_{HL}\cap S^5\subset\mathbb{C}^3$ denote the Legendrian link of the cone point.
\end{definition}
The link of $C_{HL}$ is a standard Legendrian torus in $S^5(\epsilon)\subset\mathbb{C}^3$, there sphere of radius $\epsilon>0$, and it is given by the set of points $(x_1,x_2,x_3)\in C_{HL}$ for which $|x_1| = |x_2| = |x_3| = \sqrt{\frac{\epsilon}{3}}$. Note that $C_{HL}$ is homeomorphic to the cone over a $2$-torus, where the cone point lies at the origin in $\mathbb{C}^3$. One can also see that, up to a change of coordinates, the periodized conormals~\eqref{conormal123} and~\eqref{conormal4} can be identified with the smooth part of $C_{HL}$, under an appropriate change of coordinates.

The cone point of $\Lsing$ lies at $(0,0)\in T^*Q/T^*_{\mathbb{Z}}Q$. We will let $v_0$ denote this cone point. Consider the symplectomorphism $\phi_0\colon T^*\mathbb{R}^3\to T^*\mathbb{R}^3$ induced by the linear map
\begin{align*}
\begin{pmatrix} -1 & 1 & 1 \\ 1 & -1 & 1 \\ 1 & 1 & -1 \end{pmatrix}
\end{align*}
acting on $\mathbb{R}^3$. Denote the image of a ball in $T^*\mathbb{R}^3\cong\mathbb{C}^3$ of sufficiently small radius under the composition $\phi_0$ with the universal covering map map by
\begin{align}
B_0\subset T^*T^3 \, .\label{darbouxlsing}
\end{align}
This is a Darboux ball centered at $v_0$.

We can write an explicit set of generators for $H_1(L';\mathbb{Z})$ using the diagram for the minimally-twisted five-component chain link. A set of meridians $\lbrace m_i\rbrace$ and longitudes $\lbrace\ell_i\rbrace$ for the link components are depicted in Figure~\ref{loops}.

By examining the diagram, one finds the following relations between (the homology classes of) the meridians and longitudes.
\begin{lemma}\label{relations}
The longitudes can be expressed in terms of the meridians in $H_1(L';\mathbb{Z})$ by the following formulas:
\begin{align*}
\ell_0 &= -m_1-m_4 \\
\ell_1 &=-m_0+m_2 \\
\ell_2 &= m_1-m_3 \\
\ell_3 &= -m_2+m_4 \\
\ell_4 &= -m_0+m_3 \, .
\end{align*}
\qed
\end{lemma}
In~\cite[Lemma 4.10]{Han24a}, we determined the map $H_1(L';\mathbb{Z})\to H_1(T^3;\mathbb{Z})$ induced by the composition $L'\to T^*T^3\xrightarrow{\pisyz}T^3$ in terms of these generators. Let
\begin{align*}
e_1,e_2,e_3\in H_1(T_0^*Q/T_{\mathbb{Z},0}^*Q;\mathbb{Z}) = H_1(T^3;\mathbb{Z})
\end{align*}
denote the homology classes of the circles which lift to the coordinate axes in the cotangent fiber $T^*_0Q$.
\begin{lemma}\label{inducedmap}
The map $H_1(L')\to H_1(T^3)$ is determined by the following values
\begin{align*}
m_0 &\mapsto 0, \\
m_1 &\mapsto e_2-e_3, \\
m_2 &\mapsto e_3, \\
m_3 &\mapsto -e_1+e_2, \\
m_4 &\mapsto -e_2+e_3.
\end{align*}
Consequently, the values of this map on the longitudes are as follows.
\begin{align*}
\ell_0 &\mapsto 0, \\
\ell_1 &\mapsto e_3, \\
\ell_2 &\mapsto e_1-e_3, \\
\ell_3 &\mapsto -e_2, \\
\ell_4 &\mapsto -e_1+e_2.
\end{align*}
\qed
\end{lemma}
Instead of considering $\Lsing$ in our discussion of the quintic threefold, we will need to consider a certain cover of it, denoted $\tLsing$. More precisely, this is a Lagrangian lift of the tropical curve $\widetilde{V}$ obtained from $V$ by giving all of its edges weight $5$. One can construct $\tLsing$ from $\Lsing$ using the procedure described in~\cite[\S{5.2}]{Mat20}. Recall that $\Lsing$ was constructed in terms of a subset of $T^3$ called the \textit{coamoeba}, which is identified as the image of $\Lsing$ under the bundle projection $T^*T^3\to T^3$. In this language, the coamoeba of $\tLsing$ is obtained from the coamoeba of $\Lsing$ by taking its preimage under the $125$-fold covering map $T^3\to T^3$ corresponding to the subgroup $(5\mathbb{Z})^3\subset\mathbb{Z}^3$. It is easy to see that $\tLsing$ has $125$ cone points. Let $\widetilde{L}'$ denote the (cusped hyperbolic) $3$-manifold obtained by taking the complement of singular points on $\tLsing$.
\begin{remark}\label{125foldcover}
Using Lemma~\ref{inducedmap}, we can describe the cover $\widetilde{L}'\to L'$ explicitly. By choosing a basepoint in the center of Figure~\ref{loops}, we can obtain an element $\widehat{m}_i\in\pi_1(L')$ from each meridian of the minimally-twisted five-component chain link. Consider the normal closure in $\pi_1(L')$ of the subgroup generated by the elements $\lbrace\widehat{m}_0,\widehat{m}_1^5,\widehat{m}_2^5,\widehat{m}_3^5,\widehat{m}_4^5\rbrace$. A straightforward calculation shows that the quotient of $\pi_1(L')$ by this subgroup is isomorphic to $(\mathbb{Z}/5)^3$. It is also clear from the construction of $\widetilde{L}'$ that the image of $\pi_1(\widetilde{L}')\to\pi_1(\widetilde{L})$ contains this subgroup. In~\cite{Han24a}, we wrote $L'$ as the union of two ideal cubes in the hyperbolic upper half space $\mathbb{H}^3$, and it is interesting to note that $\widetilde{L}'$ can be written as the union of $250$ such cubes. These are glued together in a manner determined by the combinatorics of the lift of the \textit{coamoeba} used to construct $\Lsing$~\cite[\S{4.2}]{Han24a}, but we will not need this fact.
\end{remark}

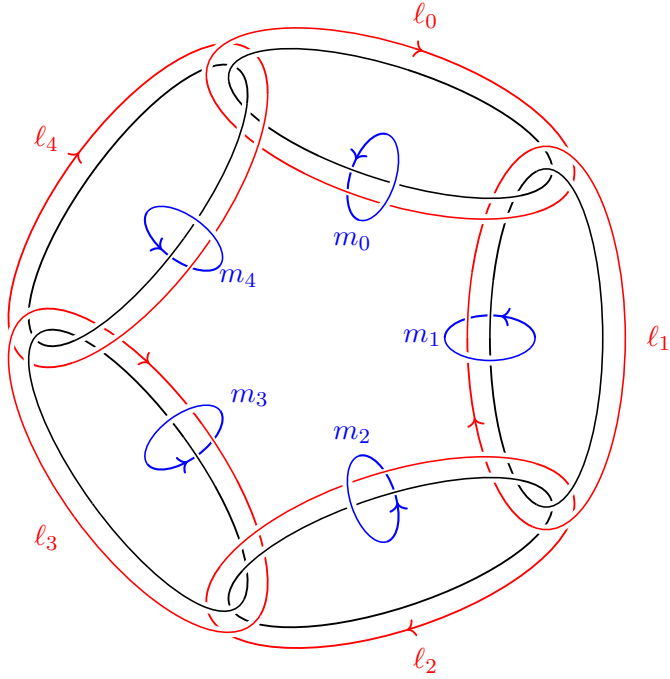
\begin{figure}
\tikzset{
 redarrows/.style={postaction={decorate},decoration={markings,mark=at position 0 with {\arrow[draw=red]{<}}},
           }}
 \tikzset{
  redarrows2/.style={postaction={decorate},decoration={markings,mark=at position -0.4 with {\arrow[draw=red]{<}}},
  			}}
\tikzset{
 bluearrows/.style={postaction={decorate},
 decoration={markings,mark=at position 0.2 with {\arrow[draw=blue]{>}}},
           }}
\begin{tikzpicture}[scale = 1.5]
\draw(2,0)[line width = 0.25 mm] ellipse (0.5 and 1.5);

\draw(2,0)[red, line width = 0.25 mm, redarrows2] ellipse (0.7 and 1.7);
\draw[rotate = 72, line width = 0.25 mm] (2,0) ellipse (0.5 and 1.5);
\draw(2,0)[rotate = 72, red, line width = 0.25 mm, redarrows] ellipse (0.7 and 1.7);

\draw[rotate = 144, line width = 0.25 mm] (2,0) ellipse (0.5 and 1.5);
\draw(2,0)[rotate = 144, red, line width = 0.25 mm, redarrows] ellipse (0.7 and 1.7);

\draw[rotate = 216, line width = 0.25 mm] (2,0) ellipse (0.5 and 1.5);
\draw(2,0)[rotate = 216, red, line width = 0.25 mm, redarrows2] ellipse (0.7 and 1.7);

\draw[rotate = 288, line width = 0.25 mm] (2,0) ellipse (0.5 and 1.5);
\draw(2,0)[rotate = 288, red, line width = 0.25 mm, redarrows] ellipse (0.7 and 1.7);

\draw (2.5,0)[white, double = black, ultra thick] arc (0:100:0.5 and 1.5);
\draw (2.7,0)[white, double = red, ultra thick] arc (0:100:0.7 and 1.7);

\draw (2.5,0)[white, double= black, ultra thick] arc (0:-100:0.5 and 1.5);
\draw (2.7,0)[white, double= red, ultra thick] arc (0:-100:0.7 and 1.7);

\draw[white, double = black, ultra thick, rotate = 72] (2.5,0) arc (0:100:0.5 and 1.5);
\draw[white, double = red, ultra thick, rotate = 72] (2.7,0) arc (0:100:0.7 and 1.7);

\draw[white, double = black, ultra thick, rotate = 144] (1.5,0) arc (180:240:0.5 and 1.5);
\draw[white, double = red, ultra thick, rotate = 144] (1.3,0) arc (180:240:0.7 and 1.7);

\draw[white, double = black, ultra thick, rotate = 72] (1.5,0) arc (180:240:0.5 and 1.5);
\draw[white, double = red, ultra thick, rotate = 72] (1.3,0) arc (180:240:0.7 and 1.7);

\draw[white, double = black, ultra thick, rotate = 144] (1.5,0) arc (180:120:0.5 and 1.5);
\draw[white, double = red, ultra thick, rotate = 144] (1.3,0) arc (180:120:0.7 and 1.7);

\draw[white, double = black, ultra thick, rotate = 216] (2.5,0) arc (0:-100:0.5 and 1.5);
\draw[white, double = red, ultra thick, rotate = 216] (2.7,0) arc (0:-100:0.7 and 1.7);

\draw[white, double = black, ultra thick, rotate = 216] (2.5,0) arc (0:100:0.5 and 1.5);
\draw[white, double = red, ultra thick, rotate = 216] (2.7,0) arc (0:100:0.7 and 1.7);

\draw[white, double = black, ultra thick, rotate = 288] (1.5,0) arc (180:240:0.5 and 1.5);
\draw[white, double = red, ultra thick, rotate = 288] (1.3,0) arc (180:240:0.7 and 1.7);

\draw[white, double = black, ultra thick, rotate = 288] (1.5,0) arc (180:120:0.5 and 1.5);
\draw[white, double = red, ultra thick, rotate = 288] (1.3,0) arc (180:120:0.7 and 1.7);

\draw[blue, thick, bluearrows] (1.5,0) ellipse (0.4 and 0.2);
\draw[white, double = blue, ultra thick] (1.1,0) arc (180:360:0.4 and 0.2);

\draw[white, double = black, thick] (1.5,0) arc (180:160:0.5 and 1.5);
\draw[white, double = red, thick] (1.3,0) arc (180:160:0.7 and 1.7);

\draw[blue, thick, rotate = 72, bluearrows] (1.5,0) ellipse (0.4 and 0.2);
\draw[white, double = blue, ultra thick, rotate = 72] (1.1,0) arc (180:360:0.4 and 0.2);

\draw[white, double = black, thick, rotate = 72] (1.5,0) arc (180:160:0.5 and 1.5);
\draw[white, double = red, thick, rotate = 72] (1.3,0) arc (180:160:0.7 and 1.7);

\draw[blue, thick, rotate = 144, bluearrows] (1.5,0) ellipse (0.4 and 0.2);
\draw[white, double = blue, ultra thick, rotate = 144] (1.1,0) arc (180:360:0.4 and 0.2);

\draw[white, double = black, thick, rotate = 144] (1.5,0) arc (180:160:0.5 and 1.5);
\draw[white, double = red, thick, rotate = 144] (1.3,0) arc (180:160:0.7 and 1.7);

\draw[blue, thick, rotate = 216, bluearrows] (1.5,0) ellipse (0.4 and 0.2);
\draw[white, double = blue, ultra thick, rotate = 216] (1.1,0) arc (180:360:0.4 and 0.2);

\draw[white, double = black, thick, rotate = 216] (1.5,0) arc (180:160:0.5 and 1.5);
\draw[white, double = red, thick, rotate = 216] (1.3,0) arc (180:160:0.7 and 1.7);

\draw[blue, thick, rotate = 288, bluearrows] (1.5,0) ellipse (0.4 and 0.2);
\draw[white, double = blue, ultra thick, rotate = 288] (1.1,0) arc (180:360:0.4 and 0.2);

\draw[white, double = black, thick, rotate = 288] (1.5,0) arc (180:160:0.5 and 1.5);
\draw[white, double = red, thick, rotate = 288] (1.3,0) arc (180:160:0.7 and 1.7);

\node[] at (3,0) {\color{red}$\ell_1$};
\node[] at (0.93,2.86) {\color{red}$\ell_0$};
\node[] at (-2.43,1.77) {\color{red}$\ell_4$};
\node[] at (-2.43,-1.77) {\color{red}$\ell_3$};
\node[] at (0.93,-2.86) {\color{red}$\ell_2$};

\node[] at (0.9,0) {\color{blue}$m_1$};
\node[] at (0.28,0.86) {\color{blue}$m_0$};
\node[] at (-0.73,0.53) {\color{blue}$m_4$};
\node[] at (-0.63,-0.53) {\color{blue}$m_3$};
\node[] at (0.28,-0.86) {\color{blue}$m_2$};

\end{tikzpicture}
\caption{The meridians $m_i$ (blue) and longitudes $\ell_i$ (red).}\label{loops}
\end{figure}

If we view $\tLsing$ as a subset of $(\mathbb{C}^*)^3$, then $\tLsing\setminus\pisyz^{-1}(B)$ can be written as the union of the following subsets
\begin{align}
\left\lbrace |u_j| = |u_k|\text{ and }u_{\ell} = re^{i\theta} \big\vert r\in[1,\infty)\text{ and }\theta = \frac{2m\pi}{5}\text{ for }m\in\mathbb{Z}\right\rbrace\subset(\mathbb{C}^*)^3 \label{coverconormal123}
\end{align}
whenever $\lbrace j,k,\ell\rbrace = \lbrace 1,2,3\rbrace$, and
\begin{align}
\left\lbrace |u_1| = |u_2| = |u_3| \text{ and }u_1u_2u_3 = re^{i\theta} \big\vert r\in(0,1] \text{ and }\theta = \frac{2m\pi}{5} \text{ for }m\in\mathbb{Z}\right\rbrace\subset(\mathbb{C}^*)^3 \, . \label{coverconormal4}
\end{align}
These are of course just lifts of~\eqref{conormal123} and~\eqref{conormal4}, respectively, under the symplectic covering map $(\mathbb{C}^*)^3\to(\mathbb{C}^*)^3$.

The Darboux ball $B_0$ centered at the cone point of $\Lsing$ has $125$ lifts, denoted
\begin{align}\label{darbouxtlsing}
\widetilde{B}_{\ell,0}\subset T^*T^3
\end{align}
for all $\ell = 1,\ldots,125$. These are all Darboux charts centered at the cone points of $\tLsing$. 

\subsection{A singular tropical Lagrangian in the quintic}
Consider the Lagrangian torus $L_{\tau,0}\subset X_{\tau}^5$ of Proposition~\ref{mirrors2points}, and fix a Weinstein neighborhood $\Wein(L_{\tau,0})\subset X_{\tau}^5$, together with a symplectomorphism $N^*_{\epsilon}T^3\to\Wein(L_{\tau,0})$ whose domain is a tubular neighborhood of the zero section in $T^*T^3$ of radius $\epsilon>0$. By Theorem~\ref{singularsyzlift} (cf. Remark~\ref{weinvertex}), we can isotope $\Lsing$, and hence $\tLsing$, so that it coincides with the periodized conormal Lagrangians outside of $N^*_{\epsilon}T^3\subset T^*T^3$. Abusing notation, we will identify $\tLsing$ with its image in $\Wein(L_{\tau,0})$. To compactify this noncompact Lagrangian in $X_{\tau}^5$, we will construct Lagrangian cones away from $\Wein(L_{\tau,0})$, and show that they patch smoothly with $\tLsing$. This gluing is controlled by the combinatorics of the tropical curve $\widetilde{V}$, as in the construction of tropical Lagrangians of~\cite{MR20}.

Let $\lbrace i,j,k,\ell\rbrace=\lbrace 1,2,3,4\rbrace$. Observe that there are five points in $X_{\tau}^5$ with homogneous coordinates $x_i = x_j = x_k = 0$ on $\mathbb{CP}^4$. Let $\widetilde{v}_{\ell,m}$, where $m = 1,\ldots,5$, denote these points. We define charts near each $\widetilde{v}_{\ell,m}$ as follows. If we set $x_{\ell} = 1$, then there is a ball in $B_0(\epsilon)\subset\mathbb{C}^3$ centered at $0$, where $\epsilon>0$ denotes the radius of the ball, such that if $(x_i,x_j,x_k)\in B_0(\epsilon)$, then any $x_5$ for which $\left[x_1:x_2:x_3:x_4:x_5\right]$ lies in $X_{\tau}^5$ is nonzero. The value of $x_5$ then determines a section of the (trivial) $5$-fold cover of $B_0(\epsilon)$. Let
\begin{align}
\widetilde{B}_{\ell,m}\subset X_{\tau}^5 \label{darbouxquintic}
\end{align}
denote the sheet of this cover containing the point $\widetilde{v}_{\ell,m}$. The ball $\widetilde{B}_{\ell,m}$ is a Darboux ball centered at $\widetilde{v}_{\ell,m}$, since the symplectic form on $X_{\tau}^5$ is pulled back from the Fubini--Study form on $\mathbb{CP}^4$.

In each of the balls ~\eqref{darbouxquintic}, there is a Lagrangian cone defined by
\begin{align}
L_{\ell,m}\coloneqq\left\lbrace\left[x_1:x_2:x_3:x_4:x_5\right]\mid|x_i| = |x_j| = |x_k|\text{ and }x_ix_jx_k\in\mathbb{R}_{\geq0}\right\rbrace\cap\widetilde{B}_{\ell,m} \, . \label{conesinxt}
\end{align}
These are just the images of the Harvey--Lawson cone in $\widetilde{B}_{\ell,m}$. To complete the construction of $\tLvgsing$, we will show that these cones can be patched smoothly with the copy of $\tLimm$ in $\Wein(T^3_{t,0})$.
\begin{theorem}\label{singularlagrquintic}
There is a singular Lagrangian $\tLvgsing$ in $X_{\tau}^5$, where $\tau$ is any sufficiently small real number, which coincides with a copy of $\tLsing$ contained in $\Wein(L_{\tau,0})\subset X_{\tau}^5$. The singular locus of $\tLvgsing$ consists of $145$ singular points modeled on the Harvey--Lawson cone, and the complement of these points in $\tLvgsing$ is diffeomorphic to the hyperbolic $3$-manifold $\widetilde{L}'$ as discussed in Remark~\ref{125foldcover}.
\end{theorem}
There are $125$ cone points that come from lifting the cone point of $\Lsing$, and another $20$ cone points attached to the ends of $\tLsing$, specifically $5$ cone points corresponding to each leg of the tropical curve $V$.
\begin{proof}
Consider the tropical curve $\widetilde{V}\subset Q\cong\mathbb{R}^3$. Choose a point $q\in\widetilde{V}_{\ell}$ which is not the vertex. If we view $(\mathbb{C}^*)^3$ as the big torus in $\mathbb{CP}^3\cong\lbrace x_5 = 0\rbrace$, then $\mathbb{R}^3$ can be identified with the interior of the moment polytope for (this copy of) $\mathbb{CP}^3$. Consider the moment fiber $L_{0,q}$ over $q$ inside $\mathbb{CP}^3$. As before, we can parallel transport this to a Lagrangian torus $L_{\tau,q}\subset X_{\tau}^5$. Choose a Weinstein neighborhood $\Wein(L_{\tau,q})$ contained in $X_{\tau}^5$, whose domain is a small tubular neighborhood of the $0$-section in $T^*T^3$.

If $W\subset Q$ is a $1$-dimensional integral affine subspace containing $\widetilde{V}_{\ell}$, then we can consider the intersection periodized conormal $N^*W/N^*_{\mathbb{Z}}W$ with the tubular neighborhood specified above. This will give us the portion of the periodized conormal lying over a finite segment of $\widetilde{V}_{\ell}$ containing $q$. Taking a sequence of such points on $\widetilde{V}_{\ell}$, we can assume that the union of Weinstein neighborhoods $\Wein(L_{\tau,q})$ is connected and intersects each of the balls $\widetilde{B}_{\ell,m}$, as well as $\Wein(L_{\tau,q})$.

Examining~\eqref{conesinxt} shows that the intersection of the union of cones $\bigcup_m\widetilde{L}_{\ell,m}$ with the Lagrangian segments in $\Wein(L_{\tau,q})$ constructed above pulls back to a segment of one of the subsets~\eqref{coverconormal123} or~\eqref{coverconormal4} inside $(\mathbb{C}^*)^3\subset\lbrace x_5 = 0\rbrace$ under symplectic parallel transport. In particular the union of the cones with the Lagrangian segment in $\Wein(L_{\tau,q})$, where $q\in Q$ is sufficiently far away from $0$, is smooth except at the cone points. For the same reason, one sees that the union of these pieces with the copy of $\tLsing\subset\Wein(L_{\tau,0})$ is smooth (away from the cone points) as well. By taking the union of all local Lagrangian pieces considered so far, we obtain $\tLvgsing$.
\end{proof}
\begin{remark}
There are several obvious similarities between the proof of Theorem~\ref{singularlagrquintic} and Mak--Ruddat's construction of tropical Lagrangians in mirror quintic threefolds of~\cite{MR20}. The Weinstein neighborhoods of moment fibers in our setting can be thought of as analogues of the charts used by Mak--Ruddat to construct the parts of their tropical Lagrangians lying away from the boundary of the moment polytope. Because we have attached Lagrangian cones to the noncompact ends of $\tLsing$, we do not require an analogue of their construction of a Lagrangian solid torus. Consequently, we have not undertaken a detailed study of the discriminant locus or singular fibers as in op. cit. In the language of the Gross--Siebert program, the relevant dual intersection complex would not be simple, and so constructing tropical Lagrangians in these terms would be different than in~\cite{MR20}.
\end{remark}
\begin{remark}
Since five subsets in~\eqref{coverconormal123}, for some fixed indices $j$, $k$, and $\ell$ are carried to each other under multiplication of the $\ell$th coordinate of $(\mathbb{C}^*)^3$ by fifth roots of unity, one might expect that they should all approach the same cone point. Since the points $\widetilde{v}_{\ell,m}$ are also carried to each other by multiplication by fifth roots of unity in the $\ell$th coordinate, however, each end of $\tLsing$ approaches a different cone point when compactified in $X_{\tau}^5$.
\end{remark}
The proof of Theorem~\ref{main1} involves understanding disks of small symplectic area with boundary on $\tLvg$. It is illuminating to recast these arguments as calculations of \textit{local} Lagrangian Floer cohomology in a Weinstein neighborhood of $\tLvgsing$ inside $X_{\tau}^5$. We will construct a Weinstein neighorhood $\widetilde{W}_5\coloneqq\Wein(\tLvgsing)$ which is invariant under the action of $(\mathbb{Z}/5)^3$ on $X_{\tau}^5$. It was shown by Joyce~\cite{Joy04} that any dilation-invariant Lagrangian cone $C$ in $\mathbb{C}^n$ with link $\Sigma$ has a dilation-invariant Weinstein neighborhood, meaning that there is an open neighborhood $U_C\subset T^*(\Sigma\times\mathbb{R}_{>0})$ of the $0$-section and a symplectic embedding $\Phi_C\colon U_C\to\mathbb{C}^n$. We also have that $\Phi_C$ restricts to the inclusion map along the $0$-section, and that it intertwines the $\mathbb{R}_{>0}$-actions on $T^*(\Sigma\times\mathbb{R}_{>0})$ and $\mathbb{C}^n$.

Let $\widetilde{v}_{\ell,m}$ denote one of the singular points of $\tLvgsing$, and let $\Psi_{\ell,m}\colon B_0(\epsilon)\to X_{\tau}^5$ denote one of the Darboux balls~\eqref{darbouxtlsing} or~\eqref{darbouxquintic}, where we take the domain to be  a ball of radius $\epsilon>0$ centered at $0\in\mathbb{C}^3$. Denote by $\psi_{\ell,m}\colon\mathbb{C}^3\to T_{x_s}X_{\tau}^5$ the linear isomorphism induced from the differential of $\Psi_{\ell,m}$ at the origin.

We can write $\Psi_{\ell,m}^{-1}(\tLvgsing)$ as the image of the $0$-section in $T^*(\Lambda_{HL}\times\mathbb{R}_{>0})$, where we recall that $\Lambda_{HL}$ denotes the link of the Harvey--Lawson cone. This induces a map $\phi_{\ell,m}\colon\Lambda_{HL}\times\mathbb{R}_{>0}\to B_0(\epsilon)$ which parametrizes the $0$-section, implying that $\Psi_{\ell,m}\colon\phi_s$ has image contained in $\widetilde{L}'\subset X_{\tau}^5$. Define the compact subset
\begin{align*}
K\coloneqq\tLvgsing\setminus\bigcup_{\ell,m}\widetilde{B}_{\ell,m}
\end{align*}
of $X_{\tau}^5$.
\begin{lemma}
There is an open neighborhood $U_{\widetilde{L}'}$ of the $0$-section in $T^*\widetilde{L}'$ and a symplectic embedding $\Phi_{\widetilde{L}'}\colon U_{\widetilde{L}'}\to X_{\tau}^5$ which restricts to the inclusion over the $0$-section. This embedding also satisfies
\begin{align}\label{singularweinsteinidentity}
\Phi_{\widetilde{L}'}\circ(d\Psi_{\ell,m}\circ\phi_{\ell,m}) = \Psi_{\ell,m}\circ\Phi_{C_{HL}}
\end{align}
meaning that the Darboux balls centered at the cone points of $\tLvgsing$ patch smoothly with the image of $U_{\widetilde{L}'}$ to give a symplectic subdomain $\tWvg$ of $X_{\tau}^5$. Moreover, we can choose all symplectic embeddings as above so that $\tWvg$ is invariant under the action of $(\mathbb{Z}/5)^3$.
\end{lemma}
\begin{proof}
The construction of $\Phi_{\widetilde{L}'}$ and $U_{\widetilde{L}'}$ satisfying~\eqref{singularweinsteinidentity} is straightforward using the techniques of~\cite{Joy04} and~\cite[Section 5.1]{Han24a}. Thus we can take $\tWvg$ to be the union of $\Phi_{\widetilde{L}'}(U_{\widetilde{L}'})$ with the balls $\widetilde{B}_{\ell,m}$. To prove that $\tWvg$ is $(\mathbb{Z}/5)^3$-invariant, first note that the cone points $\widetilde{v}_{\ell,m}$, for $\ell = 1,2,3,4$, are permuted by the group action. Because there are only finitely many such points, we can choose the Darboux balls centered at these points so that the disjoint union of all of the balls is preserved under the group action.

The action of $(\mathbb{Z}/5)^3$ on $X_{\tau}^5$ is the restriction of a corresponding $(\mathbb{Z}/5)^3$-action on $\mathbb{CP}^4$. This group action also restricts to an action of $(\mathbb{Z}/5)^3$ on $(\mathbb{C}^*)^3$, and this is in fact the action by deck transformations of the covering map $(\mathbb{C}^*)^3\to(\mathbb{C}^*)^3$ arising in the construction of $\widetilde{V}$. In particular $\tLsing$ is preserved under this action and, by Remark~\ref{graphtorus}, its symplectic parallel transport is also $(\mathbb{Z}/5)^3$-invariant. Hence wee can assume without loss of generality that the image of $\Phi_{\widetilde{L}'}$ is preserved by the $(\mathbb{Z}/5)^3$-action as well.
\end{proof}
The quotient of $\tWvg$ by the action of $(\mathbb{Z}/5)^3$ admits a similar description.
\begin{lemma}
The quotient space $\Wvg\coloneqq\tWvg/(\mathbb{Z}/5)^3$ is a smooth symplectic manifold. It contains a singular Lagrangian $\Lvgsing$ given by the image of $\tLvgsing$ under the quotient map. The Lagrangian $\Lvgsing$ has five cone points, denoted $v_0,\ldots,v_4$, all of which are modeled on the cone point of a Harvey--Lawson cone.
\end{lemma}
\begin{proof}
Consider the disjoint union
\begin{align*}
\coprod_{m=1}^5\widetilde{B}_{\ell,m}
\end{align*}
for a fixed $\ell = 1,2,3,4$. We can choose a generating set for $(\mathbb{Z}/5)^3$ with the property that one of the generators cyclically permutes the balls $\widetilde{B}_{\ell,m}$, and for which the action of the other two generators on $\widetilde{B}_{\ell,m}\cong\mathbb{C}^2$ is given by the action of
\begin{align*}
(\mathbb{Z}/5)^2\cong\lbrace(a_1,a_2,a_3)\in(\mathbb{Z}/5)^3\mid a_1+a_2+a_3 = 0\rbrace
\end{align*}
on $\mathbb{C}^3$. This group action is associated to a branched cover of $\mathbb{C}^3$ by itself. The action of $(\mathbb{Z}/5)^3$ on $\Phi_{\widetilde{L}'}(U_{\widetilde{L}'})$ is free, and thus $\Wvg$ is a smooth manifold. Abstractly, it is obtained by gluing five symplectic balls to a copy of $T^*L'$ such that the Legendrians $\Lambda_{HL}\subset S^5$ in the boundaries of these balls are identified with the cusps of $L'$, thought of as the $0$-section.
\end{proof}
We denote the Darboux ball centered at the cone point $v_{\ell}$ of $\Lvgsing$ by
\begin{align}
B_{\ell}\subset\Wvg \, .
\end{align}
This is consistent with the notation~\eqref{darbouxlsing} for the Darboux chart near the cone point of $\Lsing$ established above.

\subsection{An immersed tropical Lagrangian in the quintic} The Lagrangian immersion $\tLvg$ is closely related to the immersion $\Limm$ studied in~\cite{Han24a}. More precisely, one can think of $\tLvg$ as being obtained from a cover $\tLimm$ of $\Limm$ by attaching immersed `doubles' of the Harvey--Lawson cone, as constructed in~\cite[\S{5.3}]{Han24a}. A consequence of this is that we will be able to view $\tLvg$ as an object of the relative Fukaya category of the quintic, which will later allow us to quote the results of~\cite{She15}.

Recall that the construction of the double takes place in a small neighborhood of the origin in $\mathbb{C}^3$. We can think of this as a neighborhood of $0\in\mathbb{C}^3\setminus\lbrace y_1y_2y_3 = 1\rbrace$, which we can identify with the variety
\begin{align}\label{gssing}
Y\coloneqq\left\lbrace(y_1,y_2,y_3,u)\in\mathbb{C}^3\times\mathbb{C}^*\mid y_1y_2y_3 = v\right\rbrace \, .
\end{align}
This is, not coincidentally, one of the local Gross--Siebert models studied in~\cite{AS21}. Consider the projection
\begin{align} \label{degeneration}
w\colon Y &\to\mathbb{C} \nonumber \\
w(y_1,y_2,y_3,u) &= y_1y_2y_3 \, . 
\end{align}
The fiber $D\coloneqq w^{-1}(0)$ is the union of coordinate hyperplanes in $\mathbb{C}^3$. As shown in~\cite[Lemma 2.2]{AS21}, the coordinates $(y_1,y_2,w)$ on $Y\setminus D$ induce a commutative diagram
\begin{equation}
\begin{tikzcd}
Y\setminus D\arrow{r}{\cong}\arrow{d} & \mathbb{C}^2\times(\mathbb{C}^*\setminus\lbrace 1\rbrace)\arrow{d} \\
Y\arrow{r} & \mathbb{C}^2\times\mathbb{C}
\end{tikzcd}\label{commutative}
\end{equation}
where the top arrow is a symplectomorphism.

Let $\Larc$ denote the immersed Lagrangian arc in $\mathbb{C}\setminus\lbrace 0,1\rbrace$ depicted in Figure~\ref{immersedarc}. Notice $\Larc\subset\mathbb{C}$ bounds a holomorphic teardrop, i.e. a disk with one corner, through the origin. By choosing $\Larc$ appropriately, we can assume that this teardrop is arbitrarily small.
\begin{figure}
\begin{tikzpicture}\draw[smooth, thick, fill = gray!30] 
  plot[domain=135:225,samples=200] (\x:{3*cos(2*\x)});
  
  \draw[rounded corners, thick] (0,0) -- (0.7,0.7) -- (5.7,0.7);
  \draw[rounded corners, thick] (0,0) -- (0.7,-0.7) -- (5.7,-0.7);
  
  \node[] at (-3.5,0) {$\Larc$};
  \node[circle, fill, inner sep = 1.5 pt] at (-1.75,0) {};
  \node[anchor = west] at (-1.75,0) {$0$};
 
\end{tikzpicture}
\caption{The immersed Lagrangian $\Larc\subset\mathbb{C}$, and the teardrop it bounds.}\label{immersedarc}
\end{figure}
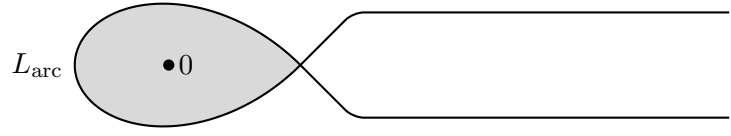
If we restrict to a neighborhood of $0\in\mathbb{C}^3$, then the portion of the Harvey--Lawson cone contained in this neighborhood projects, under $w$, to a portion of the nonnegative real axis in $\mathbb{C}$. Using~\eqref{commutative}, we can define an immersed Lagrangian in $\mathbb{C}^3$.
\begin{definition}\label{lydef}
Let $T^2_{Cl}\subset\mathbb{C}^2$ denote the standard Clifford torus, and consider the product $T^2\times\Larc\subset\mathbb{C}^2\times(\mathbb{C}^*\setminus\lbrace 1\rbrace)$. Let $L_Y\subset Y\setminus D$ denote the image of this Lagrangian under the symplectomorphism at the top of~\eqref{commutative}.
\end{definition}
Near the cone points $\widetilde{v}_{\ell,m}$, the Lagrangian immersion $\tLvg$ will be given by the images in $\widetilde{B}_{\ell,m}$ of the Lagrangian $L_Y\cap B_0(\epsilon)\subset B_0(\epsilon)$ contained in a small ball centered at the origin.
\begin{remark}\label{orderingcoordinates}
Although the construction of $L_Y$ depends on an ordering of coordinates on $\mathbb{C}^3$, the fibers of $L_Y$ over $\Larc$, as subsets of $\mathbb{C}^3$, are independent of this choice. Therefore we can construct $\tLvg$ using any ordering of coordinates on $\mathbb{C}^3$. It will only be important to specify a choice of coordinates later, when we study the Floer theory of this Lagrangian.
\end{remark}

The part of $\tLvg$ lying outside the balls $\widetilde{B}_{\ell,m}$ will consist of two copies of $\widetilde{L}'$.
\begin{lemma}\label{morsefn}
There is a Morse function $h\colon L'\to\mathbb{R}$ with $2$ index $0$ critical points, $6$ index $1$ critical points, and $4$ index $2$ critical points. Furthermore, there exist collar neighborhoods $T^2\times\mathbb{R}\to L'$ of each cusp of $L'$ in which the gradient vector field of $h$ points outward in the $\mathbb{R}$-direction. Here the gradient vector field is taken with respect to the metric on $L'$ determined by the symplectic form on $T^*L'$ and a compatible almost complex structure.
\end{lemma}
\begin{proof}
The Morse function $h$ is associated to the decomposition of $L'$ into two ideal cubes shown in Figure~\ref{idealtriangulation}. The numbers of index-$0$, -$1$, and -$2$ critical points of $h$ correspond to the numbers of $3$-cells, $2$-cells, and $1$-cells in the ideal cubulation.
\end{proof}

\begin{definition}
Let $\widetilde{h}\colon\widetilde{L}'\to\mathbb{R}$ be the Morse function obtained by precomposing $h$ with the covering map $\widetilde{L}'\to L'$.
\end{definition}
Since $\widetilde{h}$ is obtained from $h$ by lifting it to a covering space, it admits a similar description near the cusps of $\widetilde{L}'$.
\begin{corollary}
There exist collar neighborhoods $T^2\times\mathbb{R}\to\widetilde{L}'$ of each cusp of $\widetilde{L}'$ in which the gradient vector field of $\widetilde{h}$ points outward in the $\mathbb{R}$-direction. \qed
\end{corollary}
Let $\Gamma(d\widetilde{h})\subset T^*\widetilde{L}'$ denote the graph of the $1$-form $d\widetilde{h}$. The graphs $\Gamma(d\widetilde{h})$ and $\Gamma(-d\widetilde{h})$ intersect each other transversely, and the intersection points correspond to the critical points of $\widetilde{h}$. Choosing $\widetilde{h}$ with small $C^1$-norm means that we can assume that both of these graphs are arbitrarily $C^0$-close to the $0$-section, and in particular that they are contained in the neighborhood $U_{\widetilde{L}'}$.

\begin{lemma}
The union
\begin{align}
\tLvg\coloneqq\Phi_{\widetilde{L}'}(\Gamma(d\widetilde{h})\cup\Gamma(-d\widetilde{h}))\cup\bigcup_{\ell,m}\Psi_{\ell,m}(L_Y)\subset\tWvg
\end{align}
is the image of a smooth Lagrangian immersion, possibly after applying Hamiltonian isotopies to $L_Y$ away from the origin in $\mathbb{C}^3$.
\end{lemma}
\begin{proof}
This is similar to the discussion in~\cite[\S{5.3}]{Han24a}. Modifying $L_Y$ by Hamiltonian isotopies gives us enough freedom to guarantee that $\tLvg$ is smooth away from the self-intersection points.
\end{proof}
Notice that $\tLvg$ is contained in the very affine quintic hypersurface $X_{\tau}^5\setminus D_{\tau}$, where we recall that $D_{\tau}$ is the intersection of the quintic with the toric boundary of $\mathbb{CP}^4$. By construction, $L_Y$ avoids the coordinate hyperplanes (which coincide with the coordinate hyperplanes of $\mathbb{CP}^4$ in our choices of coordinates). A straightforward adaptation of the proof of~\cite[Lemma 5.5]{Han24a} yields the following.
\begin{lemma}\label{exactness}
The Lagrangian immersion $\tLvg\subset X^5_{\tau}\setminus D_{\tau}$ is exact with respect to the Liouville form on $X^5_{\tau}$ for which $L_{\tau,q}$ is exact. \qed
\end{lemma}

\begin{remark}\label{immersiondomain}
The domain of $\tLvg$ is a closed $3$-manifold that has a JSJ decomposition with two irreducible pieces, both of which are diffeomorphic to $\widetilde{L}'$. More precisely, we can describe the domain of $\tLvg$ as the $3$-manifold obtained by gluing together two copies of $\widetilde{L}'$ by deleting small collar neighborhoods of their cusps. Abusing notation, we will also call this manifold with boundary $\widetilde{L}'$. The curves $\lbrace m_i,\ell_i\rbrace$ drawn in Figure~\ref{loops} lift to curves in $\widetilde{L}'$ as in Remark~\ref{125foldcover}. We glue these two manifolds together in a way that respects these homology classes.
\end{remark}
Since we defined the Morse function $\widetilde{h}$ by pulling back a Morse function on $L'$, we have the following.
\begin{lemma}
The Lagrangian $\tLvg$ is invariant under the action of $(\mathbb{Z}/5)^3$, and thus it descends to a Lagrangian immersion whose image is $\Lvg\subset\Wvg$. \qed
\end{lemma}
The intersection of $\Lvg$ with one of the balls $B_{\ell}$ can be identified with a copy of $L_Y$. The intersection of $\Lvg$ with $\Wvg\setminus\bigcup_{\ell} B_{\ell}$ can be identified with a copy of
\begin{align}\label{smoothpart}
\Gamma(dh)\cup\Gamma(-dh)\subset T^*L' \, . 
\end{align}
There is a Lagrangian immersion whose image is $\Lvg$ and whose domain is a quotient of the domain of $\tLvg$ under the action of $(\mathbb{Z}/5)^3$. Specifically, the domain of this immersion is the $3$-manifold obtained by gluing two copies of the minimally-twisted five component chain link complement together using an identification of their boundaries which respects the distinguished generators of $H_1(L';\mathbb{Z})$, as in Remark~\ref{immersiondomain}.

Notice that there is an immersed Lagrangian $\Limm\subset T^*T^3$ gotten by gluing a copy of $L_Y$ in the Darboux ball $B_0$ of~\eqref{darbouxlsing} to the images of the graphs~\eqref{smoothpart} of $\pm dh$ inside a Weinstein neighborhood of $\Lsing$. This is essentially the Lagrangian of~\cite[Theorem 1.2]{Han24a}.
\begin{remark}\label{differentmorse}
In~\cite{Han24a} we constructed $\Limm$ using a different Morse function, which we denote by $h'$ in this remark. The choice of Morse function in the present paper simplifies the computations of Floer theory carried out in \S\ref{localhfsection}. We will now explain why the conclusion of Theorem 1.2. in op. cit. still holds if one uses the Morse function $h$ constructed in Lemma~\ref{morsefn} instead. Recall that $h'$ is obtained from the ideal triangulation on $L'$ subordinate to the ideal cubulation used to construct $h$ (cf. Figure~\ref{idealtriangulation}).

There are three natural projections $T^*T^3\to T^*T^2$, which we can use to naturally identify the Lagrangian pair of pants in $T^*T^2$ constructed by Matessi~\cite{Mat20} with smooth submanifolds of $L'$. In Figure~\ref{idealtriangulation}, they can be thought of as the union of the two triangles with black edges and vertices at the $j$th, $k$th and $4$th vertices, where $j,k\in\lbrace 1,2,3\rbrace$. By generically perturbing $h$ if necessary, we can assume that it restricts to a Morse function on each of these three copies of the $2$-dimensional pair of pants. Inspecting the proof of~\cite[Lemma 6.4]{Han24a}, which describes the image of $\Limm$ under certain Lagrangian correspondences in $(T^*T^3)^{-}\times(T^*T^2)$, shows that the conclusion of this Lemma holds with no changes when we replace $h'$ with $h$. Thus the main results of~\cite{Han24a}, in particular the formulas for the support of the mirror sheaf to $\Limm$, hold for the version of $\Limm$ constructed using $h$ with no additional changes to the proofs.
\end{remark}

To make sense of the Floer theory of $\tLvg$, we need to equip it with a grading, in the sense of~\cite{Sei00}, and a spin structure. A suitable grading is obtained from a grading on $\Limm$ in a straightforward way.
\begin{lemma}\label{gradingchoice}
There is a grading $\alpha^{\#}\colon\tLvg\to\mathbb{R}$ which is approximately equal to $0$ near the critical points of $-\widetilde{h}$ and approximately equal to $1$ near the critical points of $\widetilde{h}$. Moreover, it is lifted from a grading on $\Lvg$.
\end{lemma}
\begin{proof}
In~\cite[Lemma 5.6]{Han24a}, we constructed a grading on $\Limm$ with the desired values near the critical points of $h\colon L'\to\mathbb{R}$, which we could take to be constant when restricted to the part of $\Limm$ lying over the $1$-cones of the underlying tropical curve $V$. This lifts to a grading on $\tLimm\subset T^*T^3$. Under the identification of a Weinstein neighborhood of the zero section $T^*T^3$ with $\Wein(L_{\tau,0})\subset X_{\tau}^5$, it follows that we can define a grading on the part of $\tLvg$ lying in this neighborhood. This extends to a grading on the rest of $\tLimm$ in a unique way.
\end{proof}
We can equip $\tLvg$ with a spin structure by lifting the spin structure on $L'$ specified in~\cite[\S{4.1}]{Han24a} to a spin structure on $\widetilde{L}'$, and gluing the spin structures on the two copies of $\widetilde{L}'$ in the JSJ decomposition of $\tLvg$. We can further assume that this spin structure coincides with the one obtained by lifting the spin structure on $\Limm$ used in~\cite{Han24a} and extending it over the copies of $L_Y$ in the Darboux charts $\widetilde{B}_{\ell,m}$, for $\ell\in\lbrace 1,2,3,4\rbrace$. This does not describe a unique spin structure on $\tLvg$, but we will specify a particular choice for one during the proof of Proposition~\ref{localm0}.

Having fixed a grading on $\tLvg$, we can determine its Floer cochain space. See Appendix~\ref{immersedfloerappendix} for a review of the relevant definitions.
\begin{lemma}\label{floercochainsvg}
As a graded $\Lambda_0$-module, the Floer cochain complex $CF^*(\tLvg)$ has underlying $\mathbb{C}$-vector space given $\overline{CF}^*(\tLvg)$ given by
\begin{align}
\Omega^*(\tLvg)\oplus\bigoplus_{\ell,m}\Omega^*(T^2)[1]\oplus\bigoplus_{\ell,m}\Omega^*(T^2)[-2]\oplus CM^*(\widetilde{L}')[-1]\oplus CM^*(\widetilde{L}',\partial\widetilde{L}')[1]
\end{align}
where $\Omega^*$ denotes the de Rham cochain space, and $CM^*$ denotes the Morse cochain space. The direct sums range over the set
\begin{align*}
\left\lbrace(\ell,m)\in\mathbb{Z}\times\mathbb{Z}\mid \ell\in\lbrace 0,\ldots,4\rbrace \, , \; m\in\lbrace 1,\ldots,125\rbrace\text{ if }\ell = 0 \text{ and } m\in\lbrace 1,\ldots,5\rbrace\text{ if }\ell\neq 0\right\rbrace
\end{align*}
which has $145$ elements.
\end{lemma}
We think of $CM^*(\widetilde{L}')$ and $CM^*(\widetilde{L}',\partial\widetilde{L}')$ as copies of the spaces of differential forms on zero-dimensional manifolds given by the Morse critical points of $\widetilde{h}$ or $-\widetilde{h}$, respectively.
\begin{proof}[Proof of Lemma~\ref{floercochainsvg}]
Consider the fiber product $\tLvg\times_{X_{\tau}^5}\tLvg$, and note that its switching components are either copies of $T^2$ or transverse double points.

The $1$-dimensional Lagrangian $\Larc\subset\mathbb{C}$ has a single transverse double point, which should contribute two generators to its Floer cochain space. Since $\Larc$ bounds an isolated teardrop, i.e. a disk with one corner on a switching component, it follows that one of these generators must have degree $2$, and the Poincar{\'e} dual generator must have degree $-1$. Taking the product to form $L_Y$ shows that each of these critical points are replaced by copies of $\Omega^*(T^2)$ with the corresponding grading shifts. Here we can assume that the orientation local system, which appears in the general definition~\eqref{underlyingvs} of the $\mathbb{C}$-vector space underlying the Floer cochain complex, is trivial, since the immersion is locally described as the product of an embedded Lagrangian torus in $\mathbb{C}^2$ with an immersed Lagrangian submanifold with at worst transverse double points.

Note that the $0$-dimensional switching components all correspond to pairs $(p_{-},p_{+})$, where $p_{-}$ is a Morse critical point of $-\widetilde{h}$ and $p_{+}$ is the corresponding critical point of $\widetilde{h}$. It follows from~\cite{Sei00} that the grading shifts associated to these points are determined by the Morse indices. More precisely
\begin{align*}
\deg(p_{-}) &= \deg_{\mathrm{Morse}}(p_{-})-\alpha^{\#}(p_{-})+\alpha^{\#}(p_{+}) \\
&= \deg_{\mathrm{Morse}}(p_{-})-\alpha^{\#}(p_{-})+\alpha^{\#}(p_{+})+1 \\
&= \deg_{\mathrm{Morse}}(p_{-})+1
\end{align*}
Symmetrically, one has that
\begin{align*}
\deg(p_{+}) = \deg_{\mathrm{Morse}}(p_{+})+1
\end{align*}
for generators on the positive sheet. These yield the expected grading shifts on the copies of $CM^*(\widetilde{L}')$ and $CM^*(\widetilde{L}',\partial\widetilde{L}')$.
\end{proof}

\subsection{An embedded tropical Lagrangian in the quintic}\label{embeddedtrop}
We can use $\tLvgsing$ to construct three families of \textit{embedded} Lagrangian submanifolds $\tLvgsmooth(i;\epsilon)$, for $i = 1,2,3$, of $X_{\tau}^5$ which can be thought of as the tropical Lagrangian lifts of smooth tropical curves in the quintic, similar to the tropical Lagrangians in the \textit{mirror quintic} constructed by Mak--Ruddat~\cite{MR20}. Instead of relying on Mak--Ruddat's construction of Lagrangian solid tori near the singular fibers of an SYZ fibration, however, we will instead make use of Lagrangian solid tori which are thought of as asymptotically conical fillings of the links of the cone points of $\tLvgsing$.

The link of the Harvey--Lawson cone $C_{HL}$ has three asymptotically conical fillings given in coordinates by
\begin{align}
C_{HL}(i;\epsilon) = \lbrace|y_i|^2-\epsilon = |y_j|^2 = |y_k|^2 \, , \, y_1y_2y_3\in\mathbb{R}_{\geq0}\rbrace \label{acsmoothings}
\end{align}
for all $\lbrace i,j,k\rbrace = \lbrace 1,2,3\rbrace$. We can associate to each of these Lagrangian solid tori an embedded Lagrangian submanifold of $T^*T^3$ which is $C^0$-close to $\Lsing$ and which can be thought of as the lift of a tropical smoothing of $V\subset Q$. By a tropical smoothing of $V$, we mean one of the tropical curves $V(i;\epsilon)$, for $i = 1,2,3$, where $V(1;\epsilon)$ is given explicitly by
\begin{align}\label{tropsmoothing}
V(1;\epsilon) = \left\lbrace\left(0,t+\frac{\epsilon}{2},\frac{\epsilon}{2}\right)\colon t\in[0,\infty)\right\rbrace &\cup\left\lbrace\left(0,\frac{\epsilon}{2},t+\frac{\epsilon}{2}\right)\colon t\in[0,\infty)\right\rbrace \\
&\cup\left\lbrace\left(0,-t+\frac{\epsilon}{2},-t+\frac{\epsilon}{2}\right)\colon t\in[0,\epsilon]\right\rbrace \nonumber\\
&\cup\left\lbrace\left(t,-\frac{\epsilon}{2},-\frac{\epsilon}{2}\right)\colon t\in[0,\infty)\right\rbrace \nonumber\\
&\cup\left\lbrace\left(-t,-t-\frac{\epsilon}{2},-t-\frac{\epsilon}{2}\right)\colon t\in[0,\infty)\right\rbrace \nonumber 
\end{align}
and where $V(2;\epsilon)$ and $V(3;\epsilon)$ are obtained by cyclically permuting coordinates on $Q$.
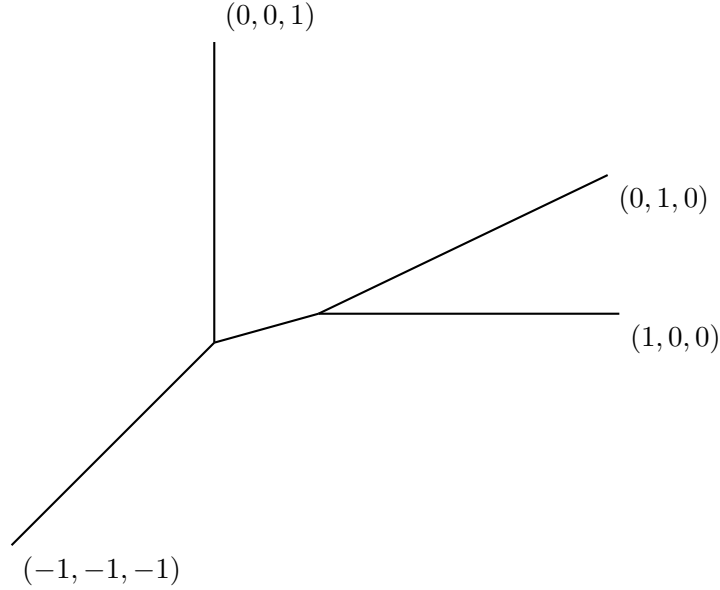
\begin{figure}
\begin{tikzpicture}
\begin{scope}
  \draw[thick] (0,0,0) -- (4,0,0) node[anchor = north west]{$(1,0,0)$};
    \draw[thick] (0,0,0) -- (5,3,3) node[anchor = north west]{$(0,1,0)$};
    \draw[thick] (0,0,0) -- (-1,0,1);
  \draw[thick] (-1,0,1) -- (-1,4,1) node[anchor = south west]{$(0,0,1)$};
  \draw[thick] (-1,0,1) -- (-1,0,8) node[anchor = north west]{$(-1,-1,-1)$};
\end{scope}
\end{tikzpicture}

\caption{A tropical curve $V(i;\epsilon)$}\label{smthfig}
\end{figure}

To construct the smoothings of $\Lsing$, recall from the proof of~\cite[Lemma 4.7]{Han24a} that the link of the singular point in $\Lsing$ maps to a $2$-sphere in $Q$ centered at the origin. The restriction of this projection map to such a Legendrian link is generically $2$-to-$1$, except over a tetrahedral graph embedded on the sphere, over which it is $1$-to-$1$. The vertices of this graph lie on the $1$-dimensional cones of $V$ (cf. Figure~\ref{front}).
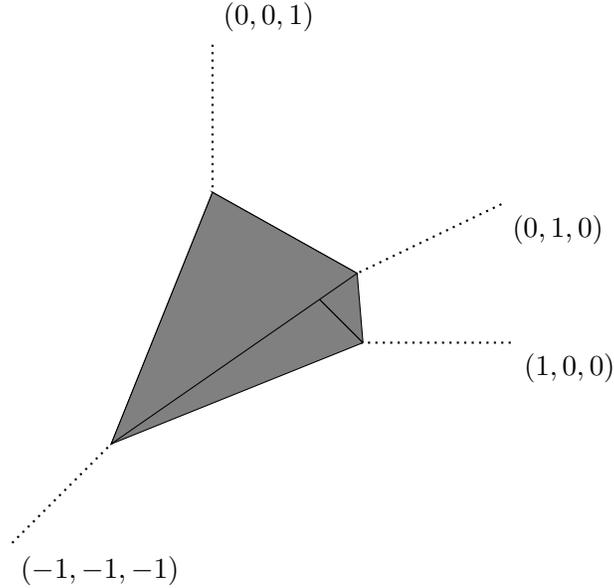
\begin{figure}
\begin{tikzpicture}
\begin{scope}
  \draw[thick, dotted] (0,0,0) -- (4,0,0) node[anchor = north west]{$(1,0,0)$};
  \draw[thick, dotted] (0,0,0) -- (0,4,0) node[anchor = south west]{$(0,0,1)$};
  \draw[thick, dotted] (0,0,0) -- (0,0,7) node[anchor = north west]{$(-1,-1,-1)$};
  \draw[thick, dotted] (0,0,0) -- (5,3,3) node[anchor = north west]{$(0,1,0)$};

  \draw[fill = gray, fill opacity = 0.6] (2,0,0) -- (0,2,0) -- (0,0,3.5) -- (2,0,0);
  \draw[fill = gray, fill opacity = 0.4] (2,0,0) -- (2.5,1.5,1.5) -- (0,2,0) -- (2,0,0);
  \draw[fill = gray, fill opacity = 0.1] (2.5,1.5,1.5) -- (0,0,3.5) -- (0,2,0) -- (2.5,1.5,1.5);
\end{scope}
\end{tikzpicture}

\caption{The image in $Q$ of a link of the cone point of $\Lsing$ is obtained by smoothing the edges of this tetrahedron. The dotted lines are the edges tropical curve $V$.}\label{front}
\end{figure}
In the Darboux coordinates we have chosen near $B_0$ in~\eqref{darbouxlsing}, the asymptotically conical fillings~\eqref{acsmoothings} will project to subsets of $Q$ containing the origin bounded by (smoothings of) tetrahedra with vertices on the semi-infinite $1$-dimensional cones of the smooth tropical curves $V(i;\epsilon)$. We can represent the subsets
\begin{align*}
\lbrace|y_i|^2-\epsilon = |y_j|^2 = |y_k|^2 \, , \, y_1y_2y_3\in\mathbb{R}_{>0}\rbrace
\end{align*}
as graphs of closed $1$-forms on the complement $T^2\times(0,1)\cong C_{HL}\setminus\lbrace 0\rbrace$ of the singular point of the Harvey--Lawson cone. Let $\partial_0 L'$ denote the link of the singular point in $\Lsing$. Since the map $H^1(L';\mathbb{R})\to H^1(\partial_0 L';\mathbb{R})$ is a surjection, we can extend such a $1$-form to a closed $1$-form $\delta_i$ on the complement $L' = \Lsing\setminus\lbrace(0,0)\rbrace$ of the cone point in $\Lsing$. 

Without loss of generality, we can assume that $\delta_i$ takes a certain standard form outside of the compact subset $\pisyz^{-1}(B)$ of Theorem~\ref{singularsyzlift} of $T^*Q/T^*_{\mathbb{Z}}Q$, in which $\Lsing$ coincides with the periodized conormals to the legs $V_{\ell}$ of $V$. More precisely, $\delta_i$ should coincide with a closed $1$-form on each of the periodized conormals $N^*V_{\ell}/N_{\mathbb{Z}}^*V_{\ell}\cong T^2\times\mathbb{R}_{>0}$ which is constant in the $\mathbb{R}_{>0}$-direction. Deforming the periodized conormals by such a $1$-form corresponds to translating them in $Q$. Thus, perturbing $\Lsing\setminus\lbrace(0,0)\rbrace$ by $\delta_i$ and taking the union of the resulting (non-properly) embedded Lagrangian with $C_{HL}(i;\epsilon)$ yields an embedded Lagrangian submanifold in $T^*Q/T^*_{\mathbb{Z}}Q$ which projects to a neighborhood of $V(i;\epsilon)$ in $Q$. These give embedded Lagrangian submanifolds in $X_{\tau}^5$ as follows.
\begin{theorem}\label{lsmthconst}
For each $i = 1,2,3$, there is a family of embedded Lagrangian submanifolds $\tLvgsmooth(i;\epsilon)$ contained in $\tWvg$, where $\epsilon>0$ is sufficiently small, which are topologically obtained by replacing neighborhoods of the cone points $\widetilde{v}_{\ell,m}$ in $\tLvgsing$ with copies of one of the fillings~\eqref{acsmoothings}.
\end{theorem}
\begin{proof}
Consider the singular Lagrangian submanifold $\Lvgsing\subset\Wvg$. Recall that the smooth locus of $\Lvgsing$ is diffeomorphic to the minimally-twisted five-component chain link complement $L'$, and by abuse of notation let $\partial L_{\ell}$ denote the link of the $\ell$th cusp of $L'$, thought of here as a link of one of the conical singular points of $\Lvgsing$, where $\ell = 0,1,2,3,4$. The $1$-form $\delta_i$ on $L'$ can then be pulled back to a closed $1$-form $\widetilde{\delta}_i$ on $\widetilde{L}'$.

In the charts $\widetilde{B}_{\ell,m}$, for $\ell = 1,2,3,4$, the $1$-form $\widetilde{\delta}_i$ restricts to a $1$-form on the smooth part of the Harvey--Lawson cone $C_{HL}\setminus\lbrace 0\rbrace$ from which one obtains the asymptotically conical smooth fillings. From this we obtain the embedded Lagrangian submanifolds $\tLvgsmooth(i;\epsilon)$ in the statement of the theorem.
\end{proof}
We will often omit $i$ and $\epsilon$ from the notation for these embedded Lagrangian submanifolds and the $1$-forms used to construct them. As we will see in \S\ref{lsmthchern}, any of the Lagrangian submanifolds of Theorem~\ref{lsmthconst} will support an isomorphic family of objects in the Fukaya category of the quintic, so the specific choices of $i$ and $\epsilon$ will be irrelevant for our purposes.
\section{Mirrors to lines in the mirror quintic}\label{mirrorstolinessect}
In this section, we will prove Theorem~\ref{main1}. We established a local version of this result in~\cite[Theorem 1.2]{Han24a}. Equipping the Lagrangian torus $L_{\tau,0}$ with a $GL(1,\mathbb{C})$-flat connection yields an object of the Fukaya category mirror to a point. Viewing this torus as an object of the \textit{affine} Fukaya category, it follows from homological symmetry at large radius that these objects correspond to an analytic chart on the mirror. It is then almost immediate from the results of~\cite{Han24a} that Lagrangian branes supported on $\tLvg$ of a suitable form (cf. Proposition~\ref{localm0}) are mirror to sheaves supported on a line. By~\cite[Corollary 6.1]{Han24a}, this sheaf has rank $2$ away from a Zariski open subset (of $\mathbb{CP}^1$). To check that the sheaf has rank $2$ at all points on the line, we will calculate the Floer differential of $\tLvg$ with an unobstruced local system to low order. Because $\tLvg$ bounds many disks of low energy, this turns out to be enough for us to obtain an upper bound on the ranks the stalks of the sheaf.

The rest of this section is organized as follows. In \S\ref{localsystems}, we will identify a subspace of the deformation space of $\tLvg$ in the Fukaya category of the quintic. We will then explain the calculations at large radius which determine the supports of the mirrors to these objects in \S\ref{supportsection}. Finally, \S\ref{localhfsection} completes the proof of Theorem~\ref{main1}. Combined with our results about the support of the mirror sheaf, the local computation of \S\ref{localhfsection} actually determines the full Floer homology of $\tLvg$.

\subsection{Local systems}\label{localsystems}
It is clear from the construction of $\tLvg$ that it bounds holomorphic teardrops, i.e. disks with a corner on a switching component. By equipping $\tLvg$ with suitable local systems, we can ensure that the algebraic counts of these teardrops vanish, and in fact that $\mathfrak{m}_0$ vanishes for all of these local systems. We will identify a $1$-dimensional space of unobstructed rank $1$ complex local systems. We do not attempt to compute the full deformation space of $\tLvg$ in this paper, though we conjecture that all other brane structures on $\tLvg$ give trivial objects of the Fukaya category. For simplicity, we will only consider certain local systems on $\tLvg$ which come from local systems on the minimally-twisted five-component chain link complement. Working within this restricted class of brane structures already suffices to recover partial $A$-model analogues of nontrivial results~\cite{COVV12} on the moduli spaces of lines in the mirror quintic.

An easy calculation with the Mayer--Vietoris sequence shows that $H_1(\Lvg;\mathbb{Z})$, the first homology of the quotient $\tLvg/(\mathbb{Z}/5)^3$, is a free $\mathbb{Z}$-module of rank $9$. Fix a basis for $H_1(\Lvg;\mathbb{Z})$ of the form
\begin{align}
\lbrace m_0,m_1,m_2,m_3,m_4\rbrace\cup\lbrace s_1,s_2,s_3,s_4\rbrace \label{singularh1basis}
\end{align}
where $m_0,\ldots,m_4$ denote the homology classes of the meridians in $L'$, as in Figure~\ref{loops}, and $s_1,\ldots,s_4$ are classes which do not lie in the span of the meridians. Explicitly, we can assume that $s_i$ is represented by a simple closed curve which intersects the zeroth boundary torus and the $i$th boundary torus of $L'$ transversely and once each.

Consider a local system, thought of as a representation $H_1(\Lvg;\mathbb{Z})\to\mathbb{C}^*$, of the form
\begin{align}\label{localsystemslvg}
\begin{cases}
m_i\mapsto\mu_i & i = 0,\ldots, 4 \\
s_j\mapsto 1 & j = 1,\ldots,4 \, .
\end{cases}
\end{align}
In this situation, we let $\lambda_i$ denote the image of the longitude $\ell_i$, for all $i = 0,\ldots,4$. Let $\nabla$ be a rank $1$ local system on $\tLvg$ induced by a local system on $\Lvg$ of the form~\eqref{localsystemslvg}.

The following lemma classifies all holomorphic teardrops on $\tLvg$, defined with respect to the \textit{integrable} almost complex structure. There are obvious families of teardrops bounded by $\tLvg$ which are the images of teardrops on the local immersed Lagrangian $L_Y\subset Y$. The teardrops in $Y$ arise as sections of $w\colon Y\to\mathbb{C}$ as in~\eqref{degeneration}. Observe that any holomorphic teardrop on $L_Y$ which contributes to $\mathfrak{m}_0$ must come in a family of virtual dimension $2$.
\begin{lemma}\label{localtears}
Any holomorphic teardrop on $L_Y$ which contributes to $\mathfrak{m}_0$ can be written as the product, with respect to~\eqref{commutative}, of the teardrop bounded by $\Larc$ with (cf. Figure~\ref{immersedarc}) with a holomorhpic disk in $\mathbb{C}^2$ bounded by the Clifford torus. The latter disk is either constant, or contained in one of the coordinate axes of $\mathbb{C}^2$.
\end{lemma}
\begin{proof}
Since $L_Y$ is a product Lagrangian in $\mathbb{C}\times\mathbb{C}^2$, all holomorphic disks with boundary on $L_Y$ split. For any such disk to have index $2$, its second factor must be one of the disks listed in the statement of the Lemma.
\end{proof}
All disks described in Lemma~\ref{localtears} are regular, by the argument of~\cite[Lemma 2.10]{AS21}, and each moduli space of teardrops can naturally be identified with a $2$-torus.
\begin{lemma}
Any holomorphic teardrop in $\tWvg$ bounded by $\tLvg$ is one of the teardrops of Lemma~\ref{localtears}.
\end{lemma}
\begin{proof}
Let $u\colon\overline{\Delta}\setminus\lbrace 1\rbrace$ be such a teardrop. The boundary of $u$ would have to lift to a path on (the domain of) $\tLvg$ which passes between the two branches. Thus, the image of $u$ would have to intersect one of the neighborhoods $\widetilde{B}_{\ell,m}$. There are holomorphic maps $w_{\ell,m}\colon\widetilde{B}_{\ell,m}\to\mathbb{C}$ induced from~\eqref{degeneration}. By the open mapping theorem, the image of $w_{\ell,m}\circ u$ must be an open subset of $\mathbb{C}$. Since the image of $\tLsing$ under this map is an arc with a self-intersection point, it follows that the image of $u$ must be contained in $\widetilde{B}_{\ell,m}$.
\end{proof}
By reordering the coordinates on the domains of the Darboux charts $\widetilde{B}_{\ell,m}$ if necessary, cf. Remark~\ref{orderingcoordinates}, we can calculate $\mathfrak{m}_0$ for $\tLvg$ to low order.
\begin{proposition}\label{localm0}
Let $\nabla$ be a rank one $\mathbb{C}$-local system on $\tLvg$ which is induced from a local system on $\Lvg$ of the form~\eqref{localsystemslvg}. Then $(\tLvg,\nabla)$, equipped with an appropriate choice of spin structure, is unobstructed with bounding cochain $0\in CF^*(\tLvg)$ in $\tWvg$ if the holonomies $\lbrace\mu_i,\lambda_i\rbrace_{i=0}^4$ satisfy
\begin{align}
1+\mu_0^{-1}+\mu_0^{-1}\lambda_0^{-1} &= 0 \label{interiorconeobs}\\
-1-\mu_1^5+\lambda_1^{-5} &= 0 \label{1obs}\\
-1-\mu_2^{-5}-\lambda_2^5 &= 0 \label{2obs}\\
-1-\mu_3^{-5}+\lambda_3^5 &= 0 \label{3obs}\\
-1+\mu_4^5+\mu_4^5\lambda_4^5 &= 0 \, . \label{4obs}
\end{align}
There is a subspace of the space of local systems satisfying~\eqref{interiorconeobs}--\eqref{4obs} and containing the local system $\nabla^{\vG}_{\omega}$, to be described in Example~\ref{vgexample} below which can be identified with a punctured genus $6$ curve.
\end{proposition}
\begin{definition}\label{locallyunobstructed}
A local system on $\tLvg$ satisfying the conclusion of Proposition~\ref{localm0} is said to be \textit{locally unobstructed}.
\end{definition}
Before proving Proposition~\ref{localm0}, we will check that the set of locally unobstructed local systems is nonempty by exhibiting local systems mirror to van Geemen lines.
\begin{example}\label{vgexample}[Very affine van Geemen lines]
Let $\omega,a\in\mathbb{C}^*$ be constants satisfying
\begin{align*}
1+\omega+\omega^2 &= 0 \\
a^5  &= 27 \, .
\end{align*}
Consider the local system on $\Lvg$ with holonomy
\begin{equation}\label{vangeemenlocalsystem}
\begin{aligned}[c]
\mu_0 &= \omega \, ; \\
\mu_1 &= -\omega \, ; \\
\mu_2^{-1} &= \frac{a}{3}(1-\omega^2) \, ; \\
\mu_3^{-1} &= -\omega \, ; \\
\mu_4^{-1} &= -\omega^2 \, ;
\end{aligned}
\qquad
\begin{aligned}[c]
\lambda_0 &= \omega \\
\lambda_1^{-1} &= -\frac{a}{3}(1-\omega) \\
\lambda_2 &= -\omega^2 \\
\lambda_3 &= -\frac{a}{3}(1-\omega) \\
\lambda_4 &= -\omega \, .
\end{aligned}
\end{equation}
Since there are two primitive third roots of unity~\eqref{vangeemenlocalsystem} determines two local systems, which we refer to interchangeably. It is easy to check that these satisfy~\eqref{interiorconeobs}--\eqref{4obs}. Thus this local system lifts to one on $\tLimm$ via the covering map $\tLimm\to\Limm$. By~\eqref{localsystemslvg} this yields a local system on $\tLvg$ as well, which is the local system $\nabla^{\vG}_{\omega}$ mentioned in the statement of Proposition~\ref{localm0}. The proof of~\cite[Theorem 1.2]{Han24a} shows that equipping $\tLimm\subset T^*T^3$ with a local system with holonomy as in~\eqref{vangeemenlocalsystem} gives an object of the wrapped Fukaya category of $T^*T^3$ which is mirror to a coherent sheaf supposed on the line $(\mathbb{C}^*)^3$ cut out by
\begin{align*}
\lambda_{3}u_2-\mu_{2}^{-1}u_3+1 &= 0 \\
-\mu_{2}^{-1}\lambda_{2}^{-1}u_1+\mu_{1}^{-1}\lambda_{1}^{-1}u_2+1 &= 0\\
\mu_{3}\lambda_{3}u_1-\lambda_{1}^{-1}u_3+1 &= 0 \, .
\end{align*}
In this example, these linear forms specialize to
\begin{align*}
-\frac{a}{3}(1-\omega)u_2-\frac{a}{3}(1-\omega^2)u_3+1 &= 0 \\
-\frac{a}{3}(1-\omega)u_1-\frac{a}{3}(1-\omega^2)u_2+1 &= 0 \\
-\frac{a}{3}(1-\omega^2)u_1-\frac{a}{3}(1-\omega)u_3+1 &= 0 \, . 
\end{align*}
The projective closure of this line in $\mathbb{CP}^3$ is easily seen to be the limiting cycle, in $X_0^5$, of a van Geemen line $C_0^{\omega}$. By interchanging $\omega$ with $\omega^2$, we obtain another local system $\nabla^{\vG}_{\omega^2}$, to which we can associate a curve in $(\mathbb{C}^*)^3$ whose projective closure in $X_0^5$ is $C_0^{\omega^2}$.
\end{example}
\begin{proof}[Proof of Proposition~\ref{localm0}]
All three families of teardrops contained in $\widetilde{B}_{\ell,m}$, as in Lemma~\ref{localtears}, have the same output in the de Rham complex $\Omega^*(T^2)$ of the switching component. The definitions of the balls $\widetilde{B}_{\ell,m}$ and Lemma~\ref{inducedmap} imply that the contributions of these teardrops to $\mathfrak{m}_0$ weighted by holonomies given in~\eqref{interiorconeobs}--\eqref{4obs}, up to signs. 

We claim that there is a spin structure on $\tLvg$ with respect to which these teardrops contribute with the signs given in the statement of the Proposition. We will describe this as the lift of a spin structure on the quotient $\Lvg$. As previously discussed, the spin structure we choose for $\Lvg$ should restrict to the spin structure on $\Limm$ specified in~\cite[Lemma 4.2]{Han24a}. This determines the signs appearing in~\eqref{interiorconeobs} (cf.~\cite[Lemma 5.7]{Han24a}).

In~\eqref{4obs}, the terms with holonomy $\pm\mu_4^5$ and $\pm\mu_4\lambda_4^5$ must carry the same sign, since the spin structure on the minimally-twisted five-component chain link complement of~\cite[Lemma 4.2]{Han24a} is preserved by a hyperbolic isometry which interchanges the homology classes $m_4$ and $m_4+\lambda_4$. We can then choose the spin structure on $\Lvg$ such that the remaining disk, which we can assume has boundary lying on a curve representing the class $s_4\in H_1(\Lvg)$, contributes with the opposite sign. Again by~\cite[Lemma 4.2]{Han24a}, the terms of~\eqref{1obs},~\eqref{2obs}, and~\eqref{3obs} weighted by nontrivial holonomy contribute with opposite signs. Because $H_1(\Lvg)$ has rank $9$, the spin structure can be chosen so that it has the desired behavior at all four necks of $\Lvg$.

In particular the first term in each expression corresponds to the product of the teardrop bounded by $L_Y$ with a constant disk. The signs are determined using the symmetries of the spin structure on $L$ (cf. Section 6.3 and Remark 6.3 of~\cite{Han24a}).

It is clear that the subvariety of $(\mathbb{C}^*)^5$ cut out by~\eqref{interiorconeobs}-\eqref{4obs} is at most $2$-dimensional, since the value of $\mu_0$ (which determines the value of $\lambda_0$) and the value of $\mu_i^5$ or $\lambda_i^5$, for any $i = 1,2,3,4$, will uniquely determine a local system satisfying these relations using Lemma~\ref{relations}. Moreover, we calculate
\begin{align*}
1+\mu_1^5 = \lambda_1^{-5} &= \mu_0^5\mu_2^{-5} \\
&= \mu_0^5\lambda_3^5\mu_4^{-5} \\
&= \mu_0^5(1+\mu_3^{-5})\mu_4^{-5} \\
&= \mu_0^5(1+\mu_0^{-5}\lambda_4^{-5})\mu_4^{-5}.
\end{align*}
Multiplying by $\mu_4^5$ gives us
\begin{align*}
\mu_4^5+\mu_1^5\mu_4^5 &= \mu_0^5(1+\mu_0^{-5}\lambda_4^{-5}) \\
\mu_4^5+\lambda_0^{-5} &= \mu_0^5+\lambda_4^{-5} \,
\end{align*}
which we rewrite using~\eqref{interiorconeobs} and~\eqref{4obs} as
\begin{align*}
\frac{1}{1+\lambda_4^5}-\lambda_5^4 = \mu_0^5-\lambda_0^{-5} \, .
\end{align*}
This implies that the values of $\mu_0^5$ and $\lambda_0^5$ determine a value of $\lambda_4^5$ satisfying~\eqref{interiorconeobs}--\eqref{4obs}, and thus that the space of unobstructed local systems is $1$-dimensional (as we have already seen that it is nonempty). For values of $\mu_0$ and $\lambda_0$, contained outside of a finite subset of the pair of pants in $(\mathbb{C}^*)^2$ cut out by~\eqref{interiorconeobs}, the value of $\lambda_4^5$ we obtain will be nonzero. It follows that the space of local systems satisfying~\eqref{interiorconeobs}--\eqref{4obs} is naturally a Zariski open subset of a $25$-fold cover the standard pair of pants in $(\mathbb{C}^*)^2$. An application of the Riemann--Hurwitz formula shows that this cover has genus $6$.
\end{proof}
\begin{remark}\label{orderingcoordinates2}
As pointed out in Remark~\ref{orderingcoordinates}, the intersection of $\tLvg$ with the balls $\widetilde{B}_{\ell,m}$ depends on an ordering of coordinates on the domains of these charts. One observes, however, that reordering coordinates in the domains corresponds to scaling~\eqref{interiorconeobs}--\eqref{4obs} by nonzero scalars, so any choices of coordinates used to construct $L_Y$ give us Lagrangians with \textit{canonically} isomorphic deformation spaces.
\end{remark}
\begin{remark}
The appearance of a genus six curve is consistent with the results of~\cite{COVV12}, though its description here is very different from the genus six curves of op. cit. One expects the mirror map, which by~\cite{GPS15} induces the change of variables underlying the mirror functor, to embed the punctured curve of Proposition~\ref{localm0} into the moduli space of lines on the mirror quintic in a nontrivial way. Some of the punctures in the curve of Proposition~\ref{localm0}, in particular the punctures corresponding to lifts of punctures in the pair of pants~\eqref{interiorconeobs}, should, informally, correspond to taking limits of local systems on $\tLvg$ such that the holonomy around certain loops approaches $0$ or $\infty$. This suggests that $A$-branes corresponding to these punctures in the moduli space of lines should be supported on different Lagrangians.
\end{remark}

\subsection{Unobstructedness}\label{unobstructedness} Recall from~\cite{Fuk17}, or Appendix~\ref{immersedfloerappendix}, that the obstruction term $\mathfrak{m}_0$ can be thought of as a cochain in $CF^2(L)$, for any clean Lagrangian immersion $\iota\colon L\to M$. Moreover, we can represent it as a sum of elements 
\begin{align*}
\mathfrak{m}_0^0+\mathfrak{m}_0^s\in\left((\Omega^2(L))\widehat{\otimes}_{\mathbb{C}}\Lambda_0\right)\oplus\bigoplus_{\substack{a\in A\setminus\lbrace 0\rbrace \\ k-\deg(L_a) = 2}}\left((\Omega^k(L_a))\widehat{\otimes}_{\mathbb{C}}\Lambda_0\right)
\end{align*}
where $\Lambda_0$ is the Novikov ring~\eqref{novikovring} over $\mathbb{C}$ and $\widehat{\otimes}_{\mathbb{C}}$ denotes the completed tensor product. More succinctly, we split $\mathfrak{m}_0$ as a sum of elements in summands of the Floer cohcain space corresponding to the diagonal and switching components of the fiber product $L\times_{\iota}L$.

In the case of $\tLvg\to X_{\tau}^5$, we will show that these two summands of the curvature both vanish. The vanishing of $\mathfrak{m}_0^s$ in this setting follows from the definition of a locally unobstructed local system (Definition~\ref{locallyunobstructed}) combined with an SFT compactness theorem for disks with Lagrangian boundary appearing in~\cite{Cha23}.
\begin{proposition}\label{m0svanishing}
The switching part $\mathfrak{m}_0^{s}$ of the obstruction class of $\tLvg$ vanishes.
\end{proposition}
\begin{proof}
Fix a class $\beta\in H_2(X^5_{\tau},\tLvg;\mathbb{Z})$. Consider the moduli space $\mathcal{M}_1(\tLvg;\beta)$ of holomorphic teardrops, i.e. holomorphic disks with one boundary marked point asymptotic to a switching component, and let $[u]$ be an element of this moduli space. For topological reasons, the boundary of $u$ is an arc which must pass through one of the necks in $\tLvg$, and thus the image of the interior of $u$ intersects one of the balls $\widetilde{B}_{\ell,m}$, for $\ell = 0,1,2,3,4$.

The boundary $\widetilde{S}_{\ell,m} = \partial\widetilde{B}_{\ell,m}$ is a hypersurface of contact type which divides $X_{\tau}^5$ into two connected components. In particular, there is a tubular neighborhood $N\cong(-\epsilon,\epsilon)\times\widetilde{S}_{\ell,m}$ and a contact form $\alpha$ on $\widetilde{S}_{\ell,m}$ such that the symplectic form on $X_{\tau}^5$ restricts to $d(e^{kt}\alpha)$ on $N$, where $t\in(-\epsilon,\epsilon)$ and $k$ is a positive integer. We can take $\alpha$ to be the standard contact form on $S^5$. By construction, we can write $\tLvg\cap N = \Lambda\times(-\epsilon,\epsilon)$ for a Legendrian $\Lambda\subset S^5$.

We will apply an SFT compactness theorem  for disks with Lagrangian boundary, specifically~\cite[Theorem 3.13]{Cha23}, in this situation. The discussion around~\cite[Example 3.1]{Cha23} shows that the SFT compactness theorem applies in the setting described in the last paragraph. Neck stretching along $\widetilde{S}_{\ell,m}$ shows that given any teardrop $u$ as in the first paragraph, there is a holomorphic teardrop in $\widetilde{B}_{\ell,m}$ with interior punctures asymptotically converging to Reeb orbits in the standard contact $S^5$. Moreover the boundary of $u$ has the same homology class in $H_1(\tLvg)$ as the boundary of this new teardrop. 

We will think of this punctured teardrop as a curve contained in the Gross--Siebert space $Y$ of~\eqref{gssing}. This can be thought of as a meromorphic function $\overline{u}\colon\Delta\to Y$. We can replace this with a holomorphic teardrop as follows. Using the product decomposition of~\eqref{commutative}, we can write $\overline{u} = (\overline{u}_i,\overline{u}_j,\overline{w})$. Each of these component functions is meromorphic. The singular points of these functions are all (at worst) poles at points in $\Int\Delta$ corresponding to the punctures. The punctures cannot correspond to essential singularities, since the condition that they are asymptotic to Reeb orbits would contradict Picard's great theorem, for instance. If $P_i$ denotes the (finite) set of poles of $\overline{u}_i$, we can form a holomorphic function by setting
\begin{align*}
\widetilde{u}_i(z)\coloneqq\overline{u}_i(z)\cdot\prod_{p\in P_i}(z-p)^{\ord(p)} \,.
\end{align*}
This is a holomorphic disk with boundary on simple closed curve in $\mathbb{C}$ enclosing the origin. We also define $\widetilde{u}_j$ and $\widetilde{w}$ analogously. Note that $\widetilde{u}_j$ and $\widetilde{w}$ will have boundary on $1$-dimensional submanifolds of $\mathbb{C}$ that are isotopic to the unit circle or to $\Larc$, respectively. Thus we obtain a holomorphic disk $\widetilde{u}\coloneqq(\widetilde{u}_i,\widetilde{u}_j,\widetilde{w})$ with boundary on a different Lagrangian submanifold, denoted $L_Y'$, contained in $Y$. Although $L_Y'$ is only Lagrangian isotopic to $L_Y$, it is clear that teardrops bounded by $L_Y'$ can be classified in the same way as teardrops on $L_Y$. In particular, $\widetilde{u}$ is a section over the teardrop $\widetilde{w}$ per the proof of Lemma~\ref{localtears}. From this we can determine the homology class of $\partial\widetilde{u}$ in $H_1(L_Y')$, and in turn the homology class of $\partial\overline{u}\in H_1(L_Y)$. The latter must be the homology class of the boundary of a disk described in Lemma~\ref{localtears}, as is the case for the boundary $\partial u$ of the original teardrop as well.

From the teardrop $u$, we can form two other teardrops with boundary on $\tLvg$ by cyclically permuting the coordinates $(x_i,x_j,x_k)$ on $\mathbb{CP}^4$, where we recall our convention that $\lbrace i,j,k,\ell\rbrace = \lbrace 1,2,3,4\rbrace$. By the result of the previous paragraph, we can determine the first homology classes of the boundaries of each of these teardrops. Since we have equipped $\tLvg$ with a locally unobstructed local system, it follows that the contributions of these three teardrops to $\mathfrak{m}_0^{s}$ cancel.
\end{proof}
\begin{remark}
The technique of treating punctured curves arising from neck stretching as meromorphic functions with poles is inspired by the proof of~\cite[Theorem 6.27]{Cha23}, and our arguments can be thought of as a weaker version of those in~\cite[\S{6.3}]{Cha23}.
\end{remark}

To show that diagonal obstruction term $\mathfrak{m}_0^0$ vanishes, we will adapt the results of~\cite{Sol20}. The results of op. cit. apply to an embedded Lagrangian $L\subset M$ fixed, as a set, by an anti-symplectic involution $\mathfrak{d}$. Under additional assumptions about the action of $\mathfrak{d}$ on $H^*(L)$, \cite[Theorem 1.2]{Sol20} shows that $CF^*(L)$ is unobstructed (and in fact that it is formal), by calculating the sign of $\mathfrak{d}$ as it acts on the moduli spaces $\mathcal{M}_{k+1}(L)$. These arguments to apply in our setting show that $\mathfrak{m}_0^0$ for $\tLvg\to X_{\tau}^5$ vanishes with essentially no changes, since the diagonal obstruction term only counts disks with smooth boundary. We only need to explain how a canonical model for the Fukaya $A_{\infty}$-algebra of an immersed Lagrangian is constructed, and to check that complex conjugation (which acts on $X_{\tau}^5$ for real $\tau$) acts suitably on the de Rham cohomology of $H^*(\tLvg)$.

\begin{lemma}
The involution $\mathfrak{d}$ preserves $\tLvg$ as a set.
\end{lemma}
\begin{proof}
First we will show that the singular Lagrangian $\tLvgsing$ is $\mathfrak{d}$-invariant. Recall from~\cite{Han24a} that the tropical Lagrangian $\Lsing$ is fixed by complex conjugation on $(\mathbb{C}^*)^3$, which just acts by the inverse map on the $T^3$-fibers. Note that complex conjugation also restricts to the inverse map on the $2$-torus fibers of the periodized conormal bundles. There is a corresponding involution on $T^*T^3$. Thus by Remark~\ref{125foldcover}, the cover $\tLsing$ is also $\mathfrak{d}$-invariant.

All of the Lagrangian tori $L_{\tau,q}\subset X_{\tau}^5$ is preserved by $\mathfrak{d}$, and its action on $\Wein(L_{\tau,0})$ can be identified with the action by complex conjugation on $T^*T^3$. Similarly, the Lagrangian cones near $X_{\tau}^5$ used to construct $\tLvgsing$ are also preserved by $\mathfrak{d}$, and its action on the link of each cone is easily seen to coincide with the action of conjugation on the ends of the tropical Lagrangians. Therefore $\mathfrak{d}$ respects all of the gluings carried out in the proof of Theorem~\ref{singularlagrquintic}, meaning that it acts on $\tLvgsing$. In fact, the action of $\mathfrak{d}$ restricts to an action of the smooth locus $\widetilde{L}'$ as well.

The Morse function $\widetilde{h}\colon\widetilde{L}'\to\mathbb{R}$ can be made $\mathfrak{d}$-invariant, since, by construction, the action by complex conjugation on $\Lsing\subset T^*T^3$ swaps the two ideal cubes drawn in Figure~\ref{idealtriangulation}. This means that the graphs $\Gamma(d\widetilde{h})$ and $\Gamma(-d\widetilde{h})$ contained in the Weinstein neighborhood $\Phi_{\widetilde{L}'}(U_{\widetilde{L}'})$ are swapped by the action of $\mathfrak{d}$. Correspondingly, the action of complex conjugation the neighborhood $Y$ defined in~\eqref{gssing} descends to complex conjugation on $\mathbb{C}$ under the map $Y\to\mathbb{C}$ appearing in~\eqref{commutative}. Since the ends of $L_Y$ are glued to the two branches $\Gamma(d\widetilde{h})$ and $\Gamma(-d\widetilde{h})$, it follows that the action of $\mathfrak{d}$ respects this gluing.
\end{proof}
Consider a holomorphic disk $u\colon(D^2,\partial D^2)\to(X_{\tau}^5,\tLvg)$ in the homology class $\beta\in H_2(X_{\tau}^5,\tLvg;\mathbb{Z})$. There is a map on the corresponding moduli space
\begin{align}
\widetilde{\mathfrak{d}}\colon\mathcal{M}_1(\tLvg;\beta)\to\mathcal{M}_1(\tLvg;-\delta_{*}\beta) \label{mapmoduli}
\end{align}
induced by $\mathfrak{d}$ and defined as follows. Recall that an element of $\mathcal{M}_1(\tLvg;\beta)$ is represented by a holomorphic map $u\colon(\Sigma,j)\to(X_{\tau}^5,J)$ Let $(\Sigma,j)$ from a (nodal) bordered Riemann surface $(\Sigma,j)$ for which $[u] = \beta\in H_2(X_{\tau}^5,\tLvg)$ and a boundary marked point $z_0\in\partial\Sigma$. Let $\psi_{\Sigma}\colon\overline{\Sigma}\to\Sigma$ denote the antiholomorphic involution whose underlying map of sets is the identity. Then the map of moduli spaces~\eqref{mapmoduli} is defined by
\begin{align}
\widetilde{\mathfrak{d}}(\Sigma,u,z_0)\coloneqq(\overline{\Sigma},\delta\circ u\circ\psi_{\Sigma},z_0) \, .
\end{align}
The proof of~\cite[Proposition 5.6]{Sol20} extends to the case of Lagrangian immersions with clean self-intersection to show the following.
\begin{lemma}\label{signm0diag}
The sign of~\eqref{mapmoduli} is given by
\begin{align}
sgn(\widetilde{\mathfrak{d}}) = sgn(\mathfrak{d}\vert_{\tLvg})+1 \,  \label{signmodulimap}
\end{align}
provided that $\beta\in H_2(X_{\tau}^5,\tLvg;\mathbb{Z})$ is a relative homology class as described above.
\qed
\end{lemma}
Since the irreducible locus of any such moduli space consits only of maps whose domain is a disk without corners, the orientation calculation from the proof of~\cite[Proposition 5.6]{Sol20} is unaffected. Because the boundary strata of the moduli spaces of disks with boundary on $\tLvg$ are oriented coherently, the sign of~\eqref{mapmoduli} is the same on the boundary strata.

Note that because $\tLvg$ is graded, the Maslov class evaluated at any $\beta\in H_2(X_{\tau}^5;\tLvg)$ as above vanishes, explaining the relative simplicity of the formula in~\eqref{signmodulimap}.

A canonical model for the Lagrangian Floer cochain complex can be constructed as described in~\cite{FOOOI} or~\cite{Sol20}. Consider the $\mathbb{C}$-vector space $\overline{C}^*\coloneqq\overline{CF}^*(\tLvg)$ of~\eqref{underlyingvs}. Let $\overline{D}^*$ denote the finite-dimensional $\mathbb{C}$-vector space obtained by taking the cohomology of $\overline{C}^*$ with respect to the de Rham differential, denote $\mathfrak{m}_{1,0}$. By the Hodge decomposition theorem, there exist linear maps
\begin{align*}
i\colon\overline{D}\to\overline{C} \\
p\colon\overline{C}\to\overline{D} \\
h\colon\overline{C}\to\overline{C}
\end{align*}
satisfying
\begin{align*}
p\circ\mathfrak{m}_{1,0} = 0\, &, \qquad \mathfrak{m}_{1,0}\circ i = 0 \, , \\
p\circ i = \id \, &, \qquad \mathfrak{m}_{1,0}\circ h + h\circ\mathfrak{m}_{1,0} = i\circ p-\id \, .
\end{align*}
The existence of such maps implies that there is a gapped filtered $A_{\infty}$-structure on the $\mathfrak{m}_{1,0}$-cohomology $H^*(CF^*(\tLvg),\mathfrak{m}_{1,0})$ of $CF^*(\tLvg)$ which is quasi-isomorphic to $CF^*(\tLvg)$. Note that the decomposition of $\mathfrak{m}_0$ into $\mathfrak{m}_0^{0}+\mathfrak{m}_0^{s}$ still makes sense  in $H^*(CF^*(\tLvg),\mathfrak{m}_{1,0})$, since $\mathfrak{m}_{1,0}$ respects the splitting of $CF^*(\tLvg)$.
\begin{proposition}
We have that $\mathfrak{m}_0^{0}\in H^2(CF^*(\tLvg),\mathfrak{m}_{1,0})$ vanishes, implying that $\mathfrak{m}_0^{0}\in CF^*(\tLvg)$ vanishes as well.
\end{proposition}
\begin{proof}
Consider the classes $m_i,s_j\in H_1(\Lvg)$ described in~\eqref{singularh1basis}, and notice that they lift to a set of generators for $H_1(\tLvg)$. By examining the action of complex conjugation on these classes in $H_1(\Lvg)$, it is easy to see that $\mathfrak{d}$ induces $-\id$ on $H_1(\tLvg)$. From this one can also see that $\mathfrak{d}$ induces the identity on $H_2(\tLvg)$, since $H_2(\tLvg)$ is generated by products of the lifts of the classes in~\eqref{singularh1basis} and the lifts of the longitudes.  Dualizing, we see that $\mathfrak{d}$ induces the identity on the de Rham cohomology $H^2(\tLvg)$ as well. Thus by Lemma~\ref{signm0diag} we  conclude that $\mathfrak{m}_0^0 = \mathfrak{d}^*\mathfrak{m}_0^0 = -\mathfrak{m}_0^0$, meaning that $\mathfrak{m}_0^0 = 0$.
\end{proof}
\begin{corollary}
For any locally unobstructed local system $\nabla$, the Lagrangian brane $(\tLvg,\nabla)$ is unobstructed with bounding cochain zero. \qed
\end{corollary}
This corollary says that locally unobstructed local systems are unobstructed. We will refer to them as such hereafter. Observe that the gappedness of the $A_{\infty}$-algebra implies that an unobstructed local system is locally unobstructed.

\subsection{Supports of mirror sheaves}\label{supportsection}
Since we have checked that $(\tLvg,\nabla)$ is unobstructed for suitable choices of $\nabla$, we can take these to be objects of the relative Fukaya category $\mathcal{F}(X_{\tau}^5,D_{\tau})$ by Lemma~\ref{exactness}, so by Assumption~\ref{generationass}, we can consider the image of $(\tLvg,\nabla)$ under the mirror functor of~\cite{She15}, which we denote by $\mathcal{L}_{(\tLvg,\nabla)}\in D^b_{dg}\Coh(X^{5,\vee})$. The support of this complex of coherent sheaves can be determined by calculating the Floer cohomology of $(\tLvg,\nabla)$ with objects of $\mathcal{F}(X_{\tau}^5,D_{\tau})$ supported on the Lagrangian torus $L_{\tau,0}$. By Proposition~\ref{mirrors2points}, equipping $L_{\tau,0}$ with any $\mathbb{C}$-local system gives a Lagrangian brane mirror to a point. The stalk of $\mathcal{L}_{(\tLvg,\nabla)}$ at such a point can be identified with Floer cohomology group
\begin{align}
HF^0((\tLvg,\nabla),(L_{\tau,0},\nabla_p)) \label{mirrorstalk}
\end{align}
under homological mirror symmetry.

We can reduce to computation of the Floer cohomology groups~\eqref{mirrorstalk} in the quintic $X^5_{\tau}$, to a computation in the very affine quintic $X^5_{\tau}\setminus D_{\tau}$ using the open mapping theorem.
\begin{lemma}\label{stripclassification}
Suppose that $u\colon\mathbb{R}\times[0,1]\to X_{\tau}^5$ is a holomorphic strip which contributes to the Floer differential on $CF^*((L_{\tau,0},\nabla_p),(\tLvg,\nabla))$. Then the image of $u$ does not intersect the divisor $D_{\tau}$.
\end{lemma}
\begin{proof}
Suppose that $u\colon\mathbb{R}\times[0,1]\to X_{\tau}^5$ is a holomorphic strip whose image intersects $D_{\tau}$. This implies that the boundary component of $u$ mapped to $\tLvg$ must pass through one of the necks of $\tLvg$, i.e. through a copy of $L_Y$ in one of the charts $\widetilde{B}_{\ell,m}$. As in the proof of Proposition~\ref{m0svanishing}, we can neck stretch along $\partial\widetilde{B}_{\ell,m}$ to produce a holomorphic strip $u'$ with interior punctures. Since all points in $\tLvg\cap L_{\tau,0}$ lie away from $\widetilde{B}_{\ell,m}$ it follows that the boundary arc of $u'$ must exit the ball. We can rescale $u'$, using the same argument in the proof of Proposition~\ref{m0svanishing}, to produce a holomorphic strip on $\widetilde{u}$ with a boundary component on $\tLvg$ and a boundary component on $L_{\tau,0}$. This strip has the property that its image intersects $\partial\widetilde{B}_{\ell,m}$ in an arc.

If we consider the restriction $\widetilde{u}'$ of $\widetilde{u}$ to $\widetilde{u}^{-1}(\widetilde{B}_{\ell,m})$, then the composition of this map with $w\colon\widetilde{B}_{\ell,m}\to\mathbb{C}$, which is the projection defined in~\eqref{degeneration}, is holomorphic. Since $\widetilde{u}$ is not contained in $\widetilde{B}_{\ell,m}$, its holomorphicity, and in turn the holomorphicity of $u$, contradicts the open mapping principle.
\end{proof}
Exactness of $L_{\tau,0}$ and $\tLvg$ allows us to control the areas of disks contributing to the Floer differentials on $CF^*((L_{\tau,0}),(\tLvg,\nabla))$, thereby reducing the calculation of Floer cohomology to~\cite[Theorem 1.2]{Han24a}.
\begin{lemma}\label{supportstrips}
Suppose that $u\colon\mathbb{R}\times[0,1]\to X_{\tau}^5\setminus D_{\tau}$ is a holomorphic strip which contributes to the Floer differential on $CF^*((L_{\tau,0},\nabla_p),(\tLvg,\nabla))$. Then the image of $u$ is contained in a Weinstein neighborhood $\Wein(L_{\tau,0})$.
\end{lemma}
\begin{proof}
Since $\tLvg$ and $L_{\tau,0}$ are both exact in $X_{\tau}^5\setminus D_{\tau}$ the area of any strip is determined by its two corners. Thus, rescaling the $1$-form used to construct $\tLsing$ if necessary, we can bound the area of any strip contributing to the Floer differential above by an arbitrarily small constant. Consequently, the image of $u$ cannot exit $\Wein(L_{\tau,0})$.
\end{proof}
Hereafter, we will only consider local systems on $L_{\tau,0}$ which are pulled back from local systems on $T^3$ under the $125$-fold cover $L_{\tau,0}\to L_{\tau,0}$ with deck group $(\mathbb{Z}/5)^3$.
\begin{corollary}
There is an isomorphism
\begin{align}
HF^*((L_{\tau,0},\nabla_p),(\tLvg,\nabla))\cong HF^*((T^3,\nabla_p),(\tLimm,\nabla)) \label{localsupport}
\end{align}
of $\mathbb{C}$-vector spaces. The Floer group on the left hand side is computed in $X_{\tau}^5\setminus D_{\tau}$, whereas the Floer group on the right hand side is computed in $T^*T^3$, where $T^3$ is thought of as the $0$-section.
\end{corollary}
\begin{proof}
By Lemma~\ref{supportstrips}, any holomorphic strip which contributes to the Floer differential on $CF^*((L_{\tau,0},\nabla_p),(\tLvg,\nabla))$ can also naturally be thought of as a disk in $T^*T^3$ contributing to the Floer differential $CF^*((T^3,\nabla_p),(\tLimm,\nabla))$. The local system $\nabla$ on $\tLimm$ is induced from the local system $\nabla$ on $\tLvg$ in the obvious way.
\end{proof}
The main results of~\cite{Han24a} as written concern the support of the Lagrangian immersion $\Limm$, rather than that of $\tLimm$. We will relate the groups $HF^*((T^3,\nabla_p),(\Limm,\nabla))$ to the groups on the right hand side of~\eqref{localsupport} using a covering argument.

By the construction of $\tLimm$ and our choices of local systems $\nabla$ on $\tLvg$ and $\nabla_p$ on $L_{\tau,0}$, it follows that $(\mathbb{Z}/5)^3$ acts on the Floer cochain space $CF^*((T^3,\nabla_p),(\tLimm,\nabla))$, and thus on homology $HF^*((T^3,\nabla_p),(\tLimm,\nabla))$.
\begin{lemma}\label{trivialactionsupp}
The action of $(\mathbb{Z}/5)^3$ on $HF^*((T^3,\nabla_p),(\tLimm,\nabla))$ is trivial, i.e. the space of invariants can be written as
\begin{align}
HF^*((T^3,\nabla_p),(\tLimm,\nabla))^{(\mathbb{Z}/5)^3} = HF^*((T^3,\nabla_p),(\tLimm,\nabla)) \, .
\end{align}
\end{lemma}
\begin{proof}
The action of $(\mathbb{Z}/5)^3$ on Floer cohomology is induced by an action of $(\mathbb{Z}/5)^3$ on $T^*T^3$ by symplectomorphisms. Specifically, this group acts by rotations in the $T^3$-direction, which are Hamiltonian isotopies. Hence the action on homology is trivial.
\end{proof}
On the other hand, we have that the group of $(\mathbb{Z}/5)^3$-invariants can be written as
\begin{align*}
HF^*((T^3,\nabla_p),(\tLimm,\nabla))^{(\mathbb{Z}/5)^3} = H^*(CF^*((T^3,\nabla_p),(\tLimm,\nabla))^{(\mathbb{Z}/5)^3},\mathfrak{m}_1) \, ,
\end{align*}
the cohomology of the subcomplex of $(\mathbb{Z}/5)^3$-invariant chains. The following is immediate from the $(\mathbb{Z}/5)^3$-equivariance of the chain complex.
\begin{lemma}\label{supportinvariants}
The cohomology group $H^*(CF^*((T^3,\nabla_p),(\tLimm,\nabla))^{(\mathbb{Z}/5)^3},\mathfrak{m}_1)$ can be identified with $HF^*((T^3,\overline{\nabla}_p),(\Limm,\overline{\nabla}))$, where $\overline{\nabla}_p$ and $\overline{\nabla}$ denote the local systems on $T^3$ and $\Limm$ from which the local systems $\nabla_p$ and $\nabla$ are induced. Consequently
\begin{align*}
HF^*((T^3,\nabla_p),(\tLimm,\nabla))\cong HF^*((T^3,\overline{\nabla}_p),(\Limm,\overline{\nabla})) \, .
\end{align*}
\qed
\end{lemma}
\begin{remark}
One could also compute the Floer-theoretic support of $\tLimm$ in $T^*T^3$ directly by considering its geometric compositions with the Lagrangian correspondences discussed in Remark~\ref{differentmorse}. Instead of reducing the computation of $HF^*((T^3,\nabla_p),(\tLimm,\nabla))$ to a computation of the Floer-theoretic support (cf.~\cite{Han24a}) of a tropical pair of pants, one would instead need to consider the Floer-theoretic support of the Lagrangian lift of a non-smooth, trivalent tropical curve quintic plane curve. In this case, one can see directly that the nonvanishing Floer cohomology groups have rank $2$, in accordance with the results of Lemma~\ref{supportinvariants} and~\cite[Corollary 6.1]{Han24a}. We chose to describe a more abstract strategy of proof because we will also make use of it when we discuss the self Floer cohomology of $(\tLvg,\nabla)$.
\end{remark}
It follows from Lemma~\ref{supportinvariants} and~\cite[Theorem 1.2]{Han24a} that the mirror sheaf to $(\tLimm,\nabla)$ in $(\mathbb{C}^*)^3/(\mathbb{Z}/5)^3$ is supported on a rational curve with four punctures. The analogous result for the mirror sheaf to $(\tLvg,\nabla)$ follows almost immediately.
\begin{proposition}\label{globalsupport}
The mirror sheaf to $(\tLvg,\nabla)$ is supported on a line, i.e. it is the pushforward of a sheaf under an embedding of $\mathbb{P}^1$ in the mirror quintic. The stalks of this sheaf have rank $2$ in a Zariski open subset of $\mathbb{P}^1$. The line supporting the mirror to $(\tLvg,\nabla^{\vG}_{\omega})$, as defined in Example~\ref{vangeemenlocalsystem}, is a van Geemen line. 
\end{proposition}
\begin{proof}
It is immediate from our discussion above that the sheaf on $X^{5,\vee}$ mirror to $(\tLvg,\nabla)$ is supported on a Zariski open subset of a line. That the stalks are generically of rank $2$ follows from~\cite[Corollary 6.1]{Han24a}. Since the mirror is a complex of \textit{coherent} sheaves, the ranks of its stalks cannot decrease on a Zariski closed subset, and thus its support must be an entire embedded $\mathbb{P}^1$. We have already checked that $\nabla^{\vG}_{\omega}$ is unobstructed. Recall that the mirror functor of~\cite{She15} is obtained as a versal deformation of a fully faithful $A_{\infty}$ embedding
\begin{align*}
\mathrm{Perf}(X^{5,\vee}_0)\to D^{\pi}(\mathcal{F}(X_{\tau}^5))\,.
\end{align*}
We can thus determine the support of the mirror sheaf to $(\tLvg,\nabla^{\vG}_{\omega})$ by determining the restriction of the support to the central fiber. The lemma now follows from the explicit description of the support of the mirror from~\cite[\S{6.3}]{Han24a} and Example~\ref{vangeemenlocalsystem}.
\end{proof}

\subsection{Local Floer cohomology and the second Chern class}\label{localhfsection}
To prove Theorem~\ref{mainprelim}, we need a stronger version of Proposition~\ref{globalsupport}. More specifically, if we can show that the mirror sheaf to $(\tLvg,\nabla^{\vG}_{\omega})$ is the \textit{pushforward of a vector bundle} on $\mathbb{P}^1$, we will be able to conclude that its algebraic second Chern class is an integer multiple of the support, which is critical for understanding the extensions of Hodge structure associated to this object. Let $\mathcal{L}_{(\tLvg,\nabla)}$ denote the mirror object to the brane $(\tLvg,\nabla)$, where $\nabla$ is unobstructed. We have shown in Proposition~\ref{globalsupport} that this object is supported on a line in $X^{5,\vee}$, and we let $i\colon\mathbb{P}^1\to X^{5,\vee}$ denote the inclusion of the support. Then we have an isomorphism
\begin{align}
\mathcal{L}_{(\tLvg,\nabla)}\cong i_*(i^{-1}\mathcal{L}_{(\tLvg,\nabla)}) \,. \label{inversehseaf}
\end{align}
Since $i^{-1}\mathcal{L}_{(\tLvg,\nabla)}$ is a complex of \textit{coherent} sheaves on $\mathbb{P}^1$, it can be written as a direct sum of line bundles and skyscraper sheaves. The results of~\cite[\S{6.3}]{Han24a} give us a lower bound on the rank of the stalks this sheaf.

We can rule out the presence of skyscraper summands in this sheaf by establishing an upper bound on the rank of $HF^*(\tLvg,\nabla^{\vG}_{\omega})$. Luckily, it is possible to achieve this while only computing the differentials on the first page of the energy spectral sequence of Proposition~\ref{spectralseq}. This computation, combined with the mirror symmetry considerations above, then gives us enough information to completely determine the Floer cohomology of $(\tLvg,\nabla)$, and the precise object to which it is mirror.
\begin{theorem}\label{mirrorsheaf}
The Floer cohomology of $(\tLvg,\nabla)$, where $\nabla$ is an unobstructed local system, is the graded $\Lambda_0$-module given by
\begin{align}\label{floerhomology}
HF^*(\tLvg,\nabla)\cong\begin{cases}
\Lambda_0 & * = -1,4 \\
\Lambda_0^3 & * = 0,3 \\
\Lambda_0^4 & * = 1,2 \\
0 & \text{otherwise.}
\end{cases}
\end{align}
Moreover $(\tLvg,\nabla)$ is mirror to a direct sum of two copies of the same line bundle on a line in the mirror quintic, whose gradings differ by $1$.
\end{theorem}
By construction, the differentials on the $E_2$-page of the energy spectral sequence of Proposition~\ref{spectralseq} are determined by the terms of the Floer differential counting (nonconstant) disks of the lowest energy bounded by $\tLvg$. The Weinstein neighborhood $\tWvg$ contains all of these low-energy disks, allowing us to rephrase the computation of the $E_2$-differential as a computation of local Floer homology. More precisely, the local Floer homology $HF^*_{\tWvg}(\tLvg,\nabla)$ calculated inside $\tWvg$ completely determines the $E_3$-page of the energy spectral sequence.
\begin{lemma}\label{localstrips}
There is a constant $\epsilon_0>0$ such that every holomorphic disk $u$ with corners and boundary on $\tLvg$ for which $\omega([u])\leq\epsilon_0$ is contained in $\tWvg$. In particular the only such holomorphic disks which contribute to the Floer differential are either
\begin{itemize}
\item[(i)] holomorphic teardrops as in Lemma~\ref{localtears} with an additional smooth boundary marked point; or
\item[(ii)] holomorphic strips with two corners on different switching components (at least one of which corresponds to a Morse critical point of $\widetilde{h}$).
\end{itemize}
\end{lemma}
\begin{proof}
Using the same arguments as in the proof of Lemma~\ref{exactness}, one can show that $\tLvg$ is exact for some choice of primitive on $\tWvg$. An easy consequence of this is that the areas of these disks are controlled by the choice of Morse function $\widetilde{h}\colon\widetilde{L}'\to\mathbb{R}$ and the area of the disk bounded by $\Larc$. These can both be made arbitrarily small by scaling $\widetilde{h}$ appropriately.

The classification of strips on $\tLvg$ follows immediately from this, where in particular the strips in (ii) correspond to gradient flowlines of $\pm\widetilde{h}'$. Notice that the $2$-fold covers of the teardrops on $\tLvg$ do not contribute to the Floer differential for degree reasons.
\end{proof}

Since $\tLvg$ is built using $125$-fold covers $\widetilde{L}'$ of the minimally-twisted five-component chain link complement, its de Rham cohomology $H^*(\tLvg;\Lambda_0)$ will have high rank as a free $\Lambda_0$-module. This makes directly computing even the $E_2$-differentials in the energy spectral sequence difficult. To remedy this, we will using a covering argument along the lines of Lemma~\ref{trivialactionsupp}, which allows us to compute the low energy terms of the Floer differential for the quotient Lagrangian $\Lvg$ instead.
\begin{lemma}~\label{trivialactionfloer}
The group $(\mathbb{Z}/5)^3$ acts trivially on $HF^*(\tLvg,\nabla)$. Thus we have isomorphisms
\begin{align}\label{symmetryisos1}
HF^*_{X_{\tau}^5}(\tLvg,\nabla)= HF^*_{X_{\tau}^5}(\tLvg,\nabla)^{(\mathbb{Z}/5)^3} = H^*(CF^*_{X_{\tau}^5}(\tLvg,\nabla)^{(\mathbb{Z}/5)^3},\mathfrak{m_1})
\end{align}
where the subscripts indicate that all Floer chain complexes and cohomology groups are taken in the Weinstein neighborhood $\tWvg$, and
\begin{align}\label{symmetryisos2}
H^*(CF^*_{\tWvg}(\tLvg,\nabla)^{(\mathbb{Z}/5)^3},\mathfrak{m_1})\cong HF^*_{\Wvg}(\Lvg,\overline{\nabla})
\end{align}
where the group on the right is the Floer cohomology of $(\Lvg,\overline{\nabla})$ in the quotient $\Wvg$.
\end{lemma}
\begin{proof}
The proof mostly uses the arguments of the previous subsection. Recall that the action of $(\mathbb{Z}/5)^3$ is inherited from an action of $(\mathbb{Z}/5)^3$ on $\mathbb{CP}^4$. The very affine quintic $X_{\tau}^5\setminus D_{\tau}$ is contained in the big torus $(\mathbb{C}^*)^4\subset\mathbb{CP}^4$. As before, we have that $(\mathbb{Z}/5)^3$ acts on $(\mathbb{C}^*)^4$ by Hamiltonian isotopies, and it follows that it acts on $X_{\tau}^5\setminus D_{\tau}$, and hence on $X_{\tau}^5$ by Hamiltonian isotopies as well. This proves the nontrivial equality~\eqref{symmetryisos1}.

The isomorphism of~\eqref{symmetryisos2} follows from Lemma~\ref{localstrips} and a completely analogous classification of holomorphic strips in $\Wvg$ with boundary on $\Lvg$. In particular, $(\mathbb{Z}/5)^3$ still acts on the Floer cochain space of $CF^*(\tLvg,\nabla)$, and the subcomplex of invariant Floer cochains is naturally identified with the Floer complex $CF^*(\Lvg,\overline{\nabla})$.
\end{proof}
The lemma above reduces the problem of computing $HF^*_{\tWvg}(\tLvg,\nabla)$, to the problem of computing the (local) Floer cohomology group $HF^*_{\Wvg}(\Lvg,\overline{\nabla})$ in $\Wvg$.
\begin{figure}
\begin{tikzpicture}
\begin{scope}[scale = 1.5, very thick,decoration={
    markings,
    mark=at position 0.5 with {\arrow{>}}}
    ]
    
\node[anchor = south east] at (0,2) {$0$};
\node[anchor = north] at (0,-2) {$0$};
\node[anchor = west] at (3,2) {$0$};
\node[anchor = north] at (1,1.9) {$0$};

\node[anchor = west] at (1,0) {$2$};
\node[anchor = east] at (-1,0) {$1$};
\node[anchor = south] at (2,4) {$4$};
\node[anchor = west] at (2,0) {$3$};

\draw[blue, postaction = {decorate}] (-1/4,3/2) -- (1/4,3/2);
\draw[blue, postaction = {decorate}] (-1/4,-3/2) -- (1/4,-3/2);
\draw[blue, postaction = {decorate}, dotted] (5/4,3/2) -- (5/4,5/2);
\draw[blue, postaction = {decorate}, dotted] (11/4,3/2) -- (11/4,5/2);

\node[anchor = north] at (0,3/2) {\color{blue} $m_0$};
\node[anchor = south] at (0,-3/2) {\color{blue} $m_0$};
\node[anchor = west] at (5/4,2) {\color{blue} $m_0$};
\node[anchor = east] at (11/4,2) {\color{blue} $m_0$};

\draw[blue, postaction = {decorate}, dotted] (-3/4,1/2) -- (-1/3,2/3); 
\draw[blue, postaction = {decorate}] (3/4,-1/2) -- (3/4,1/2); 
\draw[blue, postaction = {decorate}, dotted] (9/4,1/2) -- (7/4,1/2); 
\draw[blue, postaction = {decorate}, dotted] (7/4,7/2) -- (4/3,10/3); 

\node[anchor = west] at (-1/3,2/3) {\color{blue}$m_1$};
\node[anchor = east] at (3/4,0) {\color{blue}$m_2$};
\node[anchor = west] at (9/4,1/2) {\color{blue}$m_3$};
\node[anchor = south east] at (4/3,10/3) {\color{blue}$m_4$};

\draw[red, postaction = {decorate}, dotted] (1/2,5/2) -- (-1/4,3/2);
\draw[red, postaction = {decorate}] (1/4,-3/2) -- (2/3,-4/3);
\draw[red, postaction = {decorate}] (7/3,4/3) -- (11/4,3/2);
\draw[red, postaction = {decorate}, dotted]  (5/4,5/2) -- (1/2,3/2);

\node[anchor = south east] at (1/2,5/2) {\color{red} $\ell_0$};
\node[anchor = north west] at (2/3,-4/3) {\color{red} $\ell_0$};
\node[anchor = south east] at (7/3,4/3) {\color{red}$\ell_0$};
\node[anchor = west] at (1/2,2) {\color{red}$\ell_0$};

\draw[red, postaction = {decorate}] (-3/4,-1/2) -- (-3/4,1/2); 
\draw[red, postaction = {decorate}] (3/4,1/2) -- (3/2,1/2); 
\draw[red, postaction = {decorate}, dotted] (7/4,1/2) -- (3/2,-1/2); 
\draw[red, postaction = {decorate}, dotted] (9/4,7/2) -- (7/4,7/2); 

\node[anchor = west] at (-3/4,0) {\color{red}$\ell_1$};
\node[anchor = south] at (9/8,1/2) {\color{red}$\ell_2$};
\node[anchor = north west] at (3/2,-1/2) {\color{red}$\ell_3$};
\node[anchor = north] at (2,7/2) {\color{red}$\ell_4$};

\draw[thick] (0,2) -- (1,0) -- (0,-2) -- (-1,0) -- cycle;

\draw[thick, dotted] (2,4) -- (3,2) -- (2,0) -- (1,2) -- cycle;

\draw[thick] (0,2) -- (2,4);
\draw[thick] (1,0) -- (3,2);
\draw[thick] (0,-2) -- (2,0);
\draw[thick, dotted] (-1,0) -- (1,2);

\end{scope}
\end{tikzpicture}
\caption{The meridians and longitudes drawn as arcs on one of the ideal cubes of $L'$.}\label{loopscube}
\end{figure}
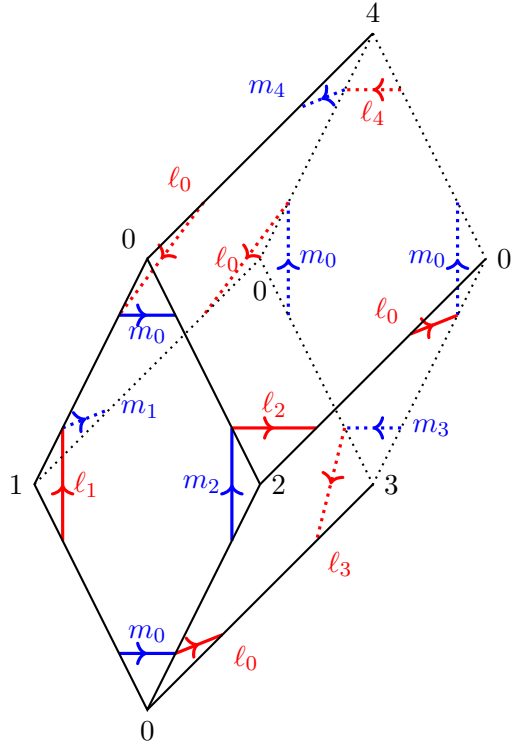
\begin{proposition}\label{hardlocal}
The Floer cohomology groups $HF^*_{\Wvg}(\Lvg,\overline{\nabla})$ coincide with those given in~\eqref{floerhomology}.
\end{proposition}
We will break the proof of this proposition into several smaller lemmas. As a $\mathbb{Z}$-graded $\Lambda_0$-module, the Floer cochain complex of $(\Lvg,\overline{\nabla})$ in $\Wvg$ is the completed tensor product 
\begin{align*}
CF^*_{\Wvg}(\Lvg,\overline{\nabla}) = \overline{CF}^*_{\Wvg}(\Lvg,\overline{\nabla})\widehat{\otimes}_{\mathbb{C}}\Lambda_0
\end{align*}
where $\overline{CF}^*_{\Wvg}(\Lvg,\overline{\nabla})$ is the $\mathbb{Z}$-graded $\mathbb{C}$-vector space
\begin{align*}
\Omega^*(\Lvg)\oplus\bigoplus_{\ell = 0}^4(\Omega^*(T^2)[1]\oplus\Omega^*(T^2)[-2])\oplus CM^*(L')[-1]\oplus CM^*(L',\partial L')[1] \, .
\end{align*}
This implies that the $E_2$-page of the energy spectral sequence is determined by the $\mathbb{C}$-vector space
\begin{align*}
&H^{\bullet}(\overline{CF}^*_{\Wvg}(\Lvg,\overline{\nabla})) \\
&\coloneqq H^*(\Lvg)\oplus\bigoplus_{\ell = 0}^4 (H^*(T^2)[1]\oplus H^*(T^2)[-2])\oplus CM^*(L')[-1]\oplus CM^*(L',\partial L')[1]
\end{align*}
obtained by taking de Rham cohomology. The terms of the $E_2$-page are given by
\begin{align}
E_2^{p,q}\coloneqq H^p(\overline{CF}^*_{\Wvg}(\Lvg,\overline{\nabla}))\otimes(Q^{q\epsilon_0}\Lambda_0/Q^{(q+1)\epsilon_0}\Lambda_0) \label{ssdiffs}
\end{align}
where $\epsilon_0$ denotes the constant of Lemma~\ref{localstrips}. The differential $E_2^{p,q}\to E_2^{p+1,q+1}$ is determined from the Floer differential on $CF^*(\Lvg)$ by setting
\begin{align*}
\delta_2^{p,q}[x] = [\mathfrak{m}_1(x)] \in E_2^{p+1,q+1}
\end{align*}
for any $[x]\in E_2^{p,q}$. Notice that in $\Wvg$, the energy spectral sequence collapses at the $E_2$-page by exactness, so computing the differentials on this page amounts to computing the full Floer cohomology groups.

The differential $\delta_2^{p,q}$ is determined by $\mathfrak{m}_1^p$, the degree $p$ part of the differential
\begin{align*}
\mathfrak{m}_1^p\colon CF^p(\Lvg)\to CF^{p+1}(\Lvg) \, .
\end{align*}
Again by the exactness of $\Lvg$ in $\Wvg$, we can think of this as a $\mathbb{C}$-linear map. Consequently, the $E_2$-differentials in the spectral sequence are determined by $\mathbb{C}$-linear maps
\begin{align}
\delta^{p}\colon H^{p}(\overline{CF}^*_{\Wvg}(\Lvg,\overline{\nabla}))\to H^{p+1}(\overline{CF}^*_{\Wvg}(\Lvg,\overline{\nabla}))
\end{align}
induced from the Floer differential $\mathfrak{m}_1^p$ as in~\eqref{ssdiffs}. More specifically, $\delta^p$ can be written as a sum of linear maps \begin{align*}
\delta^p = \sum_{\beta\in H_2(\Wvg,\Lvg)}\delta^p_{\beta}
\end{align*}
where $\delta^p_{\beta}$ is defined as in~\eqref{ssdiffs} using $\mathfrak{m}^p_{1;\beta}$. We then have that $\delta^{p,q}_2$ is represented by
\begin{align*}
\delta^{p,q}_2 = \sum_{\beta\in H_2(\Wvg,\Lvg)}\delta^p_{\beta}\otimes Q^{\omega(\beta)}\id_{\Lambda_0}
\end{align*}
on $E^{p,q}_2$ for all $q$. In the following three lemmas, we will compute the maps $\delta^p$ for $p = 3,2,1$.

\begin{lemma}\label{m3lemma}
The map
\begin{align}
\delta^{3}\colon H^3(\Lvg)\oplus\bigoplus_{\ell = 0}^4 H^1(T^2)[-2]\oplus CM^2(L')[-1]\to\bigoplus_{\ell = 0}^4 H^2(T^2)[-2] \label{m3e2diff}
\end{align}
induced by $\mathfrak{m}_1^3$ has rank $4$.
\end{lemma}
\begin{proof}
This component of the $E_2$-differential is induced, at the chain level, by the Floer differential
\begin{align*}
\mathfrak{m}_1^3\colon \Omega^3(\Lvg)\oplus\bigoplus_{\ell = 0}^4 \Omega^1(T^2)[-2]\oplus CM^2(L')[1]\to\bigoplus_{\ell = 0}^4 \Omega^2(T^2)[-2] \,. 
\end{align*}
The only holomorphic strips that contribute to this part of the Floer correspond to gradient trajectories of $h$ starting at index $2$ critical points and approaching the cusps of $L'$. This implies that the only nontrivial component of the Floer differential is
\begin{align*}
CM^2(L')[-1]\to\bigoplus_{\ell = 0}^4 \Omega^2(T^2)[-2] \,.
\end{align*}
The critical points of $h$ in degree $2$ correspond to edges of the cube in Figure~\ref{loopscube}. There are two gradient trajectories which emanate from each of these critical points: one of these approaches the $0$th cusp, and the other approaches the $i$th cusp, where $i\in\lbrace 1,2,3,4\rbrace$. From this, the image of $\mathfrak{m}_1^3$ is easily seen to be a free module of rank $4$, from which the statement of the lemma follows.
\end{proof}

Calculating other components of the $E_2$-differential requires the determination of some nontrivial values of the Floer differential. Let $\beta\in H_2(\Wvg,\Lvg;\mathbb{Z})$ denote a relative second homology class represented by one of the disks of Lemma~\ref{localstrips}, and consider the (Gromov compactified) moduli spaces
\begin{align}
\mathcal{M}_2(\beta) \label{modulispacelocaldiff}
\end{align}
of such disks. By the argument in the proof of Lemma 2.10 of~\cite{AS21}, all moduli spaces~\eqref{modulispacelocaldiff} are transversely cut out. These moduli spaces are either $3$-dimensional, in the case of Lemma~\ref{localstrips}(i), or $0$-dimensional, in the case of Lemma~\ref{localstrips}(ii). When these moduli spaces are $3$-dimensional they can be identified with the product of a $2$-torus with a closed interval.

Recall from Appendix~\ref{immersedfloerappendix} that the Floer differential is defined by pulling back differential forms on $CF^*(\Lvg)$ under the evaluation map at one of the boundary marked points, and pushing forward the resulting differential form under the evaluation map at the other marked point. The boundary evaluation maps are submersions, which we see because all switching components are either $0$-dimensional, or they are contained in neighborhoods in which $\Lvg$ can be written as the product of a $1$-dimensional Lagrangian with $T^2$, which is a Lie group, so that submersivity follows by the argument of~\cite[Example 1.5]{ST22}.

Also recall that $\Lvg$ is obtained by gluing two smooth manifolds with boundary diffeomorphic to the disjoint union of five copies of $T^2$. Below, we will call the images of these boundary tori in $\Lvg$ the \textit{splitting tori}. With the above understood, we can now compute the remaining terms of the $E_2$-differential.

\begin{lemma}\label{m2lemma}
The $\mathbb{C}$-linear map $\delta^2$ induced by $\mathfrak{m}_1^2$ has domain and codomain
\begin{align}\label{m2e2diff}
\begin{tikzcd}
H^2(\Lvg)\oplus\bigoplus_{\ell = 0}^4 H^0(T^2)[-2]\oplus CM^3(L',\partial L')[1]\oplus CM^1(L')[-1] \arrow{d} \\
H^3(\Lvg)\oplus\bigoplus_{\ell = 0}^4 H^1(T^2)[-2]\oplus CM^2(L')[-1]
\end{tikzcd}
\end{align}
and has rank $8$.
\end{lemma}
\begin{proof}
At the chain-level, the $E_2$-differentials are induced by the Floer differential $\mathfrak{m}_1^2$ in degree $2$. The only strips whose output is a class in $\Omega^3(\tLvg)$ come from teardrops as in Lemma~\ref{localstrips}(i). Since these teardrops are weighted by holonomies and signs whose sum vanishes, it follows that their total contribution to the differential vanishes. Therefore $\Omega^3(\tLvg)$ does not lie in the image of the Floer differential. Similarly, the component of the Floer differential mapping out of $CM^3(L',\partial L')[1]$ is trivial.

Fix a basis for $H_1(\Lvg)$ consisting of the classes $s_1,\ldots,s_4\in H_1(\Lvg)$ as in~\eqref{singularh1basis}, together with classes $\ell_0,\ldots,\ell_4$ representing the longitudes in $L'$. Using the de Rham isomorphism and Poincar{\'e} duality, we can identify these with generators for the de Rham cohomology $H^2(\tLvg)$. Fix differential forms in $\Omega^2(\tLvg)$ representing these classes. The images of these classes under $\mathfrak{m}_1^2$ are all closed $1$-forms in $\bigoplus_{\ell = 0}^4 \Omega^1(T^2)[-2]$ by general properties of integration along the fiber. We their images in $\bigoplus_{\ell = 0}^4 H^1(T^2)[-2]$ span a subspace of dimension $5$, since all of the forms dual to the classes $s_1,\ldots,s_4$ which are not generated by the longitudes in $L'$ map to zero.

The Floer cocycles in $CM^1(L')[-1]$ correspond to faces of the cube in Figure~\ref{loopscube}. The only holomorphic strips contributing to the part of $\mathfrak{m}_1^2$ that maps out of this summand are as in Lemma~\ref{localstrips}(ii), and correspond to gradient trajectories of $h$ starting at index $1$ critical points and approaching the cusps of $L'$. There are two such trajectories starting at any such critical point and emanating towards the zeroth cusp, as can be seen from Figure~\ref{loopscube}. The corresponding holomorphic strips contribute with the opposite signs, so their total contribution to the Floer differential is trivial. There are two more gradient trajectories starting at any such critical point, but they will approach different cusps of $L'$.

In total, the images of the classes considered above span an $8$-dimensional subspace of the codomain of~\eqref{m2e2diff}. To see this, notice the the images of generators in $CM^1(L')[-1]$, together with the images of the forms corresponding to $\ell_1,\ldots,\ell_4$, span all of $\bigoplus_{\ell=1}^4 H^1(T^2)[1]$. It then follows from Lemma~\ref{relations} that the image of the form corresponding to $\ell_0$ lies in the span of these classes. This implies that $\mathfrak{m}_1^2$ induces a map of rank $8$ on de Rham cohomology.
\end{proof}
The last $E_2$-differentials that we will need to compute come from the Floer differential in degree $1$.
\begin{lemma}
The $\mathbb{C}$-linear map $\delta^1$ induced by $\mathfrak{m}_1^1$ has domain and codomain
\begin{equation*}
\begin{tikzcd}
H^1(\Lvg)\oplus\bigoplus_{\ell = 0}^4 H^2(T^2)[1]\oplus CM^0(L')[-1]\oplus CM^2(L',\partial L')[1] \arrow{d} \\
H^2(\Lvg)\oplus\bigoplus_{\ell = 0}^4 H^0(T^2)[2]\oplus CM^1(L')[-1]\to CM^3(L',\partial L')[1]
\end{tikzcd}
\end{equation*}
and has rank $10$.
\end{lemma}
\begin{proof}
First observe that the Floer differential is nontrivial on the components
\begin{align*}
CM^2(L',\partial L')[1] &\to CM^3(L',\partial L')[1] \\
CM^0(L')[-1] &\to CM^1(L')[-1]
\end{align*}
since there are strips as in Lemma~\ref{localstrips}(ii) which correspond to gradient flow trajectories connecting the two index $0$ points to the index $1$ critical points. Both of the above components of the Floer differential have rank $1$ (cf. Figure~\ref{loopscube}).

By Lemma~\ref{localstrips}, the only other possibly nontrivial components of $\mathfrak{m}_1^1$ are
\begin{align}
\bigoplus_{\ell = 0}^4\Omega^2(T^2)[1]\to CM^3(L',\partial L')[1] \label{deg1trivial1} \\
\bigoplus_{\ell = 0}^4 CM^0(L')[-1]\to\bigoplus_{\ell = 0}^4\Omega^0(T^2)[-2] \label{deg1trivial2}
\end{align}
which count holomorphic strips described in Lemma~\ref{localstrips}(ii), or
\begin{align}
\Omega^1(\Lvg)\to\bigoplus_{\ell = 0}^4\Omega^0(T^2)[-2] \label{deg1nontriv1} \\
\bigoplus_{\ell = 0}^4\Omega^2(T^2)[1]\to \Omega^2(\Lvg) \label{deg1nontriv2}
\end{align}
which count holomorphic disks as in Lemma~\ref{localstrips}(i).

Observe that~\eqref{deg1trivial2} vanishes, because for any holomorphic strip with a corner at one of the two generators of $CM^0*(L')[-1]$ and another corner on a $T^2$-switching component, there is a corresponding holomorphic strip with a corner at the other generator of $CM^0*(L')[-1]$, and these two strips are counted with opposite signs. Since~\eqref{deg1trivial1} counts precisely the same strips, it also vanishes. Consequently, these components of the Floer differential induce trivial maps on the $E_2$-page of the spectral sequence.

We will compute the terms of the $E_2$-differential corresponding to~\eqref{deg1nontriv1}. To that end, choose closed $1$-forms in $\Omega^1(\Lvg)$ whose cohomology classes correspond to the generators in~\eqref{singularh1basis} under the de Rham isomorphism. The only disks which contribute to~\eqref{deg1nontriv1} come from teardrops as in Lemma~\ref{localstrips}(i). Let $\beta_{\ell}\in H_2(\Wvg,\Lvg)$ denote the relative homology class represented by such a disk. The moduli space of strips in this class, where one marked point is the corner of the teardrop, and the other is at a smooth point on the boundary, is denoted $\mathcal{M}_2(\beta_{\ell})$, as in~\eqref{modulispacelocaldiff}. Let $\evb_0^{\beta_{\ell}}$ denote the evaluation map at the corner, and $\evb_1^{\beta_{\ell}}$ denote the evaluation map at the smooth marked point. If $\delta_{s_i}$ denotes the differential form dual to class $s_i$, then the form
\begin{align}
(\evb_0^{\beta_{\ell}})_*(\evb_1^{\beta_{\ell}})^*\sigma_i\in\Omega^0(T^2)[-2] \label{nontrivialcoeff}
\end{align}
is a volume form on the $\ell$th splitting torus. The images of the classes $\delta_{s_i}$ under $\mathfrak{m}_1^1$ form a $4$-dimensional subspace, spanned by sums of volume forms. Note that for appropriately chosen classes $s_i$, the moduli space $\mathcal{M}_2(\beta_{\ell})$ will contribute nontrivially to the Floer differential of $\delta_{s_i}$ if and only if $\beta_{\ell}$ is represented by the product of a teardrop on $\Larc$ with a constant disk. In particular, this value of the Floer differential does not vanish. On the other hand, if $\delta_{m_i}$ is a differential form dual $m_i$, for $i = 0,\ldots,4$, then its image under the Floer differential vanishes. Hence the component~\eqref{deg1nontriv1} descends to a map of rank $4$ on de Rham cohomology. The map on homology induced by~\eqref{deg1nontriv2} is the dual of this map, and so it also has rank $4$. In total, this shows that the map in the statement of the lemma has rank $10$.
\end{proof}

\begin{proof}[Proof of Proposition~\ref{hardlocal}]
The results of the previous three lemmas, combined with the collapse of the energy spectral sequence computed in $\Wvg$ at the $E_2$-page, show that
\begin{align}
HF^*_{\Wvg}(\Lvg,\overline{\nabla})\cong\begin{cases}
\Lambda_0 & * = 4 \\
\Lambda_0^3 & * = 3 \\
\Lambda_0^4 & * = 2 \,.
\end{cases}
\end{align}
The values of the Floer cohomology groups for $* = -1,0,2$ follow from this by Poincar{\'e} duality.
\end{proof}

\begin{proof}[Proof of Theorem~\ref{mirrorsheaf}]
Observe that by Lemma~\ref{trivialactionfloer}, the result of Proposition~\ref{hardlocal} gives an upper bound on the ranks of $HF^*(\tLvg,\nabla)$. On the other hand, we have already seen in Proposition~\ref{globalsupport} that the mirror to $(\tLvg,\nabla)$ is a sheaf whose stalks all have rank at least $2$. Since each of these supports is a line contained in a $1$-dimensional family in a Calabi--Yau threefold, it follows that the rank cannot be greater than $2$ at any point in the support, or else the total rank of $HF^*(\tLvg,\nabla)$ would be higher than the rank of $HF^*_{\tWvg}(\tLvg,\nabla)$. As before, let $\mathcal{L}_{(\tLvg,\nabla)}$ denote the mirror object to $(\tLvg,\nabla)$, and let $i\colon\mathbb{P}^1\to X^{5,\vee}$ denote the inclusion map of its support. We know, by~\eqref{inversehseaf}, that $\mathcal{L}_{(\tLvg,\nabla)}$ is the pushforward of the sheaf $i^{-1}\mathcal{L}_{(\tLvg,\nabla)}$ on $\mathbb{P}^1$. In particular, it is the pushforward of a vector bundle of rank $2$. 

The grading of the $\Lambda_0$-module $HF^*_{\tWvg}(\tLvg,\nabla)$ computed in Proposition~\ref{hardlocal}, together with the fact that any vector bundle on $\mathbb{P}^1$ splits, implies that the mirror object to $(\tLvg,\nabla)$ is a direct sum of two rank $1$ vector bundles on a line, both of which have the same degree, but with gradings that differ by $1$.
\end{proof}

From the result of Theorem~\ref{mirrorsheaf} and the Grothendieck--Riemann--Roch theorem, we can compute the algebraic second Chern classes of the sheaves mirror to $(\tLvg,\nabla)$. 
\begin{corollary}\label{algc2computation}
The mirror sheaves $\mathcal{L}_{(\tLvg,\nabla)}$ to the objects $(\tLvg,\nabla)$ in $\mathcal{F}(X_{\tau}^5)$ have algebraic second Chern classes given by
\begin{align*}
c_2(\mathcal{L}_{(\tLvg,\nabla)}) = -2[C_{\nabla}]
\end{align*}
where $C_{\nabla}$ is the support  of $\mathcal{L}_{(\tLvg,\nabla)}$.
\end{corollary}
\begin{proof}
See~\cite[p. 29]{Fri98}.
\end{proof}

\subsection{Lagrangian surgery and direct summands}\label{lsmthchern}
Theorem~\ref{mirrorsheaf} and the classification of coherent sheaves on $\mathbb{P}^1$ shows that any object of the form $(\tLvg,\nabla)$, where $\nabla$ is a local system satisfying~\eqref{interiorconeobs}--\eqref{4obs}, splits as a direct sum in the split-closed derived Fukaya category. In this subsection, we will show that the summands of such objects can be identified with Lagrangian branes supported on $\tLvgsmooth$, which will be enough to apply Theorem~\ref{main2} to calculate their open Gromov--Witten potentials. This is achieved by showing that two copies of $\tLvgsmooth$ can be obtained from a copy of $\tLvg$ under Lagrangian isotopy and a clean (anti-)surgery. 

Let $\tLvgsmooth\coloneqq\tLvgsmooth(i;\epsilon)$ denote one of the Lagrangian submanifolds constructed in Theorem~\ref{lsmthconst}, and let $\widetilde{\delta}$ the closed $1$-form on $\widetilde{L}'$ used to construct it. By taking $\epsilon$ sufficiently small, we can assume that $\widetilde{\delta}$ is arbitrarily $C^1$-small. Observe that $\widetilde{\delta}$ extends to a $1$-form on the domain of $\tLvg$, since we already assumed that it can be expressed as the product of a $1$-form on $T^2$ with a constant function on the real line near the cusps of $\widetilde{L}'$. Let $\widetilde{\delta}_{\mathrm{im}}\in\Omega^1(\tLvg)$ denote the $1$-form obtained by patching two copies of $\widetilde{\delta}$ defined on the sheets of $\tLvg$.
\begin{lemma}\label{isotopyequiv}
Let $(\tLvg,\nabla)$ denote an object of $\mathcal{F}(X_{\tau}^5)$, where $\nabla$ is an unobstructed local system. Deforming $\tLvg$ in a Weinstein neighborhood (contained in $\tWvg$) through the graphs $\tLvg(t)\coloneqq\Gamma(t\widetilde{\delta}_{\imm})$ of $t\widetilde{\delta}$, for $t\in[0,1]$ and equipping $\tLvg(1)$ with a suitable rank one $\Lambda$-local system $\nabla_{\Lambda}^{\imm}$ yields a Lagrangian brane $(\tLvg(1),\nabla_{\Lambda}^{\imm})$ for which we have a quasi-isomorphism
\begin{align}\label{isotopyqis}
CF^*(\tLvg,\nabla)\simeq CF^*(\tLvg(1),\nabla_{\Lambda}^{\imm}) \,. 
\end{align}
\end{lemma}
\begin{remark}\label{localisotopy}
Before proceeding with the proof, we will describe the effect of this Lagrangian isotopy in the Darboux charts $\widetilde{B}_{\ell,m}$. Recall that these charts are identified with open neighborhoods of the origin in $Y$ of~\eqref{gssing}, and that in these charts $\tLvg$ coincides with the Lagrangian submanifold $L_Y$ of Definition~\ref{lydef}. In these charts, isotoping the intersection of $\tLvg$ through the graphs $\Gamma(t\widetilde{\delta}_{\mathrm{im}})$ determines a Lagrangian isotopy $L_Y(t)$ of Lagrangian submanifolds of $Y$. Each $L_Y(t)$ still projects to $\Larc$, as drawn in Figure~\ref{immersedarc}, under the projection $Y\to\mathbb{C}$ of~\eqref{degeneration}.
\end{remark}
\begin{proof}[Proof of Lemma~\ref{isotopyequiv}]
Observe that $\tLvg$ is isotoped to $\tLvg(1)$ in a way that determines a bijective correspondence between holomorphic disks bounded by the two immersed Lagrangian submanifolds. More precisely, for any class $\beta\in H_2(X^5_{\tau},\tLvg)$, there is a corresponding class $\beta_1\in H_2(X^5_{\tau},\tLvg(1))$ and bijections between the moduli spaces of holomorphic disks representing these homology classes. The contribution of any such disk to the Fukaya $A_{\infty}$-algebra of $\tLvg(1)$ differs from its contribution to the $A_{\infty}$-algebra of $\tLvg$ by a factor of $Q^{-\int_{\partial\beta}\widetilde{\delta}_{\imm}}$. This factor and $\nabla$ then determine the holonomy representation of the local system $\nabla^{\imm}_{\Lambda}$ on $\tLvg(1)$.
\end{proof}
We remark that the local system $\nabla^{\imm}_{\Lambda}$ is non-unitary over the Novikov field, but it is clear from its construction, and in particular~\eqref{isotopyqis}, that all sums involved in the definition of the $A_{\infty}$-algebra converge.

A key observation is that $\tLvg(1)\cap\widetilde{B}_{\ell,m} = L_Y(1)$ admits a description as a clean Lagrangian surgery on two smoothings of the Harvey--Lawson cone. More precisely, recall that $\delta$ restricts to a $1$-form on the complement of the cone point $C_{HL}\setminus\lbrace 0\rbrace$ in the Harvey--Lawson cone. This consequently determines a smooth Lagrangian solid torus $L_Y(i;\epsilon)$ in $Y$ given by one of~\eqref{acsmoothings}. The exact $1$-forms $\pm d\widetilde{h}$ used to construct $\tLvg$ restrict to exact $1$-forms on $C_{HL}\setminus\lbrace 0\rbrace$, and we can use these to construct a pair of Lagrangian solid tori $L_Y(i;\epsilon)_{\pm}$ intersecting cleanly in a circle and a $2$-torus. These two solid tori project to arcs in $\mathbb{C}$ under~\eqref{degeneration} as shown in Figure~\ref{solidtoriprojections}. The $S^1$ self-intersection component lies above the origin, and the $T^2$-component lies above the other intersection point in $\mathbb{C}$. The latter component corresponds to the switching component of $L_Y(1)$.
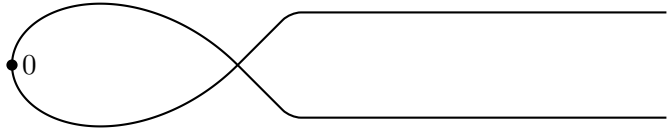
\begin{figure}
\begin{tikzpicture}
\draw[smooth, thick] 
  plot[domain=135:225,samples=200] (\x:{3*cos(2*\x)});
  
  \draw[rounded corners, thick] (0,0) -- (0.7,0.7) -- (5.7,0.7);
  \draw[rounded corners, thick] (0,0) -- (0.7,-0.7) -- (5.7,-0.7);
  
  \node[] at (-3.5,0) {};
  \node[circle, fill, inner sep = 1.5 pt] at (-3,0) {};
  \node[anchor = west] at (-3,0) {$0$};
\end{tikzpicture}
\caption{The images of $L_Y(i;\epsilon)_{\pm}$ in $\mathbb{C}$ . These Lagrangian solid tori intersect each other cleanly in a circle, lying in the fiber over $0$, and a $2$-torus, lying in the fiber over the other intersection point.}\label{solidtoriprojections}
\end{figure}
The following is immediate from this local description.
\begin{lemma}
By performing a clean Lagrangian surgery along the $S^1$-component of \begin{align*}
L_Y(i;\epsilon)_{-}\cap L_Y(i;\epsilon)_{+}
\end{align*}
we obtain an immersed Lagrangian submanifold of $Y$ Hamiltonian isotopic to $L_Y(1)$. \qed
\end{lemma}
Similarly, by using the exact $1$-form $\pm d\widetilde{h}$ to perturb the graph $\Gamma(t\delta)$ used to construct $\tLvgsmooth$, we obtain two embedded Lagrangian submanifolds $\widetilde{L}^5_{\mathrm{sm},\pm}$ of $X_{\tau}^5$ which intersect each other cleanly. Both of these are Hamiltonian isotopic to $\tLvgsmooth$, and the switching components of the self-intersection correspond either to
\begin{itemize}
\item[•] switching components of $\tLvg$; or
\item[•] switching components of $L_Y(i;\epsilon)_{-}\cap L_Y(i;\epsilon)_{+}$ diffeomorphic to $S^1$.
\end{itemize}
This follows from the fact that, by construction, the union $\widetilde{L}^5_{\mathrm{sm},-}\cup\widetilde{L}^5_{\mathrm{sm},+}$ agrees with (a copy of) $L_Y(1)_{-}\cup L_Y(1)_{+}$ inside the charts $\widetilde{B}_{\ell,m}$ and with $\tLvg(1)$ in the complement of these charts. 
\begin{definition}\label{antisurgerybranes}
We give $\widetilde{L}^5_{\mathrm{sm},-}\cup\widetilde{L}^5_{\mathrm{sm},+}$ the structure of a Lagrangian brane as follows.
\begin{itemize}
\item[•] Both of $\widetilde{L}^5_{\mathrm{sm},\pm}$ are equipped with the local system $\nabla_{\Lambda}$ and spin structure inherited from $\widetilde{L}'$ by Dehn filling;
\item[•] the gradings on $\widetilde{L}^5_{\mathrm{sm},\pm}$ by a shift of $1$; and
\item[•] the immersion $\widetilde{L}^5_{\mathrm{sm},-}\cup\widetilde{L}^5_{\mathrm{sm},+}$ is equipped with bounding cochain $0$.
\end{itemize}
In particular, the bounding cochain vanishes on the $S^1$-switching components along which surgery is performed.
\end{definition}

The main result of this subsection is a verification that the $A_{\infty}$-algebra of this Lagrangian brane structure on $\widetilde{L}^5_{\mathrm{sm},-}\cup\widetilde{L}^5_{\mathrm{sm},+}$ is quasi-isomorphic to the $A_{\infty}$-algebra of the corresponding object supported on $\tLvg(1)$.
Since the former $A_{\infty}$-algebra splits as a direct sum, this realizes the direct sum decomposition in the derived Fukaya category implicit in Theorem~\ref{mirrorsheaf}. This is inspired by the arguments of~\cite{PW19}, in that we use SFT neck stretching to compare holomorphic disks on $\tLvg(1)$ with holomorphic disks on $\tLvgsmooth(i;\epsilon)_{-}\cup\tLvgsmooth(i;\epsilon)_{+}$. The results of~\cite{PW19} do not apply directly in this setting, but a very simplified version of the techniques use in op. cit. will suffice for our purposes. To be more precise, we can appeal the form of the switching loci of $\tLvg(1)$ and $\tLvgsmooth(i;\epsilon)_{-}\cup\tLvgsmooth(i;\epsilon)_{+}$ to entirely rule out the existence of disks with corners on the $S^1$ self-intersection component of the latter Lagrangian immersion. This implies that we do not need to equip $\widetilde{L}^5_{\mathrm{sm},-}\cup\widetilde{L}^5_{\mathrm{sm},+}$ with a bounding cochain with nontrivial support on the $S^1$-switching components.

\begin{lemma}\label{surgeryequiv}
The Lagrangian brane on $\widetilde{L}^5_{\mathrm{sm},-}\cup\widetilde{L}^5_{\mathrm{sm},+}$ specified in Definition~\ref{antisurgerybranes} is quasi-isomorphic to the corresponding object $(\tLvg(1),\nabla_{\Lambda}^{\imm})$ in the split-closed derived Fukaya category.
\end{lemma}
\begin{proof}
Let $\mathcal{A}^5$ denote the union of split-generators for $\mathcal{F}(X_{\tau}^5)$ constructed in~\cite{She15}, where we have identified $X_{\tau}^5$ with $X_{\infty}^5$. It suffices to show that there is a quasi-isomorphism 
\begin{align*}
CF^*(\mathcal{A}^5,(\widetilde{L}^5_{\mathrm{sm},-}\cup\widetilde{L}^5_{\mathrm{sm},+},\nabla_{\Lambda}))\simeq CF^*(\mathcal{A}^5,(\tLvg(1),\nabla_{\Lambda}^{\imm})
\end{align*}
of $A_{\infty}$-modules over $CF^*(\mathcal{A}^5)$.

To that end, let $u\colon\Sigma\to X_{\tau}^5$ be a holomorphic map from a boundary punctured Riemann surface which contributes to $CF^*(\mathcal{L}^5,(\tLvg(1),\nabla_{\Lambda}^{\imm})$. If $u$ has a boundary component which passes through a $T^2$-neck of $\tLvg(1)$, then by the SFT neck-stretching argument used in the proofs of Proposition~\ref{m0svanishing} and Lemma~\ref{stripclassification}, we would be able to produce a holomorphic teardrop bounded by $\tLvg(1)$ contained in one of the charts $\widetilde{B}_{\ell,m}$. Remark~\ref{localisotopy} and Proposition~\ref{localtears} implies that no such teardrop can exist. The same argument, using SFT neck-stretching and the open mapping theorem, shows that if $u$ is a boundary punctured Riemann surface which constributes to the $A_{\infty}$-module structure on $CF^*(\mathcal{L}^5,(\widetilde{L}^5_{\mathrm{sm},-}\cup\widetilde{L}^5_{\mathrm{sm},+},\nabla_{\Lambda}))$, then it cannot have a corner at the $S^1$-switching component of $\widetilde{L}^5_{\mathrm{sm},-}\cup\widetilde{L}^5_{\mathrm{sm},+}$. In particular, we have a natural bijection between the sets of intersection points of $\mathcal{A}^5$ with the two Lagrangian branes in question, since $\mathcal{A}^5$ is contained away from the balls $\widetilde{B}_{\ell,m}$, and between the sets of holomoprhic disks contributing to the $CF^*(\mathcal{A}^5)$-module structures.
\end{proof}
One could also make this argument using Nohara--Ueda's split-generators~\cite{NU12} for the Fukaya category of the quintic. Combining this with Lemma~\ref{isotopyequiv} implies the following.
\begin{corollary}
There is a quasi-isomorphism of $A_{\infty}$-algebras
\begin{align*}
CF^*(\widetilde{L}^5_{\mathrm{sm},-}\cup\widetilde{L}^5_{\mathrm{sm},+},\nabla_{\Lambda})\simeq CF^*(\tLvg,\nabla)
\end{align*}
where the brane structures are as given in Definition~\ref{antisurgerybranes}. \qed
\end{corollary}
\begin{remark}
The quasi-isomorphism we have constructed tell us that there is an $A_{\infty}$-functor
\begin{align}\label{a-inf-functor}
(\tLvgsmooth,\nabla_{\Lambda})\to D^{\pi}\mathcal{F}(X^5)
\end{align}
from the $A_{\infty}$-category with one object and hom set $CF^*(\tLvgsmooth,\nabla_{\Lambda})$ to the split-closed derived Fukaya category. This functor takes the object in the former category to a direct summand of $(\tLvg,\nabla)$ mirror to the pushforward of a line bundle on a curve.
\end{remark}

\begin{corollary}
The functor~\eqref{a-inf-functor} carries $(\tLvgsmooth,\nabla_{\Lambda})$ to an object of the Fukaya category mirror to the pushforward of a line bundle on $C$, where $C$ is a generic line in the mirror quintic. The second Chern class of this object is represented by $-[C]$. Then there is an isomorphism of vector spaces
\begin{align*}
HF^*(\tLvgsmooth,\nabla_{\Lambda})\cong H^*(S^1\times S^2;\Lambda)
\end{align*}
over the Novikov field, where the right hand side refers to the ordinary cohomology of $S^1\times S^2$.
\end{corollary}

\section{Open Gromov--Witten theory for Lagrangian immersions}\label{ogw-immersed}
The approach to defining the open Gromov--Witten potential taken in~\cite{Fuk11, ST21, ST23, Han24b} is to correct that na{\"i}ve count $\mathfrak{m}_{-1}^J$ of pseudoholomorphic disks in a closed symplectic manifold $M$ with boundary on a Lagrangian $L$, which depends nontrivially on the almost complex structure $J$, by equipping $L$ with a bounding cochain. For ease of notation, we will only define the open Gromov--Witten potential when that $L$ is a graded spin Lagrangian immersion with clean self-intersections in a Calabi--Yau $3$-fold. This implies that the moduli spaces $\mathcal{M}_0(\beta;J)$ have virtual dimension $0$, so for any nonzero $\beta\in H_2(M,L;\mathbb{Z})$, we define
\begin{align}
\mathfrak{m}_{-1}^{\beta;J} \coloneqq\int_{\mathcal{M}_0(\beta;J)}1\in\mathbb{C}
\end{align}
and we set $\mathfrak{m}_{-1}^{0;J} = 0$. These define an element of the Novikov ring 
\begin{align}
\mathfrak{m}_{-1}^J &\coloneqq\sum_{\beta\in H_2(M,L;\mathbb{Z})}\mathfrak{m}_{-1}^{\beta;J} Q^{\omega(\beta)} \in\Lambda_+
\end{align}
of positive valuation. For embedded Lagrangian submanifolds, a construction of the open Gromov--Witten potential was proposed by Fukaya~\cite{Fuk11}. When $L$ is embedded, the integration pairing $\langle\cdot,\cdot\rangle$ on $CF^*(L) = \Omega^*(L)\widehat{\otimes}_{\mathbb{C}}\Lambda_{+}$, given by
\begin{align}
\langle\alpha,\beta\rangle = (-1)^{\deg\beta}\int_M\alpha\wedge\beta
\end{align}
is strictly cyclically symmetric, meaning that
\begin{align}\label{strictcyclic}
\langle\mathfrak{m}_k(\alpha_1\otimes\cdots\otimes\alpha_k),\alpha_0\rangle = (-1)^{\clubsuit_k}\langle\mathfrak{m}_k(\alpha_0\otimes\cdots\otimes\alpha_{k-1}),\alpha_k\rangle
\end{align}
where $\clubsuit_k = (\deg'\alpha_0)\sum_{j=1}^k(\deg'\alpha_j)$ and $\deg'\alpha_j = \deg\alpha_j+1$. When~\eqref{strictcyclic} holds, one can define a function $\Psi\colon\mathcal{M}(L)\to\Lambda_{+}$ on the moduli space of bounding cochains for $CF^*(L)$ by
\begin{align}
\Psi(b)\coloneqq\mathfrak{m}_{-1}+\sum_{k=0}^{\infty}\frac{1}{k+1}\langle\mathfrak{m}_k(b^{\otimes k}),b\rangle \, .
\end{align}
Strict cyclicity is used to show that $\Psi(b)$ only depends on the gauge-equivalence class of $b\in\mathcal{M}(L)$~\cite[Proposition 2.2]{Fuk11}.

The proof of~\eqref{strictcyclic} in~\cite{Fuk10} uses the existence of a system of Kuranishi structures on the moduli spaces $\mathcal{M}_{k+1}(\beta)$ of disks with $k+1$ boundary marked points representing the class $\beta\in H_2(M,L)$ that is compatible with the forgetful maps of boundary marked points~\cite[Corollary 3.1]{Fuk10}. Compatibility with forgetful maps is used there to establish that the Kuranishi structures are invariant under cyclic permutations of boundary marked points, since they are pulled back from the Kuranishi structures on $\mathcal{M}_0(\beta)$. We cannot speak of such forgetful maps for disks with boundary on an \textit{immersed} Lagrangian, since some of the boundary marked points of such a disk are mapped to self-intersection points of the immersion. As such, it is unclear how to construct cyclically symmetric perturbations of the moduli spaces of disks in the immersed setting, and so one cannot immediately extend the construction of the open Gromov--Witten potential in~\cite{Fuk11} to immersed Lagrangians.

The main results of~\cite{Han24b}, however, show that only the existence of a \textit{homotopy cyclic} pairing is required to define the open Gromov--Witten potential, and such a pairing can be extracted from the trace associated to the cyclic open-closed map (cf. Assumption~\ref{traceass}). Let $\mathcal{A}\coloneqq CF^*(L,\nabla)$ denote the curved $A_{\infty}$-algebra of a graded clean Lagrangian immersion $\iota_L\colon L\looparrowright M$ equipped with a rank one $U_{\Lambda}$-local system $\nabla$ as constructed in Appendix~\ref{immersedfloerappendix}. A trace as given in Assumption~\ref{traceass} can be thought of as a positive cyclic cocycle in $CC^*_{+}(\mathcal{A})$, from which one can construct a \textit{homotopy cyclic $\infty$-inner product}, or a strong homotopy inner product as it is called in~\cite{Cho08}, on $\mathcal{A}$ following~\cite{CL}.

An $\infty$-inner product on $\mathcal{A}$, in the sense of Tradler~\cite{Tra08}, is an $A_{\infty}$-bimodule homormophism $\psi\colon\mathcal{A}_{\Delta}\to\mathcal{A}^{\vee}$ from the diagonal bimodule over $\mathcal{A}_{\Delta}$ to the (linear) dual. For a more thorough review of these notions and of Hochschild and cyclic (co)homology for gapped filtered $A_{\infty}$-algebras, see~\cite[\S{2}]{Han24b}. Recall that such an $A_{\infty}$-bimodule homomorphism consists of linear maps $\lbrace\psi_{p,q}\colon\mathcal{A}^{\otimes p}\otimes\underline{\mathcal{A}}\otimes\mathcal{A}\to\mathcal{A}^{\vee}\rbrace_{p,q\in\mathbb{Z}_{\geq0}}$, where the bimodule factor has been underlined for readability, for which the associated map of tensors algebras commute with the $A_{\infty}$-bimodule structure maps. The precise meaning of this condition is reviewed in~\cite[(2.10)]{Han24b}. Such a $\psi$ is said to be homotopy cyclic if it is \textit{skew-symmetric}, \textit{closed}, and \textit{homologically nondegenerate}, which mean, respectively, that
\begin{itemize}
\item[•] for $\alpha\in\mathcal{A}^{\otimes p}$, $\beta\in\mathcal{A}^{\otimes q}$, and $v,w\in\mathcal{A}$, we have that
\begin{align}
\psi_{p,q}(\alpha\otimes\underline{v}\otimes\beta)(w) = (-1)^{\kappa}\psi_{q,p}(\beta,w,\alpha)(v)
\end{align}
where
\begin{align*}
\kappa = \left(\sum_{i=1}^p |a_i|'+|v|'\right)\cdot\left(\sum_{j=1}^q|b_j|'+|w|'\right)
\end{align*}
\item[•] for $a_1\otimes\cdots\otimes a_{\ell+1}\in\mathcal{A}^{\otimes\ell+1}$ and any triple $1\leq i<j<k\leq\ell+1$, we have that
\begin{align}
&(-1)^{\kappa_i}\psi(\cdots\otimes\underline{a_i}\otimes\cdots)(a_j)+(-1)^{\kappa_j}\psi(\cdots\otimes\underline{a_j}\otimes\cdots)(a_k) \nonumber \\
&+(-1)^{\kappa_k}\psi(\cdots\otimes\underline{a_k}\otimes\cdots\otimes)(a_i) = 0
\end{align}
where the sign is determined by
\[ \kappa_* = \left(|a_1|'+\cdots+|a_*|'\right)\cdot\left(|a_{*+1}|'+\cdots+|a_k|'\right)\]
and where the inputs are cyclically ordered; and

\item[•] the pairing on the de Rham cohomology $H^*(\mathcal{A},\mathfrak{m}_{1,0})$ is nondegenerate.
\end{itemize}
Having obtained a positive cyclic cocycle from Assumption~\ref{traceass}, one obtain a negative cyclic cocycle using the connecting homomorphism in the long exact sequence for cyclic cohomology (cf.~\cite[Lemma 2.8]{Han24b}). From $\psi_0$, the part of this cocycle that is dual to the inclusion of Hochschild cycles, one defines $\psi\colon\mathcal{A}_{\Delta}\to\mathcal{A}^{\vee}$ by setting
\begin{align}
\psi_{p,q}\coloneqq\psi_0(\alpha\otimes v\otimes\beta)(w) - \psi_0(\beta\otimes w\otimes\alpha)(v) \, . \label{infinityinner}
\end{align}
That the $\infty$-inner product $\psi$ given by~\eqref{infinityinner} is skew-symmetric and closed is immediate from the definition. For the cocycle associated to the cyclic open-closed map, one can check that the $\infty$-inner product $\psi$ obtained this way induces the Poincar{\'e} duality pairing on de Rham cohomology, which shows that it is homologically nondegenerate. The Poincar{\'e} duality pairing for an immersed Lagrangian comes from a chain-level integration pairing on the Fukaya $A_{\infty}$-algebra of a clean Lagrangian immersion defined as follows.
\begin{definition}
Let $\iota\colon L\to M$ be a clean Lagrangian immersion, where $L$ is closed. Let $A$ denote the index set for the components of $L\times_{\iota}L$. The integration pairing on $\overline{CF}^*(L)$ is defined as follows. If $a_{\pm}\in A$ are a pair of labels that are swapped under the natural involution, then for a pair of forms
\begin{align*}
\alpha_{\pm}\in\Omega^*(L_{a_{\pm}};\Theta_{a_{\pm}}^-) \,.
\end{align*}
we define
\begin{align}
\langle\alpha_{-},\alpha_{+}\rangle\coloneqq(-1)^{\deg\alpha_{+}}\int_{L_{a_{-}}}\alpha_{-}\wedge\alpha_{+} \,.
\end{align}
Set $\langle\alpha_{-},\alpha_{+}\rangle = 0$ if $\alpha_{\pm}$ are forms on switching components that are not related this way. This naturally extends to a pairing on $CF^*(L)$.
\end{definition}
Strict cyclic symmetry of such a pairing would imply that it induces a homotopy cyclic inner product $\psi\colon\mathcal{A}_{\Delta}\to\mathcal{A}^{\vee}$ for which all terms $\psi_{p,q}$ with $p>0$ or $q>0$ vanish~\cite{Cho08}. This morally explains why the open Gromov--Witten potential defined below generalizes the one defined in~\cite{Fuk11}, as discussed in~\cite[\S{6}]{Han24b}.
\begin{definition}[{\cite{Han24b}}]\label{infinityogwdef}
The $(\infty-)$open Gromov--Witten potential on $CF^*(L,\nabla)$ is the function $\Psi\colon\mathcal{M}(L,\nabla)\to\Lambda_{+}$ given by
\begin{align}
\Psi_J(b)&\coloneqq\mathfrak{m}_{-1}^J+\Psi'_J(b) \\
&\coloneqq\mathfrak{m}_{-1}^J+\sum_{N=0}^{\infty}\sum_{p+q+k = N}\frac{1}{N+1}\psi_{p,q}(b^{\otimes p}\otimes\mathfrak{m}_k(b^{\otimes k})\otimes b^{\otimes q})(b)
\end{align}
where $\psi$ is the homotopy cyclic infinity inner product associated to the cyclic open-closed trace (Assumption~\ref{traceass}).
\end{definition}
The gauge-invariance of $\Psi$ in characteristic $0$ follows from~\cite[Theorem 2.17]{Han24b}. For general Lagrangians $L$, it will still not be the case that $\Psi(b)$ is independent of the almost complex structure $J$ used to define it. Instead, $\Psi$ will only satisfy a wall-crossing formula involving closed $\underline{J}$-holomorphic curves, where $\underline{J} = \lbrace J_t\rbrace_{t\in[0,1]}$ is a $1$-parameter family of almost complex structures. This wall-crossing formula is most cleanly stated using pseudo-isotopies of $A_{\infty}$-algebras, as reviewed in \S\ref{pseudoisotopiessection}. It is most convenient for us to think of a pseudo-isotopy as  a (gapped filtered) $A_{\infty}$-structure on the space of differential forms on $[0,1]\times(L\times_{\iota_L} L)$. The structure maps are denoted $\lbrace\widetilde{\mathfrak{m}}_k\rbrace_{k\geq0}$. 
\begin{remark}[Bounding cochains on a pseudo-isotopy]
A pseudo-isotopy determines $A_{\infty}$-quasi-isomorphisms
\begin{align*}
\mathfrak{c}^t\colon(CF^*(L,\nabla),\lbrace\mathfrak{m}_k^{J_0}\rbrace_{k\geq0})\to(CF^*(L,\nabla),\lbrace\mathfrak{m}_k^{J_t}\rbrace_{k\geq0})
\end{align*}
for all $t\in[0,1]$. An $A_{\infty}$-quasi-isomorphism is comprised of a sequence of linear maps
\begin{align*}
\mathfrak{c}^t_k\colon(CF^*(L,\nabla)[1])^{\otimes k}\to CF^*(L,\nabla)[1]
\end{align*}
of degree $0$ for all $k\geq0$. We set $b_t = \mathfrak{c}^t_*(b_0)$ for all $t\geq0$, where the pushforward of bounding cochains is defined by $\mathfrak{c}^t_*(b_0) = \sum_{k=0}^{\infty}\mathfrak{c}^t_k(b_0)$. This path of bounding cochains $\widetilde{b} = \lbrace b_t\rbrace_{t\in[0,1]}$ is a bounding cochain for the pseudo-isotopy, by the so-called pointwise condition ~\cite[Definition 21.27]{FOOO20}.
\end{remark}
\begin{theorem}[Wall-crossing]
For a path $\underline{J}$ as above and any path of bounding cochains $\widetilde{b} = \lbrace b_t\rbrace_{t\in[0,1]}$ as above, we have that
\begin{align}
\Psi_{J_0}(b_0) = \Psi_{J_1}(b_1)+\widetilde{GW}(L) \, .
\end{align}
The number $\widetilde{GW}(L)\in\Lambda_0$ is defined in~\eqref{wallcrossingterm} as the count of $\underline{J}$-holomorphic closed rational curves intersecting $L$ in a point.
\end{theorem}
\begin{proof}
The proof closely follows the proof of~\cite[Theorem 5.2]{Han24b}, with some modifications owing to the different chain-level model used in the present context. Recall that the terms of $\Psi'_{J_i}(b_i)$ are defined in terms of the open-closed map $\mathcal{OC}_0$.
Let $\widetilde{\mathcal{A}}$ denote the pseudo-isotopy defined using $\underline{J}$. By Assumption~\ref{traceass}, there is an $\infty$-inner product $\widetilde{\psi}$ on $\widetilde{\mathcal{A}}$ defined using the trace associated to the cyclic open-closed map on a pseudo-isotopy.
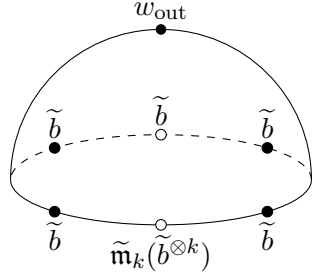
\begin{figure}
\begin{tikzpicture}
\node (N) at (0,0) {};
\node[anchor = south] (E) at (0,2) {$w_{\out}$};

\begin{scope}[very thick,decoration={
	markings,
	mark=at position 0.75 with {\arrow{>}}}]
\draw[postaction={decorate}, thick] (N) -- (E);
\end{scope}
\fill[white] (0,0) circle (2cm);
\fill[black] (0,2) circle (2pt);

\draw (2,0) arc (0:180:2);
\draw (-2,0) arc (180:360:2 and 0.6);
\draw[dashed] (2,0) arc (0:180:2 and 0.6);

\coordinate (A) at (270: 2 and 0.6);
\node[anchor = north] at (270: 2 and 0.6) {$\widetilde{\mathfrak{m}}_k(\widetilde{b}^{\otimes k})$};
\draw[fill = white] (A) circle (2pt);

\coordinate (B) at (225: 2 and 0.6);
\node[anchor = north] at (225: 2 and 0.6) {$\widetilde{b}$};
\draw[fill = black] (B) circle (2pt);

\coordinate (C) at (315:2 and 0.6);
\node[anchor = north] at (315: 2 and 0.6) {$\widetilde{b}$};
\draw[fill = black] (C) circle (2pt);

\coordinate (D) at (90:2 and 0.6);
\draw[fill = white] (D) circle (2pt);
\node[anchor = south] at (90:2 and 0.6) {$\widetilde{b}$};

\coordinate (E) at (45:2 and 0.6);
\draw[fill = black] (E) circle (2pt);
\node[anchor = south] at (45:2 and 0.6) {$\widetilde{b}$};

\coordinate (F) at (135:2 and 0.6);
\draw[fill = black] (F) circle (2pt);
\node[anchor = south] at (135:2 and 0.6) {$\widetilde{b}$};
\end{tikzpicture}\caption{Elements of~\eqref{ocmodpseudo}, with boundary marked points labeled by the relevant inputs. The marked points corresponding to the input and output of $\widetilde{\psi}$ are the white dots.}
\end{figure}
Applying Stokes' theorem~\cite[Proposition 4.2]{Fuk10} to the moduli spaces $\widetilde{\mathcal{M}}_{k,1}(L;\vec{a};\beta;\underline{J})$ of~\eqref{ocmodpseudo} and the definition of the $\infty$-inner products shows that
\begin{align}
&\Psi'_1(b_1)-\Psi'_{-1}(b_{-}) \\
&=\sum_{N=0}^{\infty}\sum_{\substack{p+q+k  = N \\ k_1+k_2 = k+1}}\frac{1}{N+1}\sum_{r+s = k_1-1}\widetilde{\psi}_{p,q}(\widetilde{b}^{\otimes p}\otimes\underline{\widetilde{\mathfrak{m}}_{k_1}(\widetilde{b}^{\otimes r}\otimes\widetilde{\mathfrak{m}}_{k_2}(\widetilde{b}^{\otimes k_2})\otimes\widetilde{b}^{\otimes s})}\otimes\widetilde{b}^{\otimes q})(\widetilde{b}) \label{ksplit}\\
&+\sum_{N=0}^{\infty}\sum_{\substack{p+q+k  = N \\ k_1+k_2 = k+1}}\frac{1}{N+1}\sum_{r+s = p-1}\widetilde{\psi}_{p,q}(\widetilde{b}^{\otimes r}\otimes\widetilde{\mathfrak{m}}_{k_2}(\widetilde{b}^{\otimes k_2})\otimes\widetilde{b}^{\otimes s}\otimes\underline{\mathfrak{m}_{k_1}(\widetilde{b}^{\otimes k_1})}\otimes\widetilde{b}^{\otimes q})(\widetilde{b}) \label{psplit}\\
&+\sum_{N=0}^{\infty}\sum_{\substack{p+q+k  = N \\ k_1+k_2 = k+1}}\frac{1}{N+1}\sum_{r+s = q-1}\widetilde{\psi}_{p,q}(\widetilde{b}^{\otimes p}\otimes\underline{\widetilde{\mathfrak{m}}_{k_1}(\widetilde{b}^{\otimes k_1})}\otimes\widetilde{b}^{\otimes r}\otimes\widetilde{\mathfrak{m}}_{k_2}(\widetilde{b}^{\otimes k_2})\otimes\widetilde{b}^{\otimes s})(\widetilde{b}) \label{qsplit}\\
&+\sum_{N=0}^{\infty}\sum_{\substack{p+q+k  = N \\ k_1+k_2 = k+1}}\frac{k_2}{N+1}\widetilde{\psi}_{p,q}(\widetilde{b}^{\otimes p}\otimes\underline{\widetilde{\mathfrak{m}}_{k_1}(\widetilde{b}^{\otimes k_1})}\otimes\widetilde{b}^{\otimes q})(\widetilde{\mathfrak{m}}_{k_2}(\widetilde{b}^{\otimes k_2})) \, . \label{output}
\end{align}
Because the value of $\widetilde{\psi}$ on inputs of the form under consideration are expressed as the difference
\[ \widetilde{\mathcal{OC}}_0(\widetilde{b}^{\otimes p}\otimes\widetilde{\mathfrak{m}}_k(\widetilde{b}^{\otimes k})\otimes\widetilde{b}^{\otimes q}\otimes\widetilde{b})-\widetilde{\mathcal{OC}}_0(\widetilde{b}^{\otimes p}\otimes\widetilde{b}\otimes\widetilde{b}^{\otimes q}\otimes\widetilde{\mathfrak{m}}_k(\widetilde{b}^{\otimes k}))\]
we prove the identity above by applying Stokes' theorem to two copies of the moduli space~\eqref{ocmodpseudo}. Since $\widetilde{\psi}_{p,q}$ is defined by taking a difference corresponding to these two moduli spaces, the contributions of Figure~\eqref{cancelingstrata} cancel, and thus they do not contribute to the sum above.

We can rewrite the sum of~\eqref{ksplit},~\eqref{psplit}, and~\eqref{qsplit} as
\begin{align}
&\sum_{N=0}^{\infty}\sum_{\substack{p+q+k  = N \\ k_1+k_2 = k+1}}\frac{N+1-k_2}{N+1}\widetilde{\psi}_{p,q}(\widetilde{b}^{\otimes p}\otimes\underline{\widetilde{\mathfrak{m}}_{k_1}(\widetilde{b}^{\otimes k_1})}\otimes\widetilde{b}^{\otimes q})(\widetilde{\mathfrak{m}}_{k_2}(\widetilde{b}^{\otimes k_2})) \label{clsum}
\end{align}
using~\cite[Lemma 2.15]{Han24b}. By the Maurer--Cartan eqation, the sum of~\eqref{output} and~\eqref{clsum} gives
\begin{align}
\sum_{p,q\geq0}\widetilde{\psi}_{p,q}(\widetilde{b}^{\otimes p}\otimes\underline{\widetilde{\mathfrak{m}}_0}\otimes\widetilde{b}^{\otimes q})(\widetilde{\mathfrak{m}}_0)\, . \label{errortermv0}
\end{align}
By passing to a canonical model for $\widetilde{\mathcal{A}}$ in which the $\infty$-inner product coincides with the Poincar{\'e} pairing, which exists by~\cite{CL}, we can rewrite this as
\begin{align}
\mathcal{OC}_0(\widetilde{\mathfrak{m}}_2(\widetilde{\mathfrak{m}}_0,\widetilde{\mathfrak{m}}_0)) = \langle\widetilde{\mathfrak{m}}_0,\widetilde{\mathfrak{m}}_0\rangle \label{errortermv2}
\end{align}
where the pairing on the right hand side is an integration pairing on $CF^*([0,1]\times L)$. Here we have used the fact that for degree reasons, only constant disks can contribute to $\widetilde{\mathfrak{m}}_2(\widetilde{\mathfrak{m}}_0,\widetilde{\mathfrak{m}}_0)$.

We can also analyze the leading terms $\mathfrak{m}_{-1}^{J_i}$ of $\Psi_{J_i}(b_i)$ simultaneously. To that end, consider the moduli spaces $\widetilde{\mathcal{M}}_{-1}(\beta;\underline{J})$. One type of boundary stratum in this moduli space is given by the fiber product of two moduli spaces of the form $\widetilde{\mathcal{M}}_{0}(\vec{a};\beta;\underline{J})$ along the evaluation map at the boundary marked point. This evaluation map can have a switching component of the immersion as its codomain, which is the only new geometric subtlety as compared to the case of embedded Lagrangians. Notice that the contributions of these terms are precisely given by~\eqref{errortermv2}, with the opposite sign. The contributions of the remaining boundary strata are given by the wall-crossing term $\widetilde{GW}(L)$ that counts closed rational curves, just as for embedded Lagrangians.
\end{proof}
\begin{figure}
\begin{tikzpicture}
\begin{scope}
\node (N) at (0,0) {};
\node[anchor = south] (E) at (0,2) {$w_{\out}$};

\begin{scope}[very thick,decoration={
	markings,
	mark=at position 0.75 with {\arrow{>}}}]
\draw[postaction={decorate}, thick] (N) -- (E);
\end{scope}
\fill[white] (0,0) circle (2cm);
\fill[black] (0,2) circle (2pt);

\draw (2,0) arc (0:180:2);
\draw (-2,0) arc (180:360:2 and 0.6);
\draw[dashed] (2,0) arc (0:180:2 and 0.6);

\begin{scope}[xshift = 3cm]
\draw (1,0) arc (0:180: 1 and 1);
\draw (-1,0) arc (180:360:1 and 0.3);
\draw[dashed] (1,0) arc (0:180:1 and 0.3);
\end{scope}

\coordinate (A) at (270: 2 and 0.6);
\node[anchor = north] at (270: 2 and 0.6) {$\widetilde{b}$};
\draw[fill = black] (A) circle (2pt);

\coordinate (B) at (225: 2 and 0.6);
\node[anchor = north] at (225: 2 and 0.6) {$\widetilde{b}$};
\draw[fill = black] (B) circle (2pt);

\coordinate (C) at (315:2 and 0.6);
\node[anchor = north] at (315: 2 and 0.6) {$\widetilde{b}$};
\draw[fill = black] (C) circle (2pt);

\coordinate (D) at (90:2 and 0.6);
\draw[fill = white] (D) circle (2pt);
\node[anchor = south] at (90:2 and 0.6) {$\widetilde{b}$};

\coordinate (G) at (0: 2 and 0.6);
\draw[fill = black] (G) circle (2pt);

\begin{scope}[xshift = 3cm]
\coordinate (E) at (90:1 and 0.3);
\draw[fill = black] (E) circle (2pt);
\node[anchor = south] at (90:1 and 0.3) {$\widetilde{b}$};

\coordinate (F) at (270:1 and 0.3);
\draw[fill = white] (F) circle (2pt);
\node[anchor = north] at (270:1 and 0.3) {$\widetilde{\mathfrak{m}}_k(\widetilde{b}^{\otimes k})$};
\end{scope}
\end{scope}
\end{tikzpicture}
\caption{Elements of the boundary stratum contributing to term~\eqref{ksplit}.}\label{ksplitstrata}
\end{figure}
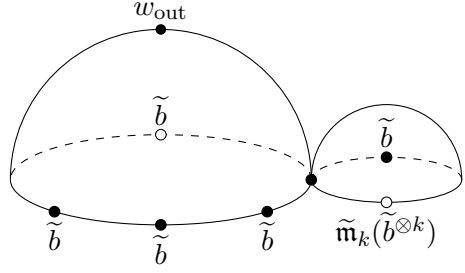
 
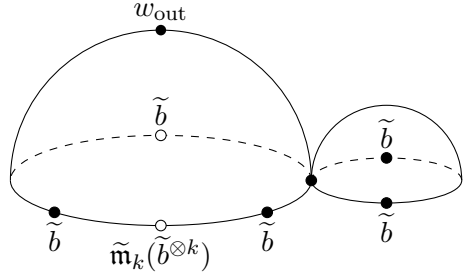
\begin{figure}
\begin{tikzpicture}
\begin{scope}
\node (N) at (0,0) {};
\node[anchor = south] (E) at (0,2) {$w_{\out}$};

\begin{scope}[very thick,decoration={
	markings,
	mark=at position 0.75 with {\arrow{>}}}]
\draw[postaction={decorate}, thick] (N) -- (E);
\end{scope}
\fill[white] (0,0) circle (2cm);
\fill[black] (0,2) circle (2pt);

\draw (2,0) arc (0:180:2);
\draw (-2,0) arc (180:360:2 and 0.6);
\draw[dashed] (2,0) arc (0:180:2 and 0.6);

\begin{scope}[xshift = 3cm]
\draw (1,0) arc (0:180: 1 and 1);
\draw (-1,0) arc (180:360:1 and 0.3);
\draw[dashed] (1,0) arc (0:180:1 and 0.3);
\end{scope}

\coordinate (A) at (270: 2 and 0.6);
\node[anchor = north] at (270: 2 and 0.6) {$\widetilde{\mathfrak{m}}_k(\widetilde{b}^{\otimes k})$};
\draw[fill = white] (A) circle (2pt);

\coordinate (B) at (225: 2 and 0.6);
\node[anchor = north] at (225: 2 and 0.6) {$\widetilde{b}$};
\draw[fill = black] (B) circle (2pt);

\coordinate (C) at (315:2 and 0.6);
\node[anchor = north] at (315: 2 and 0.6) {$\widetilde{b}$};
\draw[fill = black] (C) circle (2pt);

\coordinate (D) at (90:2 and 0.6);
\draw[fill = white] (D) circle (2pt);
\node[anchor = south] at (90:2 and 0.6) {$\widetilde{b}$};

\coordinate (G) at (0: 2 and 0.6);
\draw[fill = black] (G) circle (2pt);

\begin{scope}[xshift = 3cm]
\coordinate (E) at (90:1 and 0.3);
\draw[fill = black] (E) circle (2pt);
\node[anchor = south] at (90:1 and 0.3) {$\widetilde{b}$};

\coordinate (F) at (270:1 and 0.3);
\draw[fill = black] (F) circle (2pt);
\node[anchor = north] at (270:1 and 0.3) {$\widetilde{b}$};
\end{scope}
\end{scope}
\end{tikzpicture}
\caption{Boundary strata contributing to terms~\eqref{psplit} and~\eqref{qsplit}.}\label{pqsplitstrata}
\end{figure}

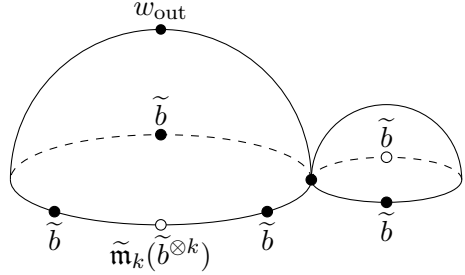
\begin{figure}
\begin{tikzpicture}
\begin{scope}
\node (N) at (0,0) {};
\node[anchor = south] (E) at (0,2) {$w_{\out}$};

\begin{scope}[very thick,decoration={
	markings,
	mark=at position 0.75 with {\arrow{>}}}]
\draw[postaction={decorate}, thick] (N) -- (E);
\end{scope}
\fill[white] (0,0) circle (2cm);
\fill[black] (0,2) circle (2pt);

\draw (2,0) arc (0:180:2);
\draw (-2,0) arc (180:360:2 and 0.6);
\draw[dashed] (2,0) arc (0:180:2 and 0.6);

\begin{scope}[xshift = 3cm]
\draw (1,0) arc (0:180: 1 and 1);
\draw (-1,0) arc (180:360:1 and 0.3);
\draw[dashed] (1,0) arc (0:180:1 and 0.3);
\end{scope}

\coordinate (A) at (270: 2 and 0.6);
\node[anchor = north] at (270: 2 and 0.6) {$\widetilde{\mathfrak{m}}_k(\widetilde{b}^{\otimes k})$};
\draw[fill = white] (A) circle (2pt);

\coordinate (B) at (225: 2 and 0.6);
\node[anchor = north] at (225: 2 and 0.6) {$\widetilde{b}$};
\draw[fill = black] (B) circle (2pt);

\coordinate (C) at (315:2 and 0.6);
\node[anchor = north] at (315: 2 and 0.6) {$\widetilde{b}$};
\draw[fill = black] (C) circle (2pt);

\coordinate (D) at (90:2 and 0.6);
\draw[fill = black] (D) circle (2pt);
\node[anchor = south] at (90:2 and 0.6) {$\widetilde{b}$};

\coordinate (G) at (0: 2 and 0.6);
\draw[fill = black] (G) circle (2pt);

\begin{scope}[xshift = 3cm]
\coordinate (E) at (90:1 and 0.3);
\draw[fill = white] (E) circle (2pt);
\node[anchor = south] at (90:1 and 0.3) {$\widetilde{b}$};

\coordinate (F) at (270:1 and 0.3);
\draw[fill = black] (F) circle (2pt);
\node[anchor = north] at (270:1 and 0.3) {$\widetilde{b}$};
\end{scope}
\end{scope}
\end{tikzpicture}
\caption{Boundary stratum contributing to the term~\eqref{output}.}\label{outputstrata}
\end{figure}

\begin{figure}
\begin{tikzpicture}
\begin{scope}
\node (N) at (0,0) {};
\node[anchor = south] (E) at (0,2) {$w_{\out}$};

\begin{scope}[very thick,decoration={
	markings,
	mark=at position 0.75 with {\arrow{>}}}]
\draw[postaction={decorate}, thick] (N) -- (E);
\end{scope}
\fill[white] (0,0) circle (2cm);
\fill[black] (0,2) circle (2pt);

\draw (2,0) arc (0:180:2);
\draw (-2,0) arc (180:360:2 and 0.6);
\draw[dashed] (2,0) arc (0:180:2 and 0.6);

\begin{scope}[xshift = 3cm]
\draw (1,0) arc (0:180: 1 and 1);
\draw (-1,0) arc (180:360:1 and 0.3);
\draw[dashed] (1,0) arc (0:180:1 and 0.3);
\end{scope}

\coordinate (A) at (270: 2 and 0.6);
\node[anchor = north] at (270: 2 and 0.6) {$\widetilde{b}$};
\draw[fill = black] (A) circle (2pt);

\coordinate (B) at (225: 2 and 0.6);
\node[anchor = north] at (225: 2 and 0.6) {$\widetilde{b}$};
\draw[fill = black] (B) circle (2pt);

\coordinate (C) at (315:2 and 0.6);
\node[anchor = north] at (315: 2 and 0.6) {$\widetilde{b}$};
\draw[fill = black] (C) circle (2pt);

\coordinate (D) at (90:2 and 0.6);
\draw[fill = black] (D) circle (2pt);
\node[anchor = south] at (90:2 and 0.6) {$\widetilde{b}$};

\coordinate (G) at (0: 2 and 0.6);
\draw[fill = black] (G) circle (2pt);

\begin{scope}[xshift = 3cm]
\coordinate (E) at (90:1 and 0.3);
\draw[fill = white] (E) circle (2pt);
\node[anchor = south] at (90:1 and 0.3) {$\widetilde{b}$};

\coordinate (F) at (270:1 and 0.3);
\draw[fill = white] (F) circle (2pt);
\node[anchor = north] at (270:1 and 0.3) {$\widetilde{\mathfrak{m}}_k(\widetilde{b}^{\otimes k})$};
\end{scope}
\end{scope}
\end{tikzpicture}
\caption{Canceling boundary strata.}\label{cancelingstrata}
\end{figure}
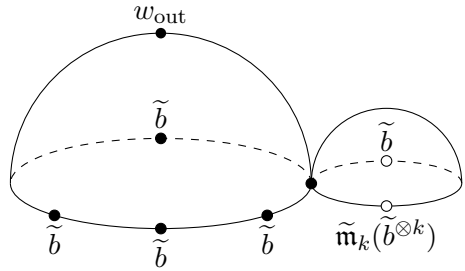
When $\iota\colon L\to M$ is nullhomologous, we can obtain an invariant from $\Psi^{J}(b)$ by considering a smooth singular $4$-chain $\Gamma\in H_4(M,L;\mathbb{Z})$ with boundary $\partial\Gamma = L$.
\begin{corollary}\label{bounding-4chain}
If $[L] = 0$ in $H_3(M,L;\mathbb{Z})$ and $\Gamma\in H_4(M,L;\mathbb{Z})$ is a bounding $4$-chain for $L$, then
\begin{align}
\Psi_{L,\Gamma}(b)\coloneqq\Psi^{J}(b)-\int_{\Gamma}\mathfrak{m}^J_{\emptyset}
\end{align}
is independent of the almost complex structure $J$, where $\mathfrak{m}^J_{\emptyset}$ is defined in~\eqref{closedoperators}.
\end{corollary}
\begin{proof}
Consider the pseudo-isotopy associated to a path $\underline{J}$ of $\omega$-compatible almost complex structures, and let $p_M\colon[0,1]\times M\to M$ denote the projection. By applying Stokes' theorem to the moduli spaces $\widetilde{\mathcal{M}}^{cl}_1(\beta)$ defined in \S\ref{wallcrossingsubsection}, we find that
\begin{align}
\mathfrak{m}_{\emptyset}^{J_1}-\mathfrak{m}_{\emptyset}^{J_0} = -d((p_M)_*\widetilde{\mathfrak{m}}_{\emptyset}) \, . \label{stokesclosedoperators}
\end{align}
Integrating both sides of~\eqref{stokesclosedoperators} over $\Gamma$ and applying Stokes' theorem shows that
\begin{align}
\int_{\Gamma}\mathfrak{m}_{\emptyset}^{J_1}-\int_{\Gamma}\mathfrak{m}_{\emptyset}^{J_0} = -\widetilde{GW} \, . 
\end{align}
\end{proof}

\section{Background on Hodge structures and cyclic homology}\label{hodgecyclicsect}
In this section, we will review some background material on variations of semi-infinite Hodge structure (VSHS), including those obtained from the negative cyclic homology of an $A_{\infty}$-category, and the VSHSs appearing in closed string mirror symmetry, following~\cite{GPS15}.

Consider a complete discrete valuation ring $R$ with valuation $\nu$ and maximal ideal $\mathfrak{m}$ and residue field $\mathbb{C}$. Denote by $\mathbb{K}$ its field of fractions. We call $\mathcal{M}\coloneqq\Spec\mathbb{K}$ a formal punctured disk. We also denote by $q$ an element $q\in R$ with $\nu(q) = 1$. This determines an isomorhpism $R\cong\mathbb{C}[[q]]$. This can also be thought of as a coordinate on $\mathcal{M}$. Also define $\mathcal{O}_{\mathcal{M}}\coloneqq\mathbb{K}$ and $T\mathcal{M}\coloneqq\Der_{\mathbb{C}}\mathcal{M}$.

Variations of semi-infinite Hodge structure, introduced by Barranikov~\cite{Bar01}, can roughly be thought of as variations of Hodge structure, in the classical sense, without an integral lattice, provided that $\mathcal{M}$ trivial grading. It is nonetheless more natural to describe the VSHSs associated to negative cyclic homology and quantum cohomology slightly differently.
\begin{definition}\label{VSHSdef}
Let $u$ be a formal variable of degree $2$, and consider the ring $\mathbb{K}[[u]]$. For any $f\in\mathbb{K}[[u]]$, we define the element
\begin{align}
f^{\wedge}(u)\coloneqq f(-u) \,. \label{conjugate}
\end{align}

A ($\mathbb{Z}$-graded unpolarized) \textit{VSHS} over the formal punctured disk $\mathcal{M}$ is a pair $\mathcal{H} = (\mathcal{E},\nabla)$ where:
\begin{itemize}
\item[•] $\mathcal{E}$ is a graded, finitely-generated free $\mathbb{K}[[u]]$-module;
\item[•] $\nabla$ is a flat connection
\begin{align*}
\nabla\colon T\mathcal{M}\otimes\mathcal{E}\to u^{-1}\mathcal{E}
\end{align*}
of degree $0$.
\end{itemize}
Furthermore, a \textit{polarization} of dimension $n\in\mathbb{Z}/2$ for $\mathcal{H}$ is a pairing
\begin{align*}
(\cdot,\cdot)\colon\mathcal{E}\otimes\mathcal{E}\to\mathbb{K}[[u]]
\end{align*}
of degree $0$ which is:
\begin{itemize}
\item[•] sesquilinear:
\begin{align*}
(s_1+s_2,t) &= (s_1,t)+(s_2,t) \\
(t,s_1+s_2) &= (t,s_1)+(t,s_2)
\end{align*}
and
\begin{align*}
(fs_1,s_2) &= f(s_1,s_2) = (s_1,f^{\wedge}s_2)
\end{align*}
for all $f\in\mathbb{K}[[u]]$ and $s_1,s_2,t\in\mathcal{E}$;
\item[•] covariantly constant:
\begin{align*}
X(s_1,s_2) = (\nabla_X s_1,s_2)+(s_1,\nabla_X s_2)
\end{align*}
for all $X\in T\mathcal{M}$; and
\item[•] graded symmetric:
\begin{align*}
(s_1,s_2) = (-1)^{n+\deg s_1}(s_2,s_1)^{\wedge} \, .
\end{align*}
\end{itemize}
Moreover, we require that the induced pairing of $\mathbb{K}$-modules
\begin{align*}
\mathcal{E}/u\mathcal{E}\otimes_{\mathbb{K}}\mathcal{E}/u\mathcal{E}\to\mathbb{K}
\end{align*}
is nondegenerate.
\end{definition}
The relation between the notion of a VSHS and classical variations of Hodge structure is established in~\cite{GPS15}.
\begin{lemma}[{\cite[Lemma 2.7]{GPS15}}]\label{VSHSlemma}
A $\mathbb{Z}$-graded unpolarized  VSHS $\mathcal{H} = (\mathcal{E},\nabla)$ over a punctured formal disk $\mathcal{M}$ is equivalent to the following data:
\begin{itemize}
\item[•] a free, finite rank, $\mathbb{Z}/2$-graded $\mathbb{K}$-module $\mathcal{V}\coloneqq\mathcal{V}_{\even}\oplus\mathcal{V}_{\odd}$;
\item[•] a flat connection $\nabla$ on each $\mathcal{V}_{\sigma}$; and
\item[•] a pair of decreasing Hodge filtrations $\lbrace F^{\geq p}\mathcal{V}_{\even}\rbrace$ and $\lbrace F^{\geq p-\frac{1}{2}}\mathcal{V}_{\odd}\rbrace$ which satisfy Griffiths transversality:
\begin{align*}
\nabla F^{\geq p}\mathcal{V}_{\sigma}\subset F^{\geq p-1}\mathcal{V}_{\sigma}\, .
\end{align*}
\end{itemize}
An $n$-dimensional polarization on $\mathcal{H}$ is equivalent to a pair of covariantly constant bilinear pairings
\begin{align*}
(\cdot,\cdot)\colon\mathcal{V}_{\sigma}\otimes\mathcal{V}_{\sigma}\to\mathbb{K}
\end{align*}
such that $(\alpha,\beta) = (-1)^n(\beta,\alpha)$, and for which
\begin{align*}
(F^{\geq p}\mathcal{V}_{\sigma},F^{\geq q}\mathcal{V}_{\sigma}) = 0
\end{align*}
if $p+q>0$. Furthermore, the induced pairing on the associated graded modules
\begin{align*}
\Gr_F^p\mathcal{V}_{\sigma}\otimes\Gr_F^{-p}\mathcal{V}_{\sigma}\to\mathbb{K}
\end{align*}
is nondegenerate for all $p$. \qed
\end{lemma}
It will be illuminating to sketch the proof of one direction of this correspondence following the proof of~\cite[Lemma 2.7]{GPS15}.
\begin{proof}[Proof sketch]
If $\mathcal{H} = (\mathcal{E},\nabla)$ is as in Definition~\ref{VSHSdef}, then the construction of the $\mathbb{K}$-module and Hodge filtration of Lemma~\ref{VSHSlemma} can be summarized as follows. We define a $\mathbb{K}[u,u^{-1}]$-module by setting $\widetilde{\mathcal{E}}\coloneqq\mathcal{E}\otimes_{\mathbb{K}[[u]]}\mathbb{K}((u))$. Multiplication by $u$ induces isomorphisms
\begin{align}
\widetilde{\mathcal{E}}_k\xrightarrow{u\cdot}\widetilde{\mathcal{E}}_{k+2} \label{periodicity}
\end{align}
so we set
\begin{align*}
\mathcal{V}_{[k]}\coloneqq\widetilde{\mathcal{E}}_k \, .
\end{align*}
The Hodge filtration on $\mathcal{V}_{[k]}$ is given by powers of $u$:
\begin{align*}
F^{\geq p-\frac{k}{2}}\mathcal{V}_{[k]}\coloneqq\left(u^{\geq p}\cdot\mathcal{E}\right)_k\subset\widetilde{\mathcal{E}}_k \,.
\end{align*}
The connection on $\mathcal{V}_{[k]}$ is obtained from the connection on $\mathcal{E}$, and Griffiths transversality follows because $\nabla$ carries $\mathcal{E}$ to $u^{-1}\mathcal{E}$. The pairing is inherited from the polarization on $\mathcal{E}$, up to rescaling by a constant prefactor so that it respects~\eqref{periodicity}.
\end{proof}

Let $\mathcal{A}$ be a strictly unital uncurved $\mathbb{K}$-linear $A_{\infty}$-category. Further assume that $\mathcal{A}$ is proper and that it carries a weak proper Calabi--Yau structure.
\begin{remark}
When $\mathcal{A}$ is the Fukaya category of unobstructed Lagrangian branes, the weak proper Calabi--Yau structure is given by the negative cyclic open-closed map under Assumption~\ref{traceass}. On homology, this amounts to the Poincar{\'e} pairing, as noted in the previous section. If $\mathcal{A}$ is a $dg$-enhancement of the derived category of a smooth Calabi--Yau variety over $\mathbb{C}$, the weak proper Calabi--Yau structure is determined by Serre duality and a choice of volume form. For this purpose, we will always use the Hodge-theoretically normalized volume form, in the terminology of~\cite{CK99} (see also~\cite[\S{2.4}]{GPS15})
\end{remark}
In this setting, a VSHS on the negative cyclic homology $HC_*^{-}(\mathcal{A})$ is constructed by Sheridan~\cite{She20}. The description of the VSHS in following Theorem uses the characterization of Definition~\ref{VSHSdef}.
\begin{theorem}[\cite{She20}]
Recall that $HC_*^{-}(\mathcal{A})$ is a $\mathbb{K}[[u]]$-module. This can be equipped with a flat connection $\nabla^{\GGM}$ called the \textit{Getzler--Gauss--Manin} connection. There is a $\mathbb{K}[[u]]$ sesquilinear pairing
\begin{align*}
\langle\cdot,\cdot\rangle_{res}\colon HC_*^{-}(\mathcal{A})\otimes HC_*{-}(\mathcal{A})\to\mathbb{K}[[u]]
\end{align*}
called the \textit{Mukai pairing}, which is graded symmetric of dimension $n$ when $\mathcal{A}$ admits an $n$-dimensional weak proper Calabi--Yau structure.
\end{theorem}
The negative cyclic homology is only a \textit{pre-VSHS} \textit{a priori}, but for the Fukaya category or derived category, in the geometric setting to be considered below, the results of~\cite{GPS15} imply that it is in fact a VSHS. In the case of the Fukaya category, this requires Assumption~\ref{traceass}, which is used to construct a weak proper Calabi--Yau structure.

Extracting (closed string) Hodge-theoretic mirror symmetry~\cite{Mor93} from homological mirror symmetry amounts to comparing this VSHS to geometrically defined VSHSs associated to quantum cohomology and to a suitable family of Calabi--Yau varieties near a large complex structure limit point.

Let $(X,\omega)$ be a connected integral symplectic Calabi--Yau $3$-fold, meaning that $[\omega] = H^2(X;\mathbb{Z})$ and $c_1(TX) = 0$. Also let $R_A\coloneqq\mathbb{C}[[Q]]$ and $\mathbb{K}_A\coloneqq\mathbb{C}((Q))$, so that $R_A$ is a subring of the Novikov ring $\Lambda_0$.
\begin{definition}\label{AVSHS}
The small $A$-model VSHS, denoted $\mathcal{H}^A(X,\omega)\coloneqq(\mathcal{E},\nabla,(\cdot,\cdot))$ is given by the data
\begin{align*}
\mathcal{E} &\coloneqq H^*(X;\mathbb{C})\widehat{\otimes}_{\mathbb{C}}\mathbb{K}_A[[u]][n] \\
\nabla_{Q\partial_Q}\alpha &\coloneqq Q\partial_Q\alpha-u^{-1}[\omega]\star\alpha \\
(\alpha,\beta) &\coloneqq\int_X\alpha\cup\beta^{\wedge}
\end{align*}
where $[\omega]\star\alpha$ denotes the small quantum product and $\beta^*$ is defined in~\eqref{conjugate}.
\end{definition}
\begin{remark}\label{svwbasisA}
Suppose that $h^{1,1}(X) = h^{2,2}(X) = 1$. Then there is a basis $\lbrace e_3,e_2,e_1,e_0\rbrace$ for the even degree part of $H^*(X;\mathbb{C})$ given by
\begin{align*}
e_3 = [X] \,;\quad e_2 = [D] \,; \quad e_1 = -[\ell] \, ;\quad e_0 [\pt]
\end{align*}
where $[D]$ is the hyperplane class dual to the complexified K{\"a}hler form on $X$, and $e_1$ is chosen to have intersection number $e_1\cdot e_2 = 1$. These extend to sections of $\mathcal{E}$.

In this basis, a connection matrix for the quantum connection can be expressed as
\begin{align}
\begin{pmatrix} 0 & 0 & 0 & 0 \\ 1 & 0 & 0 & 0 \\ 0 & -\Phi'' & 0 & 0 \\ 0 & 0 & -1 & 0\end{pmatrix} \label{connectionmatrixA}
\end{align}
where $\Phi''$ is the closed Gromov--Witten potential with two interior constraints.
\end{remark}
The $B$-model VSHS is described in more classical terms, per Lemma~\ref{VSHSlemma}. Let $R_B\coloneqq\mathbb{C}((z))$ and $\mathbb{K}_B\coloneqq\mathbb{C}((z))$. Let $X^{\vee}\to\mathcal{M}_B$ be a smooth projective connected scheme of relative dimension $3$, with trivial relative canonical sheaf. We further assume that $X^{\vee}$ is \textit{maximally unipotent} (cf.~\cite[\S{1.1}]{GPS15}).

\begin{definition}\label{BVSHS}
The small $B$-model ($\mathbb{Z}$-graded polarized) VSHS $\mathcal{H}^B(X^{\vee})$ consists of
\begin{itemize}
\item[•] the relative de Rham cohomology $\mathcal{V}$ of $X^{\vee}$ over the formal punctured disk $\mathcal{M}_B$, with a $\mathbb{Z}/2$-grading induced from the cohomological $\mathbb{Z}$-grading;
\item[•] the filtration
\begin{align}
F^{\geq s}\mathcal{V}\coloneqq\bigoplus_{p} H^p\left(\Omega^{\geq p+2s}_{X^{\vee}/\mathcal{M}_B}\right) \label{Bfiltration}
\end{align}
which comes from the classical Hodge filtration;
\item[•] the Gauss--Manin connection; and
\item[•] the integration pairing.
\end{itemize}
\end{definition}
\begin{remark}
The filtration can be written
\begin{align*}
F^{\geq\frac{3}{2}}\mathcal{V}\subset F^{\geq\frac{1}{2}}\mathcal{V}\subset F^{\geq-\frac{1}{2}}\mathcal{V}\subset F^{\geq-\frac{3}{2}}\mathcal{V} \,.
\end{align*}
With this convention for the filtration, the lowest level $F^{\geq\frac{3}{2}}\mathcal{V} = H^0\left(\Omega^3_{X^{\vee}/\mathcal{M}_B}\right)$ is the space of (relative) holomorphic volume forms on $X^{\vee}$. Under classical conventions, this level of the filtration would be denoted $F^3\mathcal{V}$ (cf.~\cite{CK99, HW22}).
\end{remark}

\begin{remark}\label{svwbasisB}
Suppose $F^{\geq\frac{3}{2}}\mathcal{V}$ is one-dimensional and that rank of $F^{\geq p+\frac{1}{2}}\mathcal{V}$ increases by $1$ at each level, as is the case for the mirror quintic. In this setting~\cite{SVW17} constructs a basis $\lbrace e_3,e_2,e_1,e_0\rbrace$ for the odd-degree part of the module $\mathcal{V}$. These sections are such that $e_i\in F^{\geq i-\frac{3}{2}}\mathcal{V}_{\odd}$, and in particular we can take $e_3 = \Omega\in F^{\geq\frac{3}{2}}\mathcal{V}_{\odd}$ to the Hodge-theoretically normalized volume form. This basis is constructed using the fact that $0\in\Delta$ is a point of maximally unipotent monodromy, where the monodromy is associated to the Gauss--Manin connection on the $B$-side. 

In this basis, the connection matrix for the Gauss--Manin connection can be expressed as
\begin{align}
\begin{pmatrix} 0 & 0 & 0 & 0 \\ 1 & 0 & 0 & 0 \\ 0 & -\mathfrak{C} & 0 & 0 \\ 0 & 0 & -1 & 0\end{pmatrix} \label{connectionmatrixB}
\end{align}
where $\mathfrak{C}$ is the Yukawa coupling~\cite{CK99}.
\end{remark}
These VSHSs are related in the following way.
\begin{theorem}[{cf.~\cite[Theorem A]{GPS15}}]\label{gpsthm}
Suppose that $X$ and $X^{\vee}$ as above are homologically mirror. Under Assumption~\ref{connectionsass} and Assumption~\ref{traceass}, there are isomorphisms of VSHS given by
\begin{equation}
\begin{tikzcd}
HC_*^{-}(\mathcal{F}(X))\arrow{r}\arrow{d}{\mathcal{OC}^{-}} & HC_*^{-}(D^b_{dg}\Coh(X^{\vee}))\arrow{d}{\widetilde{\mathfrak{J}}_{\HKR}} \\\mathcal{H}^A(X,\omega) & \mathcal{H}^B(X^{\vee})
\end{tikzcd}\label{vhsisomorphisms}
\end{equation}
where the horizontal arrow is given by homological mirror symmetry and Morita invariance.
\end{theorem}
The isomorphism $\widetilde{\mathfrak{J}}_{\HKR}$ is obtained from the HKR (Hochschild--Kostant--Rosenberg) isomorphism~\cite{Wei97}
\begin{align*}
\widetilde{I}_{\HKR}\colon HC_*^{-}(D^b_{dg}\Coh(X^{\vee})) \cong HC_*^{-}(X^{\vee})\to\mathcal{H}^B(X^{\vee})
\end{align*}
which is in turn a lit of the HKR isomorphism on Hochschild homology. We define $\widetilde{\mathfrak{J}}_{\HKR}$ to be the composition of $\widetilde{I}_{\HKR}$ with the square root of the Todd class
\begin{align*}
H^*(\Omega^{-*}X^{\vee})\xrightarrow{\td^{1/2}(X^{\vee})}H^*(\Omega^{-*}X^{\vee})\,.
\end{align*}
Tu~\cite[Remark 0.3]{Tu24} shows that $\widetilde{\mathfrak{J}}_{\HKR}$ is an isomorphism of polarized VSHS.
\begin{remark}
As explained in~\cite[\S{1.10}]{GPS15}, the mirror map appearing in Hodge-theoretic mirror symmetry (cf.~\cite{CK99}) arises as a change of coordinates relating $\mathbb{C}((Q))$ to $\mathbb{C}((z))$ in homological mirror symmetry, and the natural Calabi--Yau structure on $\mathcal{F}(X)$ associated to the negative cyclic open-closed map is mirror to the Calabi--Yau structure on the derived category determined by the Hodge-theoretically normalized volume form. These facts make it possible to recover closed string enumerative mirror symmetry from homological mirror symmetry, and they will also allow us to recover open string enumerative mirror symmetry.
\end{remark}

\section{Background on Extensions of VSHS}\label{extensionsect}
Morrison showed that enumerative mirror symmetry is encapsulated by an isomorphism of VSHS~\cite{Mor93}. Similarly, we will see that open enumerative mirror symmetry, as in e.g.~\cite{PSW08}, can be reformulated as the existence of an isomorphism of \textit{extensions of VSHS}. In this section, we will explain how to classify extensions of VSHS following~\cite{Hug24b} in a language consistent with the discussion of VSHS in the last section. We will also discuss the classical construction of an extension of VHS associated to a homologically trivial algebraic cycle, as discussed in~\cite{HW22}.

\subsection{Extensions as normal functions}
In this subsection, we will work with the characterization of VSHSs from Lemma~\ref{VSHSlemma}. Let $\mathcal{V}$ be a polarized VSHS over the formal disk $\mathcal{M}$, and consider an extension
\begin{align}
0\to\mathcal{V}\xrightarrow{a}\mathcal{V}'\xrightarrow{b}\mathbb{K}\to 0 \label{extvshs}
\end{align}
where $\mathbb{K}$ carries the trivial connection and $F^{\geq p}\mathbb{K} = \mathbb{K}$ and $F^{\geq p+1} = 0$ for a fixed $p\in\frac{1}{2}\mathbb{Z}$. Sufficiently well-behaved extensions of this form can be classified by \textit{normal functions} as we will now explain following~\cite[\S{6.1}]{Hug24b}.
\begin{definition}
We say that a variation of semi-infinite Hodge structures $\mathcal{V}$ is \textit{regular singular} if there is an $R$-submodule $\mathcal{V}_R\subset\mathcal{V}$ with $\nabla_{q\partial_q}\mathcal{V}_R\subset\mathcal{V}_R$. An extension of VSHS as in~\eqref{extvshs} is said to be regular singular if both $\mathcal{V}$ and $\mathcal{V}'$ are.
\end{definition}

The \textit{Deligne lattice} of $\mathcal{V}$ is an $R$-submodule $\widetilde{\mathcal{V}}\subset\mathcal{V}$ which is characterized by the requirement that the residue of the connection
\begin{align*}
N\colon\widetilde{\mathcal{V}}_0\to\widetilde{\mathcal{V}}_0
\end{align*}
has eigenvalues in $[0,1)$, where $\widetilde{\mathcal{V}}_0\coloneqq\widetilde{\mathcal{V}}/q\widetilde{\mathcal{V}}$ denote the fiber of $\widetilde{\mathcal{V}}$ over $0$. There is a filtration $\widetilde{F}^{\geq p}\widetilde{\mathcal{V}}$ on the Deligne lattice induced from the filtration on $\mathcal{V}$.

Taking Deligne lattices of the modules in~\eqref{extvshs} gives a short exact sequence of $R$-modules
\begin{align}
0\to\widetilde{\mathcal{V}}\to\widetilde{\mathcal{V}}'\to R\to 0
\end{align}
and restricting to the central fibers gives us a short exact sequence of $\mathbb{C}$-vector spaces
\begin{align}
0\to\widetilde{\mathcal{V}}_0\to\widetilde{\mathcal{V}}'_0\to\mathbb{C}\to 0 \,. \label{endses}
\end{align}
If we let $M$ denote the residue of the connection
\begin{align*}
M\colon\widetilde{\mathcal{V}}'_0\to\widetilde{\mathcal{V}}'_0
\end{align*}
then~\eqref{endses} is a short exact sequence of vector spaces with endomorphisms, where the endomorphism acting on $\mathbb{C}$ is $0$.

A regular singular extension of VSHS is called \textit{holomorphically flat} if~\eqref{endses} splits. In other words, this condition means that we can write
\begin{align*}
M = \begin{pmatrix} N & 0 \\ 0 & 0 \end{pmatrix}
\end{align*}
with respect to this splitting. By~\cite[Lemma 6.6]{Hug24b}, this is equivalent to the existence of an element $h\in\ker(\nabla^{\mathcal{V}'}_{q\partial_q})\subset\widetilde{\mathcal{V}'}\subset\mathcal{V}'$ such that $b(h) = 1\in R$. Hugtenburg classifies holomorphically flat regular singular extensions of VSHS in analogy with the classification of extensions of VHS due to Carlson~\cite{Car87}.
\begin{definition}\label{normalfunctiondef}
In the setting of~\eqref{extvshs}, the \textit{$k$th intermediate Jacobian} is defined to be
\begin{align*}
\widetilde{\mathcal{J}}^k\coloneqq\frac{\widetilde{\mathcal{V}}}{\widetilde{F}^{\geq k}\widetilde{\mathcal{V}}\oplus\ker(\nabla_{\partial_q})} \, .
\end{align*}
A \textit{normal function} is an element $\nu\in\widetilde{\mathcal{J}}^k$ of the intermediate Jacobian such that for a lift $\widetilde{\nu}\in\widetilde{\mathcal{V}}$ we have $\nabla_{q\partial_q}\widetilde{{\nu}}\in\widetilde{F}^{\geq k-1}\widetilde{\mathcal{V}}$.
\end{definition}
\begin{proposition}[{\cite[Proposition 6.12]{Hug24b}}]\label{normalfunprop}
There is a bijection between the set of holomorphically flat regular singular extension of VSHS with filtrations as in~\eqref{extvshs} and the set of normal functions in $\widetilde{\mathcal{J}}^k$. \qed
\end{proposition}
We sketch a proof of this Lemma, as some knowledge of the construction of a normal function will be helpful in the proof of Theorem~\ref{mainprelim} when we compare extensions of VSHS in the $B$-model.
\begin{proof}[Proof sketch]
Assume given an extension of VSHSs as in~\eqref{extvshs} which is regular singular and for which the short exact sequence~\eqref{endses} of complex vector spaces with endomorphims splits. Since the extension is regular singular, we an choose an element $f\in\widetilde{F}^{\geq k}\widetilde{\mathcal{V}}'$ with $b(f) = 1$. The existence of a splitting implies the existence of an element $h\in\ker\left(\nabla^{\mathcal{V}'}_{q\partial_q}\right)$ as discussed above. A lift of the normal function is obtained by setting
\begin{align*}
\widetilde{\nu}\coloneqq f-h\in\mathcal{\widetilde{V}}
\end{align*}
and one can check that the image $\nu\in\widetilde{\mathcal{J}}^k$ of $\widetilde{\nu}$ in the quotient is independent of the choices for $h$ and $f$. 

Conversely, given a normal function $\nu\in\widetilde{\mathcal{J}}^k$, we will construct an extension of VSHS with underlying $\mathbb{K}$-module
\begin{align*}
0\to\mathcal{V}\to\mathbb{K}\oplus\mathcal{V}\to\mathbb{K}\to 0 \,.
\end{align*}
Fix a lift $\widetilde{\nu}$ of the normal function. The connection on $\mathbb{K}\oplus\mathcal{V}$ is given by $q\partial_q\oplus\nabla^{\mathcal{V}}$ and the filtration is determined by setting
\begin{align*}
\mathcal{F}^{\geq i}(\mathbb{K}\oplus\mathcal{V})\coloneqq(\lbrace 0\rbrace\oplus F^{\geq i}\mathcal{V})+(F^{\geq i}\mathbb{K}\oplus\mathbb{K}\widetilde{\nu})
\end{align*}
where the filtration on $\mathbb{K}$ is such that $\mathcal{F}^{\geq k}\mathbb{K} = \mathbb{K}$ and $F^{\geq k+1}\mathbb{K} = 0$, consistent with~\eqref{extvshs}. Griffiths transverality follows from the definition of a normal function, and one can also check the VSHS obtained this way does not depend on the choice of lift.
\end{proof}

\subsection{Normal functions from algebraic cycles}
Let $\pi\colon\mathcal{X}^{\vee}\to\Delta^*$ be a family of smooth Calabi--Yau threefolds over $\mathbb{C}$, where $\Delta^*$ is the unit disk in $\mathbb{C}$ with the origin removed. Further assume that this family admits a semistable continuation over the unit disk $\Delta$, and that $0\in\Delta$ is a point of maximally unipotent monodromy (i.e. a large complex structure limit point in the complex moduli space). Note that up to a change of base, one can obtain a relative scheme $X^{\vee}\to\mathcal{M}_B$ as considered in~\cite{GPS15}.

Let $X_z^{\vee}$ denote the fiber of $\mathcal{X}^{\vee}$ over $z\in\Delta^*$, and assume that we are given algebraic curves $i\colon C_z^i\to X_z$, for $i = 0,1$, which are the fibers of smooth families $\mathcal{C}^i\to\Delta^*$. Suppose that these families admit semistable continuations $\overline{\mathcal{C}}^i$ over $\Delta$. These can be can be thought of as algebraic curves $C^i\subset X^{\vee}$ in the scheme over $\mathcal{M}_B$. If $[C^0_z]$ and $[C^1_z]$ lie in the same homology class, then for any integer $m>0$ we can consider a family of algebraic cycles $m\mathcal{C}$ whose fibers are the homologically trivial algebraic cycles $m C^0_z-m C^1_z$.

From such a family of homologically trivial algebraic cycles, one can construct an extension of VHS using classical techniques~\cite{Car87}. We will summarize this construction following~\cite{Voi} and~\cite{HW22}, in particular explaining how the normal function thus obtained gives a normal function in the sense of Definition~\ref{normalfunctiondef}.

Since each $m[C_z]\in H_2(X_z^{\vee};\mathbb{Z})$ vanishes, we have a short exact sequence
\begin{align}
0\to H_3(X_z^{\vee};\mathbb{Z})\to H_3(X_z^{\vee},mC_z;\mathbb{Z})\to\mathbb{Z}\to 0 \label{nullcycleses}
\end{align}
for all $z$, where the rightmost entry is thought of as a generator for $H_2(mC_z;\mathbb{Z})$. This allows us to associate a normal function to $m\mathcal{C}$ as follows. In each fiber, choose any $3$-chain $\Gamma_z\in H_3(X_z^{\vee},mC_z;\mathbb{Z})$ with $\partial\Gamma_z = mC_z$.
\begin{lemma}
The integral
\begin{align*}
\nu_z(\eta)\coloneqq(2\pi i)^2\int_{\Gamma_z}\eta \, ;\qquad \eta\in F^{\geq\frac{1}{2}}H^3(X_z^{\vee};\mathbb{C})
\end{align*}
is independent of $\Gamma_z$ up to periods of the form $\int_{A_z}\eta$ for $A_z\in H_3(X_z^{\vee};\mathbb{Z})$. In particular, $\nu$  can be thought of as an element of the (classical) intermediate Jacobian
\begin{align}
(F^{\geq\frac{1}{2}}H^3(X_z^{\vee};\mathbb{C}))^{\vee}/H^3(X_z^{\vee};\mathbb{Z})^{\vee}\cong\frac{H^3(X_z^{\vee};\mathbb{C})}{F^{\geq\frac{1}{2}}H^3(X^{\vee};\mathbb{C})\oplus H^3(X_z^{\vee};\mathbb{Z})} \,. \label{classicaljacobian}
\end{align}
\qed
\end{lemma}
The group~\eqref{classicaljacobian} is a fiber of an intermediate Jacobian. Let $h = (2\pi i)^2\delta_{\Gamma_z}$, where $\delta_{\Gamma_z}$ is a differential form representing the class in $H^3(X_z^{\vee}\setminus C_z;\mathbb{C})$ Poincar{\'e} dual to $\Gamma_z$, and choose a form $f\in H^3(X_z^{\vee}\setminus C_z;(2\pi i)^2\mathbb{Z})$ such that $df = (2\pi i)^2\delta_{C_z}$, which exists because $C_z$ is nullhomologous. Then we can write
\begin{align}
\nu_z(\eta) = (2\pi i)^2\int_{X_z^{\vee}}(h-f)\wedge\eta = (2\pi i)^2\int_{X_z^{\vee}}h\wedge\eta = (2\pi i)^2\int_{\Gamma_z}\eta \label{normalfunctionintegral}
\end{align}
for any $\eta\in F^{\geq\frac{1}{2}}H^3(X_z^{\vee};\mathbb{C}) = H^{3,0}(X_z^{\vee})\oplus H^{2,1}(X_z^{\vee})$. The integral
\begin{align*}
\int_{X_z^{\vee}}f\wedge\eta
\end{align*}
vanishes by type considerations. Notice that the difference $h-f\in H^3(X_z^{\vee})$ is well-defined modulo $F^{\geq\frac{1}{2}}H^3(X_z^{\vee};\mathbb{C})\oplus H^3(X_z^{\vee};(2\pi i)^2\mathbb{Z})$. By carrying out this construction in each fiber $X_z^{\vee}$ over $z\in\Delta^*$, we obtain a function on $\Delta^*$. By a theorem of Griffiths, this function is holomorphic on $\Delta^*$ and horizontal~\cite[Lemma 7.9]{Voi}, where the latter condition is the analogue of the defining condition in Definition~\ref{normalfunctiondef}.

When $\mathcal{X}^{\vee}$ is the mirror quintic family, $\nu_z(\eta)$ takes a particularly simple form. Slightly more generally, we will assume that the Hodge numbers of $\mathcal{X}^{\vee}$ are as follows.
\begin{assumption}\label{BHodgenumberass}
The Hodge numbers of $X^{\vee}_z$ satisfy
\begin{align}
h^{3,0}(X^{\vee}_z) = h^{2,1}(X^{\vee}_z) = h^{1,2}(X^{\vee}_z) = h^{0,3}(X^{\vee}_z) = 1
\end{align}
and $X^{\vee}_z$ is simply connected.
\end{assumption}
Assumption~\ref{BHodgenumberass} enables us to express $\nu$ in terms of $[C_z]$ and the Hodge-theoretically normalized volume form. As noted in Remark~\ref{svwbasisB}, this Assumption allows us to write the connection matrix for the Gauss--Manin connection in the form~\eqref{connectionmatrixB}.

As explained in~\cite[\S{3.1}]{HW22}, one can choose a canonical lift of $\nu$ by considering the monodromy logarithm, which we write as
\begin{align*}
\widetilde{\nu} = \mathcal{W}_1e_1+\mathcal{W}_0e_0 \,.
\end{align*} 
This is a section of the local system $R^3\pi_*\mathbb{Z}\otimes\mathcal{O}_{\Delta^*}$, where $\pi\colon\mathcal{X}^{\vee}\to\Delta^*$ is the projection map for the family. Here we can write
\begin{align}
\mathcal{W}_0(z) = \int_{\Gamma_z}\Omega_z \label{relativeperiodintegral}
\end{align}
where $\Omega_z$ denotes the restriction of the Hodge-theoretically normalized relative volume form $\Omega$ to $X_z^{\vee}$. Using~\eqref{connectionmatrixB} and horizontality, it follows that $W_1 = z\partial_z \mathcal{W}_0$, meaning that the normal function, and hence the extension of Hodge structures, is completely determined by~\eqref{relativeperiodintegral}. It is shown in~\cite{HW22} that $\mathcal{W}_0$ can be written in the form
\begin{align*}
\mathcal{W}_0(z) = \frac{1}{(2\pi i)^2}\frac{\lambda}{r^2}\log^2(z)+\frac{1}{2\pi i}\frac{s}{r}\log(z)+c+w(z)
\end{align*}
where $w(z)$ extends to a holomorphic function which vanishes at $0\in\Delta$, up to passing to an $r$-fold cover
\begin{align*}
\Delta^*\to\Delta^* \\
z\mapsto z^r
\end{align*}
to account for branching of $\mathcal{C}$ near the origin.
\begin{remark}
The integer $r$ is the order of the large complex structure limit monodromy when restricted to the trivial VHS associated to $\mathcal{C}$. Here $\lambda$ and $s$ are both integers. It turns out that for (multiples of) the van Geemen lines, we have that $r = 1$, because they are preserved by the monodromy map on $X_z^{\vee}$. A direct calculation carried out by Jockers--Morrison--Walcher shows that $\lambda = s = 0$. These terms of the normal function are not included in the data of a VSHS, so we will not explain their significance any further.
\end{remark}
We obtain an element of the $\frac{1}{2}$th intermediate Jacobian of Definition~\ref{normalfunctiondef} by considering the power series in $\mathbb{C}[[z]]$ corresponding to the \textit{holomorphic part} $w(z)$ of the normal function. Denote by
\begin{align}
\mathcal{H}^B(\mathcal{X}^{\vee},\mathcal{C}) \label{geometricbext}
\end{align}
the extension of the trivial VSHS $\mathbb{K}_B[[u]]$ by $\mathcal{H}^B(\mathcal{X}^{\vee})$. Removing the $\log(z)$ and $\log^2(z)$ terms is necessary because normal functions, as we have defined them in Definition~\ref{normalfunctiondef}, are required to lie in a quotient of the Deligne lattice, which is a module over $\mathbb{C}[[z]]$, not $\mathbb{C}((z))$.
\begin{remark}
While the constant term $c$ is part of the normal function considered in Definition~\ref{normalfunctiondef}, we note that any constant section is flat, so the extension of VSHS does not depend on the constant term of the normal function. Classically, the constant term determines the integral lattice on an extension of VHS, so this reflects the fact that a VSHS does not include the data of an integral lattice. This is the reason that we are unable to directly give an interpretation of the integral structure on the $A$-model using homological mirror symmetry. Analogously, in the closed string case one cannot obtain the Gamma conjecture directly from Theorem~\ref{gpsthm}.
\end{remark}

\section{Open Gromov--Witten invariants and relative period integrals}\label{comparisonthms}
We will prove Theorem~\ref{mainprelim} in this section by establishing an analogue of Theorem~\ref{gpsthm} for extensions of VHS. Stating a fully general open string analogue of this Theorem would require an appropriate notion of a mirror pair of objects. It would be difficult to formulate this notion in a canonical way compatible with homological mirror symmetry, since the support of a sheaf is not a categorical invariant, i.e. autoequivalences of the derived category need not preserve the support of an object. Thus we will content ourselves with stating our comparisons of VSHSs in terms of a specific choice of mirror functor. This is part of the reason that the statement of Theorem~\ref{main1} referred to the quasi-equivalence constructed in~\cite{She15}. 
\begin{assumption}\label{functass}
$(X,\omega)$ is a symplectic Calabi--Yau threefold and $X^{\vee}\to\mathcal{M}_B$ is a relative scheme of dimension $3$ coming from a family of smooth Calabi--Yau threefolds with maximally unipotent monodromy. There is a fixed quasi-equivalence
\begin{align}
\mathcal{F}(X,\omega)\simeq D^b_{dg}\Coh(X^{\vee}) \label{promisedequiv}
\end{align}
where the coefficient fields are related by the mirror map $\mathbb{K}_A\to\mathbb{K}_B$.

Moreover let $(L_0,\nabla_0,b_0),(L_1,\nabla_1,b_1)\in\mathcal{F}(X,\omega)$ be a pair of cleanly immersed Lagrangian branes in $X$ such $L_0$ and $L_1$ intersect cleanly and such that the Lagrangian immersion $L\coloneqq L_0\sqcup L_1\looparrowright X$ is nullhomologous. The disjoint union carries a natural local system and bounding cochain, and the resulting Lagrangian brane is denoted $(L,\nabla,b)$.

We assume that $(L_i,\nabla_i,b_i)$ corresponds under~\eqref{promisedequiv} to the pushforward of a vector bundles of rank $m$, denoted $\mathcal{L}_i$, supported on a curve $C^i\subset X^{\vee}$, and that the restrictions  $C_z^i$ of these curves to the smooth fibers $X_z^{\vee}$ lie in the same homology class. In other words, we have a family $C_z\coloneqq C^0_z-C_z^1$ of homologically trivial algebraic cycles in $X_z^{\vee}$.
\end{assumption}
Considering a pair of branes, rather than a single brane as in~\cite{Hug24b}, does not lead to a loss of generality given the other parts of Assumption~\ref{functass}, since there are no nullhomologous algebraic \textit{curves} on the $B$-side. Keeping track of the homology classes of curves will be crucial to our comparison of $B$-model extensions of VSHS. 

We will also impose a rather restrictive assumption on the Hodge numbers of the $A$-model symplectic manifold $X$, which are satisfied by the quintic.
\begin{assumption} \label{AHodgenumberass}
The Hodge numbers of $X$ satisfy
\begin{align*}
h^{0,0}(X) = h^{1,1}(X) = h^{2,2}(X) = h^{3,3}(X) = 1 
\end{align*}
and $X$ is simply connected. In particular $X^{\vee}$ satisfies Assumption~\ref{BHodgenumberass} by mirror symmetry.
\end{assumption}
This assumption enables us to make use of a basis for the VSHS $\mathcal{H}^A(X,\omega)$ as described in Remark~\ref{svwbasisA}, which in turn lets us easily construct an $A$-model extension of VSHS. Despite their restrictive nature, Assumptions~\ref{functass} and~\ref{AHodgenumberass} encompass the examples of the real quintic~\cite{PSW08} and the Lagrangian $\tLvgsmooth$, which are the only examples of open mirror pairs, as defined in~\cite{HW22}, in the literature.

One would expect the $A$-model analogue of the extension of VHS associated to a nullhomologous algebraic cycle to come from Solomon--Tukachinsky's relative quantum cohomology~\cite{ST23} of a nullhomologous Lagrangian. We cannot appeal to their construction directly, since our definition of the open Gromov--Witten potential differs from theirs at the chain level. Although we could define a relative quantum connection using the results of \S\ref{immersedogw}, we would require a different proof of the flatness of this connection.

We will avoid this issue by constructing an extension of VSHS the trivial VSHS $\mathbb{K}_A$  by $\mathcal{H}^A(X,\omega)$ at the homological level by specifying a normal function in terms of the ($\infty$-cyclic) open Gromov--Witten potential. The isomorphisms of Theorem~\ref{gpsthm}, and in particular Assumption~\ref{functass}, then allow us to construct a $B$-model extension of the trivial VSHS $\mathbb{K}_B$ by $\mathcal{H}^B(X^{\vee})$. Our description of $B$-model normal functions is then rigid enough that we can prove this extension of VSHSs to be equivalent to the one described in the previous subsection.

Fix a bounding $4$-chain $\Gamma$ for the nullhomllogous Lagrangian immersion $L$. In what follows, we will write $\Psi$ for the open Gromov--Witten potential $\Psi_{L,\Gamma}(b)$ of Corollary~\ref{bounding-4chain}, which we will now think of as a function of the Novikov variable $Q$.
\begin{proposition}
Let $\lbrace e_3,e_2,e_1,e_0\rbrace$ denote the basis for the even degree part of $\mathcal{H}^A(X,\omega)$ specified in Remark~\ref{svwbasisA}. Then
\begin{align*}
\widetilde{\nu}_A\coloneqq Q\partial_Q\Psi(Q)e_1+\Psi(Q)e_0
\end{align*}
descends to a well-defined normal function $\nu_A$ for $\mathcal{H}^A(X,\omega)$.
\end{proposition}
\begin{proof}
Horizontality follows from the form of the connection matrix in Remark~\ref{svwbasisA}, from which we can directly compute that
\begin{align*}
\nabla\widetilde{\nu}_A = -(Q\partial_Q)^2\Psi(Q)e_1 \,.
\end{align*}
That this is well-defined, i.e. independent of the almost complex structure on $X$, follows because $L$ is nullhomologous by Corollary~\ref{bounding-4chain}. Moreover, replacing $\Gamma$ with another bounding $4$-chain does not affect $\nu_A$ as an element in the intermediate Jacobian. To see this, note that if $\Gamma_1$ and $\Gamma_2$ are two choices of bounding $4$-chains, then their difference can be represented as a closed smooth singular $4$-cycle, which implies that OGW potentials obtained from these two chains only differ by classes in $\widetilde{F}^{\geq k}\widetilde{\mathcal{V}}$.  
\end{proof}
\begin{definition}
Denote by $\mathcal{H}^A(X)_L$ the extension of VSHSs associated to $\nu_{\mathcal{F}}$. 
\end{definition}
\begin{remark}
For a nullhomologous embedded Lagrangian, the extension of VSHSs constructed this way is just the relative quantum cohomology of~\cite{ST23}.
\end{remark}
Using our characterization of the open Gromov--Witten potential, we can lift this to a normal function for the negative cyclic homology of $\mathcal{F}(X,\omega)$.
\begin{proposition}
Under assumptions~\ref{qass}, \ref{connectionsass}, and~\ref{traceass}, there is a normal function $\nu_{\mathcal{F}}$ in $HC_*^{-}(\mathcal{F}(X,\omega))$ such that
\begin{align*}
\mathcal{OC}^{-}(\nu_{\mathcal{F}}) = \nu_A \,.
\end{align*}
It follows that the negative cyclic open-closed map induces an isomorphisms between the extensions of VSHS obtained from these normal functions.
\end{proposition}
\begin{proof}
Notice that the terms of $-(Q\partial_Q)^2\Psi_b(Q)$ involving $b$ lie in the image of the $\mathcal{OC}^{-}$ by construction, since the $\infty$-inner product used to define $\Psi$ is defined using values of the cyclic open-closed map. Assumption~\ref{qass}, the divisor axiom, and Assumption~\ref{AHodgenumberass} imply that we can write
\begin{align*}
-(Q\partial_Q)^2\mathfrak{m}_{-1} = -\mathfrak{q}_{-1,2}(\omega\otimes\omega) \,.
\end{align*}
On the other hand, the right hand side can be expressed in terms of operators with horocyclic constraints via
\begin{align*}
\mathfrak{q}_{-1,2}(\omega\otimes\omega) = \langle 1,\mathfrak{q}_{0,2;\perp_0}(\omega\otimes\omega)\rangle
\end{align*}
and the expression on the right hand side is in the image of the open-closed map. Hence we have shown that $\nabla\widetilde{\nu}_A$ lies in the image of the open-closed map. Using Assumption~\ref{connectionsass}, which says that $\mathcal{OC}^{-}$ is an isomorphism and that it intertwines the Getzler--Gauss--Manin and quantum connections establishes the lift of an element $\widetilde{\nu}_F\in HC_*^{-}(\mathcal{F}(X,\omega))$ whose open-closed image is $\widetilde{\nu}_A$, from which we obtain the normal function $\nu_{\mathcal{F}}$. The assumption that $\mathcal{OC}^{-}$ is an isomorphism of VSHSs implies that $\nu_{\mathcal{F}}$ is a normal function as well.
\end{proof}
\begin{definition}
Let $HC_*^{-}(\mathcal{F}(X))_L$ denote the extension of VSHSs determined by $\nu_{\mathcal{F}}$.
\end{definition}
\begin{remark}
We have not checked whether $\nu_{\mathcal{F}}$ is the normal function associated to the extension of VSHSs constructed from relative negative cyclic homology defined in~\cite[\S{2.3}]{Hug24b}. This statement is essentially the content of~\cite[\S{5}]{Hug24b}, but those arguments are phrased at the chain-level, and thus do not immediately adapt to our setting.
\end{remark}
\begin{corollary}\label{categoricalnormalfunctions}
Recall that $\mathcal{L}_i\in D^b_{dg}\Coh(X^{\vee})$ denotes the mirror sheaf to $(L_i,\nabla_i,b_i)$, for $i = 0,1$, which is a rank $m$ vector bundle over $C_i$. with respect to the particular mirror functor~\eqref{promisedequiv}. Then we have isomorphisms of (extensions of) VSHS that fit into the following diagram
\begin{equation}
\begin{tikzcd}
HC_*^{-}(\mathcal{F}(X))_L\arrow{r}\arrow{d}{\mathcal{OC}^{-}} & HC_*^{-}(D^b_{dg}\Coh(X^{\vee}))_{\mathcal{L}_0-\mathcal{L}_1}\arrow{d}{\widetilde{\mathfrak{J}}_{\HKR}} \\\mathcal{H}^A(X,\omega)_L & \mathcal{H}^B(X^{\vee})_{\mathcal{L}_0-\mathcal{L}_1} \,.
\end{tikzcd}\label{vhsextisomorphisms}
\end{equation}
Here the horizontal arrow is induced by our choice of mirror functor, and the vertical arrows come from $\mathcal{OC}^{-}$ and $\widetilde{\mathfrak{J}}_{\HKR}$ (cf. Proposition~\ref{normalfunprop}). 
\end{corollary}
Here the $B$-model VSHSs are defined such that all arrows in the diagram are isomorphisms. The corollary is immediate from Theorem~\ref{gpsthm}. We can think of the vector spaces underlying the categorical extensions of VSHS in~\eqref{vhsextisomorphisms} as the \textit{relative negative cyclic homology} of~\cite{Hug24b}.

To obtain predictions for open Gromov--Witten invariants from this, we must prove that:
\begin{proposition}\label{comparingbnormalfunctions}
The extension of VSHS $\mathcal{H}^B(X^{\vee})_{\mathcal{L}_0-\mathcal{L}_1}$ is isomorphic to the extension of VSHS associated to the family of cycles $mC$.
\end{proposition}
\begin{proof}
We begin by determining the image of the sheaf $\mathcal{L}_i$ under the HKR isomorphism. Recall that $\widetilde{\mathfrak{J}}_{\HKR}$ is induced from the HKR isomorphism
\begin{align*}
HH_*(X^{\vee})\to H^*(X^{\vee})\,.
\end{align*}
On the other hand~\cite[Theorem 4.5]{Cal04} implies that the ordinary Chern character on $K$-theory $\ch\colon K_0(X^{\vee})\to H^*(X^{\vee})$ factors as the composition of the Chern character $K_0(X^{\vee})\to HH_*(X^{\vee})$ with the HKR isomorphism. Thus we see that image of $\mathcal{L}_i$ under $\widetilde{\mathfrak{J}}_{\HKR}$ is its algebraic second Chern class $c_2(\mathcal{L}_i)$. Our assumption that $\mathcal{L}_i$ is the pushforward of a vector bundle implies that the algebraic Chern class is represented by $mC^i$, where $m$ is the rank of the vector bundle and $C^i$ is the support of the sheaf.

There is a $\mathbb{C}$-local system on $\Delta^*$ associated to $\mathcal{H}^B(X^{\vee})_{\mathcal{L}_0-\mathcal{L}_1}$, and the previous paragraph implies that it is given fiberwise by~\eqref{nullcycleses}. As discussed in the proof of Proposition~\ref{normalfunprop} and around~\eqref{normalfunctionintegral}, a normal can be determined by specifying a suitable pair of elements in $\mathcal{H}^B(X^{\vee})_{\mathcal{L}_0-\mathcal{L}_1}$ which are both mapped to $1\in\mathbb{K}_B$. By~\eqref{nullcycleses}, such elements are given by relative chains with boundary $mC$. Hence by Poincar{\'e} duality and type considerations, the normal function for $\mathcal{H}^B(X^{\vee})_{\mathcal{L}_0-\mathcal{L}_1}$ can be written as~\eqref{normalfunctionintegral}. As an element of the intermediate Jacobian, the normal function thus described is independent of the choice of bounding chain.
\end{proof}
We can now complete the proofs of Theorems~\ref{main2} and~\ref{mainprelim} from the introduction.
bo\begin{proof}[Proof of Theorem~\ref{main2}]
This is an immediate consequence of Corollary~\ref{categoricalnormalfunctions} and Proposition~\ref{comparingbnormalfunctions}.
\end{proof}
\begin{proof}[Proof of Theorem~\ref{mainprelim}]
Using the notation of this section, let $L_0$ and $L_1$ both denote the Lagrangian branes specified in the statement of Theorem~\ref{mainprelim}. More precisely, these are copies of $\tLvgsmooth$ equipped with local systems obtained from those in Example~\ref{vangeemenlocalsystem} via Definition~\ref{antisurgerybranes}. The difference of open Gromov--Witten potentials of these two branes can be thought of as the open Gromov--Witten potential associated to the immersed Lagrangian consisting of the union of these two branes, where the orientation on $L_1$ is reversed. To define the open Gromov--Witten potential, we use the degenerate $4$-chain interpolating between these two copies of the same Lagrangian submanifold.

Using the $A_{\infty}$-functor~\eqref{a-inf-functor} and the result of Theorem~\ref{mirrorsheaf}, we can think of these branes as mirrors to pushforwards of line bundles on the van Geemen lines, so that the second Chern classes of these objects are represented by $-[C^{\omega^i}]$. The result of Theorem~\ref{mirrorsheaf} says that Assumption~\ref{functass} is satisfied is satisfied in this context, so we can appeal to Theorem~\ref{main2}. 

A solution to an inhomogeneous Picard--Fuchs equation associated to the van Geemen line $C_z^{\omega}$ was computed by Walcher~\cite[(6.12)]{Wal12}, and is written there after a change of coordinates given by the mirror map. The value of the relative period integral~\eqref{relativeperiodintegral} over a chain with boundary $C_z^{\omega}-C_z^{\omega^2}$ is twice the value of the solution to the Picard--Fuchs equation (cf.~\cite[(2.9)]{Wal12}). Since the branes $(\tLvg,\nabla^{\vG}_{\omega})$ and $(\tLvg,\nabla^{\vG}_{\omega^2})$ are mirror to vector bundles on the van Geemen lines of rank $m=1$, it follows that the values of the open Gromov--Witten invariants differ from the values calculated by Walcher by an overall factor of $2$.
\end{proof}

\begin{proof}[Proof of Corollary~\ref{oldmain}]
This follows the same strategy as the proof of Theorem~\ref{main2}, except in this case $m=2$, so the open Gromov--Witten invariants differ from the values calculated by Walcher by a factor of $4$.
\end{proof}

\appendix
\section{Background on immersed Floer theory} \label{immersedfloerappendix}
This appendix is meant to fix notation and to collect some mostly standard facts regarding the Floer complex of an immersed Lagrangian with clean self-intersections. In doing so, we will also explain how the constructions of~\cite{Fuk17} should be modified in the presence of a grading and rank $1$ local system, and we will construct a spectral sequence using the energy filtration that converges to the Floer cohomology of a Lagrangian immersion.

The Lagrangian Floer cochain spaces we consider have coefficients in various versions of the Novikov ring.
\begin{definition}\label{ringdef}
The \textit{Novikov ring} over $\mathbb{C}$ is the formal power series ring
\begin{align}
\Lambda_0 = \left\lbrace\sum_{j=0}^{\infty}a_j Q^{\lambda_j}\colon a_j\in\mathbb{C} \, , \; \lambda_i\in\mathbb{R}_{\geq0} \, , \; \lim_{j\to\infty} = \infty\right\rbrace \, . \label{novikovring}
\end{align}
Here $Q$ is called the Novikov variable, and the grading on $\Lambda_0$ is trivial. There is a natural valuation $\nu\colon\Lambda_0\to\mathbb{R}_{\geq0}$ given by
\begin{align}
\nu\left(\sum_{j=0}^{\infty}a_j Q^{\lambda_j}\right) = \mathrm{min}\left\lbrace\lambda_j\colon a_j\neq 0\right\rbrace \, .
\end{align}
Let $\Lambda_+\subset\Lambda_0$ denote the subset of elements with strictly positive valuation. Correspondingly, the set of \textit{unitary Novikov elements} $U_{\Lambda}\subset\Lambda_0$ is the set of valuation $0$ elements. More explicitly, we have that
\begin{align*}
U_{\Lambda}\coloneqq\left\lbrace a_0+\sum_{j=1}^{\infty}a_j Q^{\lambda_j}\in\Lambda\colon a_0\in\mathbb{C}^*\right\rbrace \, .
\end{align*}
The \textit{Novikov field} $\Lambda$ is obtained from $\Lambda_0$ by localizing at the ideal $(Q)$, meaning that
\begin{align}
\Lambda = \left\lbrace\sum_{j=0}^{\infty}a_j Q^{\lambda_j}\colon a_j\in\mathbb{C} \, , \; \lambda_i\in\mathbb{R} \, , \; \lim_{j\to\infty} = \infty\right\rbrace \, .
\end{align}
\end{definition}
Note that any element of $\mathbb{C}^*\subset\Lambda_0$ is unitary. In general, to obtain $\mathbb{Z}$-graded $A_{\infty}$-algebras, one would require the Novikov ring to carry a nontrivial grading to account for the Maslov indices of disks, but since we only consider \textit{graded} Lagrangian immersions in this paper, we do not need this.
\subsection{Closed-open operators for immersed Lagrangians}
\begin{definition}
Let $(M,\omega)$ be a closed $2n$-dimensional symplectic manifold and $L$ a closed $n$-dimensional manifold. A Lagrangian immersion with \textit{clean self-intersections}, also called a \textit{clean Lagrangian immersion}, is a Lagrangian immersion $\iota\colon L\to M$ such that
\begin{itemize}
\item[(i)] The fiber product
\[ L\times_{\iota}L = \lbrace (p_{-},p_{+})\in L\times L\colon\iota(p_{-}) = \iota(p_{+})\rbrace \]
is a smooth submanifold of $L\times L$.

\item[(ii)] At each point $(p_{-},p_{+})$ of the fiber product, the tangent space is given by the fiber product of tangent spaces, i.e.
\[ T_{(p_{-},p_{+})}(L\times_{\iota}L) = \lbrace (v_{-},v_{+})\in T_{p_{-}}L\times T_{p_{+}}L\colon d\iota_{p_{-}}(v_{-}) = d\iota_{p_{+}}(v_{+})\rbrace.\]
\end{itemize}
\end{definition}
We will decompose the fiber product above as a disjoint union
\[ L\times_{\iota}L = \coprod_{a\in A}L_a = L_0\sqcup\coprod_{a\in A\setminus\lbrace 0\rbrace} L_a,\]
where $A$ is an index set with a distinguished element $0\in A$. Here, $L_0 \coloneqq \Delta_{L}$ denotes the diagonal component of the fiber product (which is disconnected if $L$ is disconnected). Each other component $L_a$ of the disjoint union is a non-diagonal connected component of the fiber product. The submanifolds $L_a$, where $a\in A\setminus\lbrace 0\rbrace$ are called \textit{switching components}. Observe that there is a natural free involution on $A\setminus\lbrace 0\rbrace$ induced by the function on $L\times L$ which swaps coordinates.

Suppose that $L$ is equipped with a grading $\alpha^{\#}\colon L\to\mathbb{R}$ in the sense of~\cite{Sei00}. To describe the grading on the Floer cochain in this situation, we will recall the notion of index from~\cite{AB18} following~\cite{Han24a}.
\begin{definition}
Let $(V,\Omega)$ be a symplectic vector space with a compatible almost complex structure $J$, and let $\Lambda_0,\Lambda_1\subset V$ be Lagrangian subspaces. Choose a path $\Lambda_t$, for $t\in[0,1]$, of Lagrangian subspaces from $\Lambda_0$ to $\Lambda_1$ such that
\begin{itemize}
\item $\Lambda_0\cap\Lambda_1\subset\Lambda_t\subset\Lambda_0+\Lambda_1$ for all $t$; and
\item $\Lambda_t/(\Lambda_0\cap\Lambda_1)\subset(\Lambda_0+\Lambda_1)/(\Lambda_0\cap\Lambda_1)$ is a path of positive-definite subspaces (defined with respect to the metric induced from $\omega$ and $J$) from $\Lambda_0/(\Lambda_0\cap\Lambda_1)$ to $\Lambda_1/(\Lambda_0\cap\Lambda_1)$.
\end{itemize}
Choose a path $\alpha_t$ in $\mathbb{R}$ such that
\[ \exp(2\pi i\alpha_t) = \mathrm{det}^2(\Lambda_t) \]
for all $t\in[0,1]$. Then we define the angle between $\Lambda_0$ and $\Lambda_1$ to be
\begin{align*}
\mathrm{Angle}(\Lambda_0,\Lambda_1) = \alpha_1-\alpha_0
\end{align*}
\end{definition}
This definition gives rise to the appropriate notion of index for switching components of a Lagrangian immersion.
\begin{definition}\label{index}
Given a switching component $L_{a_{-}}$ of $L\times_{\iota}L$, let $L_{a_{+}}$ denote the corresponding switching component under the involution on $A$. Then for a pair of points $p_{-}\in L_{a_{-}}$ and $p_{+}\in L_{a_{+}}$ which are swapped by this involution, we define the index
\[ \deg(p_{-},p_{+})\coloneqq n+\alpha^{\#}(p_{+})-\alpha^{\#}(p_{-})-2\cdot \mathrm{Angle}(d\iota(T_{p_{-}}L),d\iota(T_{p_{+}}L)). \]
Since this is independent of $p_{-}\in L_{a_{-}}$, we set 
\[\deg(L_{a_{-}})\coloneqq \deg(L_{a_{-}},L_{a_{+}}) \coloneqq \deg(p_{-},p_{+}).\]
\end{definition}

Now suppose that we have equipped $L$ with a spin structure. As explained in~\cite[Ch. 8]{FOOOII}, the choice of spin structure on $L$ induces \textit{orientation local systems} on each of the switching components $L_a$. Let $\Theta_a^{-}$ denote the complex line bundle on $L_a$ associated to the orientation local system. For any $a\in A$, let $\Omega^*(L_a;\Theta_a^{-})$ denote the space of smooth differential forms valued in $\Theta_a^{-}$. 

\begin{definition}
As a graded module over $\Lambda_0$, we define the Floer cochain space of a graded spin Lagrangian immersion with clean self-intersections $\iota\colon L\to M$ to be the space obtained from the graded $\mathbb{C}$-vector space
\begin{align}
\overline{CF}^*(L)\coloneqq\Omega^*(L_0)\oplus\bigoplus_{a\in A}\Omega^*(L_a;\Theta^{-}_a)[\deg(L_{a_{-}})]  \label{underlyingvs}
\end{align}
by taking the completed tensor product with $\Lambda$, i.e.
\begin{align}
CF^*(L)\coloneqq\overline{CF}^*(L)\widehat{\otimes}_{\mathbb{C}}\Lambda_0 \, . \label{filteredmodule}
\end{align}
\end{definition}

Fix an $\omega$-compatible almost complex structure $J$ on $M$. The $A_{\infty}$-structure maps on $CF^*(L)$ count $J$-holomorphic disks with boundary and corners on the image $\iota(L)$ of the Lagrangian immersion.
\begin{definition}\label{diskswithcorners}
Let $k\geq-1$ and $\ell\geq0$ be integers. A $J$-holomorphic disk with corners of degree $\beta\in H_2(M,\iota(L);\mathbb{Z})$ consists of the data $(u,\vec{z},\vec{w},\vec{a},\widetilde{\partial u})$, where
\begin{itemize}
\item[(i)] $u\colon D^2\to M$ is $J$-holomorphic, and $u(\partial D^2)\subset \iota(L)$;
\item[(ii)] $[u] = \beta$;
\item[(iii)] $\vec{z} = (z_1,\ldots,z_k)$ is a cyclically ordered collection of mutually distinct boundary marked points, ordered counterclockwise (with respect to the boundary orientation inherited form the complex structure);
\item[(iv)] $\vec{w} = (w_1,\ldots,w_{\ell})$ is an ordered collection of mutually distinct interior marked points;
\item[(v)] $\vec{a} = (a_0,\ldots,a_k)\in A^{k+1}$ assigns a component of $L\times_{\iota}L$ to each boundary marked point $z_i$;
\item[(vi)] $\widetilde{\partial u}\colon\partial D^2\setminus\lbrace z_0,\ldots,z_k\rbrace\to L$ is a smooth map satisfying
\begin{align*}
\iota\circ\widetilde{\partial u} \equiv u\vert_{\partial D^2}
\end{align*}
which asymptotically approaches $L_{a_i}$ as one approaches a boundary marked points, i.e.
\begin{align*}
\left(\lim_{z\to z_i^{-}}\widetilde{\partial u}(z),\lim_{z\to z_i^{+}}\widetilde{\partial u}(z)\right) \in L_{a_i}
\end{align*}
for all $i = 0,\ldots,k$.
\item[(vii)] The set of biholomorphic maps $\phi\colon D^2\to D^2$ for which
\begin{itemize}
\item[(a)] $u\circ\phi = u$;
\item[(b)] $\phi(z_i) = z_i$ and $\phi(w_j) = w_j$;
\item[(c)] $\widetilde{\partial u}\circ\phi = \widetilde{\partial u}$
\end{itemize}
is finite.
\end{itemize}
Two $J$-holomorphic disks with corners are said to be equivalent if they are related by an automorphism as in (vii). The \textit{(uncompactified) moduli space} $\overset{\circ}{\mathcal{M}}_{k+1,\ell}(L;\vec{a};\beta;J)$ of $J$-holomorphic disks is the set of all such equivalence classes. In the case $k=-1$, this should be understood as the moduli space of disks with no boundary marked points, and hence no corners. We will drop $L$ and $J$ from the notation for these moduli spaces when it will not cause confusion.
\end{definition}
The moduli spaces defined in Definition~\ref{diskswithcorners} have Gromov compactifications denoted $\mathcal{M}(L;\vec{a};\beta)$, which are discussed in detail in~\cite[\S{3.2}]{Fuk17}. The only essential difference between the embedded and immersed Lagrangians is that disks with boundary on the latter can form nodes which lie on the switching components, but the data of (vi) above determines how such nodal disks can be glued. The elements of these moduli spaces are represented by bordered stable maps with nodes. Note that by stability, the moduli spaces $\mathcal{M}_{1,0}(L;\vec{a};0)$ and $\mathcal{M}_{2,0}(L;\vec{a};0)$ are empty.

There are natural evaluation maps 
\begin{align*}
\evb_i^{\vec{a};\beta}\colon\mathcal{M}_{k+1,\ell}(L;\vec{a};\beta) &\to L_{a_i} \\
\evi_j^{\vec{a};\beta}\colon\mathcal{M}(L;\vec{a};\beta) &\to M
\end{align*}
for $i = 0,\ldots,k$ and $j = 1,\ldots,\ell$ at the boundary and interior marked points on these moduli spaces. On the loci of irreducible curves, these are given explicitly by the formulas
\begin{align*}
\evb_i^{\vec{a};\beta}\colon\overset{\circ}{\mathcal{M}}_{k+1,\ell}(L;\vec{a};\beta) &\to L_{a_i}\subset L\times_{\iota}L \\
[(u,\vec{z},\vec{w},\vec{a},\widetilde{\partial u})] &\mapsto \left(\lim_{z\to z_i^{-}}\widetilde{\partial u}(z),\lim_{z\to z_i^{+}}\widetilde{\partial u}(z)\right)
\end{align*}
and
\begin{align*}
\evi_j^{\vec{a},\beta}\colon\overset{\circ}{\mathcal{M}}_{k+1,\ell}(L;\vec{a};\beta) &\to M \\
[(u,\vec{z},\vec{w},\vec{a},\widetilde{\partial u})] &\mapsto u(w_j) \, .
\end{align*}
By~\cite[Theorem 3.24, Proposition 3.30]{Fuk17}, one can construct a system of Kuranishi structures on the moduli spaces $\mathcal{M}_{k+1,0}(L;\vec{a};\beta)$, along with continuous families of perturbations on these moduli spaces which allow one to treat the evaluation maps $\evb_0$ as smooth submersions. This means that there is a well-defined pushforward of differential forms.
\begin{remark}\label{kstructstabilization}
Combining this with the results of~\cite{Fuk10}, one can also obtain such Kuranishi structures for the moduli spaces $\mathcal{M}_{k+1,\ell}(L;\vec{a};\beta)$ with little effort. More precisely, we equip these moduli spaces with Kuranishi structures obtained from those on $\mathcal{M}_{k+1}(L;\vec{a};\beta)$ via pullback with respect to the forgetful maps of interior marked points, as in~\cite{Fuk10}. Defining the Kuranishi structures this way means that the evaluation maps $\evi_j^{\vec{a};\beta}$ along interior marked points will \textit{not} be weakly submersive, as explained in~\cite[Remark 3.2]{Fuk10}, but we expect that this is not needed, since it is possible to define open-closed maps valued in differential currents, as we explain below. Since we have allowed ourselves to consider Lagrangian immersions, such a definition should be sufficient for defining cyclic open-closed maps on the full Fukaya category of Definition~\ref{fukcatclean}.
\end{remark}

Fix $\beta\in H_2(M,L)$ and integers $k,\ell\geq0$ for which $(k,\ell,\beta)\not\in\lbrace(0,0,0),(1,0,0)\rbrace$. Also fix a sequence $\vec{a} = (a_0,\ldots,a_k)$. We define degree $1$ operators
\begin{align*}
&\mathfrak{q}_{k,\ell;\vec{a};\beta}\colon(\overline{CF}^*(L)[1])^{\otimes k}\otimes(\Omega^*(M)[2])^{\otimes\ell}\to\overline{CF}^*(L)[1] \\
&\mathfrak{q}_{k,\ell;\vec{a};\beta}(\otimes_{i=1}^k\alpha_i;\otimes_{j=1}^{\ell}\gamma_j)\coloneqq(-1)^{*}(\evb_0^{\vec{a};\beta})_*\left(\bigwedge_{i=1}^k(\evb_i^{\vec{a};\beta})^*\alpha_i\wedge\bigwedge_{j=1}^{\ell}(\evi_j^{\vec{a};\beta})^*\gamma_j\right)
\end{align*}
where $\alpha_i\in\Omega^*(L_{a_i})$ for all $i = 1,\ldots,k$. The sign is determined by the formula
\begin{align}
* = \sum_{i=1}^k i(\deg\alpha_i + 1)+1 \label{gappedsign}
\end{align}
where $\deg\alpha_i$ denotes its degree in $\overline{CF}^*(L)$. These determine maps
\begin{align*}
\mathfrak{q}_{k,\ell;\beta}\colon(\overline{CF}^*(L)[1])^{\otimes k}&\to\overline{CF}^*(L)[1]
\end{align*}
in an obvious way by summing over all sequences $\vec{a}\in A^{k+1}$. Also set
\begin{align*}
\mathfrak{q}_{0,0;0} &= 0 \\
\mathfrak{q}_{1,0;0} &= d
\end{align*}
where $d$ denotes sum of de Rham differentials on the components of $\overline{CF}^*(L)$.

Suppose that $L$ has been equipped with a $U_{\Lambda}$-local system denoted $\nabla$. We can extend these to operators
\begin{align}
\mathfrak{q}_{k,\ell}\colon(CF^*(L)[1])^{\otimes k}\otimes(\Omega^*(M;\Lambda_+)[2])^{\otimes\ell}\to CF^*(L)[1]
\end{align}
by setting
\begin{align}
\mathfrak{q}_{k,\ell}\coloneqq\sum_{\beta\in H_2(M;\iota(L))}\hol_{\nabla}(\partial\beta)\mathfrak{q}_{k,,\ell;\beta}Q^{\omega(\beta)}\otimes\id_{\Lambda_0} \, .
\end{align}
Let $G_L$ denote the smallest submonoid of $\mathbb{R}_{\geq0}$ which contains the symplectic area of every $J$-holomoprhic disk with corners and boundary on $\iota(L)$. As a consequence of Remark~\ref{kstructstabilization}, these operations satisfy a \textit{divisor axiom} with respect to forgetting interior marked points.

\begin{lemma}[Divisor axiom]\label{qass}
For any $\gamma_1\in\Omega^2(X)$ with $d\gamma = 0$, the operators $\mathfrak{q}_{k,\ell;\beta}$ should satisfy the \textit{divisor axiom}
\begin{align}\label{divisorq}
\mathfrak{q}_{k,\ell;\beta}(\otimes_{i=1}^k\alpha_i;\otimes_{j=1}^{\ell}\gamma_j) = \left(\int_{\beta}\gamma_1\right)\cdot\mathfrak{q}_{k,\ell-1;\beta}(\otimes_{i=1}^k\alpha_i;\otimes_{j=2}^{\ell}\gamma_j) \, .
\end{align}
\qed
\end{lemma}
The divisor axiom is used to relate derivatives of the open Gromov--Witten potential to the open Gromov--Witten invariants with interior constraints.

Imitating the proof of~\cite[Proposition 3.35]{Fuk17}, while keeping track of gradings and holonomy, we obtain the following.
\begin{proposition}
For all $k\geq0$, define $\mathfrak{m}_k\coloneqq\mathfrak{q}_{k,0}$. Then $CF^*(L,\nabla)\coloneqq(CF^*(L);\lbrace\mathfrak{m}_k\rbrace_{k=0}^{\infty})$ has the structure of a strictly unital $G_L$-gapped filtered $A_{\infty}$-algebra, as defined in~\cite[Ch. 3.2]{FOOOI}. In particular, the operators $\mathfrak{m})_k$ satisfy the curved $A_{\infty}$-relations
\begin{equation}
\sum_{i,\ell}(-1)^{\maltese_i}\mathfrak{m}_{k-\ell+1}(\alpha_1\otimes\cdots\otimes\mathfrak{m}_{\ell}(\alpha_{i+1}\otimes\cdots\otimes\alpha_{i+\ell})\otimes\cdots\otimes\alpha_k) = 0.\label{ainftyrels}
\end{equation}
\end{proposition}
The strict unit above is given by the constant $0$-form on the diagonal component with value $1$. We have denoted the $A_{\infty}$-algebra by $CF^*(L,\nabla)$ to emphasize the role of the local system. Gappedness is a consequence of Gromov compactness. By~\cite[Remark 3.44]{Fuk17}, the zeroth order part of $\mathfrak{m}_2$ is given by
\begin{align}
\mathfrak{m}_{2,0}(\alpha_1,\alpha_2) = (-1)^{\deg\alpha_1}\alpha_1\wedge\alpha_2
\end{align}
where the right hand side denotes the wedge product of two differential forms on the same connected component of $L\times_{\iota} L$.
\subsection{Horocyclic operators}
To understand how the open Gromov--Witten invariants are determined by the Fukaya category, we will need to introduce operators
\begin{align*}
\mathfrak{q}_{k,\ell;\perp_i}\colon(CF^*(L;\nabla)[1])^{\otimes k}\otimes(\Omega^*(M;\Lambda)[2])^{\otimes\ell}\to CF^*(L;\nabla)[1]
\end{align*}
using moduli spaces of disks with horocyclic constraints, as in~\cite{Hug24b}. Similar moduli spaces are used by~\cite{ST23} to prove flatness of the relative quantum connection. The irreducible loci of the moduli spaces we consider are the subsets $\overset{\circ}{\mathcal{M}}_{k+1,\ell,\perp_i}(\vec{a};\beta)\subset\overset{\circ}{\mathcal{M}}_{k+1,\ell}(\vec{a};\beta)$ defined by requiring the boundary marked point $z_i$ and the interior marked points $w_1$ and $w_2$ to lie on a horocycle in the domain (i.e. a circle in the disk tangent to a point on the boundary, in this case $z_i$). We require that $(z_i,w_1,w_2)$ are ordered counterclockwise on the horocycle. The moduli space $\overset{\circ}{\mathcal{M}}_{k+1,\ell,\perp_i}(\vec{a};\beta)$ can be written as a fiber product
\begin{align*}
\overset{\circ}{\mathcal{M}}_{k+1,\ell,\perp_i}(\vec{a};\beta) = I\times_{D^2}\overset{\circ}{\mathcal{M}}_{k+1,\ell}(\vec{a};\beta) \, .
\end{align*}
Thus, given a Kuranishi structure on ${\mathcal{M}}_{k+1,\ell}(\vec{a};\beta)$, it should also be possible to construct one on ${\mathcal{M}}_{k+1,\ell,\perp_i}(\vec{a};\beta)$. For any $\alpha = \alpha_1\otimes\cdots\otimes\alpha_k$ with $\alpha_i\in\Omega^*(L_{a_i})$ and $\gamma = \gamma_1\otimes\cdots\otimes\gamma_{\ell}\in\Omega^*(M)^{\otimes\ell}$ set
\begin{align*}
\mathfrak{q}_{k,\ell;\vec{a};\beta;\perp_i}(\alpha;\gamma)\coloneqq(-1)^{*_{\perp}}(\evb_0^{\vec{a};\beta})_*\left(\bigwedge_{j=1}^{\ell}(\evi_j^{\vec{a};\beta})^*\gamma_j\wedge\bigwedge_{i=1}^k(\evb_i^{\vec{a};\beta})^*\alpha_i\right) \,
\end{align*}
where
\begin{align*}
*_{\perp}\coloneqq *+\sum_{i=1}^k(\deg\alpha_i+1)+\sum_{j=1}^{\ell}\deg\gamma_j
\end{align*}
for $*$ as in~\eqref{gappedsign}. Then set
\begin{align*}
\mathfrak{q}_{k,\ell;\perp_i}\coloneqq\sum_{\vec{a}\in A^k}\sum_{\beta\in H_2(M,L;\mathbb{Z})}\hol_{\nabla}(\partial\beta)\mathfrak{q}_{k,\ell;\vec{a};\beta;\perp_i} Q^{\omega(\beta)}\otimes\id_{\Lambda_0} \,.
\end{align*}
\subsection{The Fukaya category and the open-closed map}
\begin{definition}
Let $b\in\overline{CF}^1(L)\widehat{\otimes}_{\mathbb{C}}\Lambda_{+}$ be a degree $1$ Floer cocycle of strictly positive valuation (cf. Definition~\ref{ringdef}). We say that it is a \textit{bounding cochain} if it is a solution to the \textit{Maurer--Cartan equation}
\begin{align}
\sum_{k=0}^{\infty}\mathfrak{m}_k(b^{\otimes k}) = 0 \, .
\end{align}
Let $\mathcal{M}(L,\nabla)$ denote the space of equivalence classes of bounding cochains on $CF^*(L,\nabla)$ up to \textit{gauge-equivalence}, in the sense of~\cite[Definition 4.3.1]{FOOOI}.
\end{definition}
In the presence of a bounding cochain, one can define uncurved $A_{\infty}$-operations
\begin{align*}
\mathfrak{m}_k^b(\alpha_1\otimes\cdots\otimes\alpha_k)\coloneqq\sum_{\substack{\ell\geq0 \\ \ell_0+\cdots+\ell_k = \ell}}\mathfrak{m}_{k+\ell}(b^{\otimes\ell_0}\otimes\alpha_1\otimes b^{\ell_1}\otimes\cdots\otimes b^{\ell_{k-1}}\otimes\alpha_k\otimes b^{\ell_k})
\end{align*}
where the sum is over all $k$-part partitions of arbitrary length. In this situation, we define the \textit{Lagrangian Floer homology} $HF^*(L,\nabla,b)$ to be the cohomology of $CF^*(L)$ with respect to the deformed differential $\mathfrak{m}_1^b$. The objects of the Fukaya category should roughly be \textit{unobstructed} Lagrangians. We formalize this notion in the following definition.

\begin{definition}
An \textit{(unobstructed) Lagrangian brane} consists of a Lagrangian immersion $\iota\colon L\to M$ with clean self-intersections, together with a rank one $U_{\Lambda}$-local system $\nabla$, a spin structure on $L$, a grading $\alpha^{\#}\colon L\to\mathbb{R}$, and a bounding cochain $b\in CF^1(L)$.
\end{definition}

The construction of (any finite subcategory of) the Fukaya category $\mathcal{F}(M)$, with coefficients in the Novikov field $\Lambda$, can be reduced to the construction of the $A_{\infty}$-algebra of a single \textit{immersed} Lagrangian. Specifically, fix a finite collection of unobstructed Lagrangian branes $\mathbb{L}\coloneqq\lbrace L_c\colon c\in\mathfrak{D}\rbrace$, where $\mathfrak{D}$ is an index set, which intersect each other cleanly. In particular, there is a clean Lagrangian immersion $L_{\mathfrak{D}}\coloneqq\coprod_{c\in\mathfrak{D}}L_c$ obtained by taking the union of Lagrangian immersions over $\mathfrak{D}$. We give this a brane structure by requiring the bounding cochain on $L_{\mathfrak{D}}$ to vanish on components of $L_{\mathfrak{D}}\times_M L_{\mathfrak{D}}$ corresponding to components of $L_c\cap L_{c'}$ for $c\neq c'\in\mathfrak{D}$ (and to restrict to the given bounding cochains on each $L_c$). The local system, spin structure, and grading are defined on the domain of an immersion, so the correct choices of such data for $L_{\mathfrak{D}}$ are clear. 
\begin{definition}\label{fukcatclean}
The Fukaya category $\mathcal{F}(M)$ has objects $\Ob\mathcal{F}(M)\coloneqq\mathbb{L}$. The hom sets and $A_{\infty}$-composition maps on $\mathcal{F}(M)$ are inherited from $CF^*(L_{\mathfrak{D}})\otimes\Lambda$ (cf.~\cite[\S{3.4}]{Fuk17}).
\end{definition}
\begin{remark}
This construction depends on a choice of branes $\mathbb{L}$. In practice, we can usually take a collection of objects whose disjoint union satisfies Abouzaid's generation criterion. For Calabi--Yau hypersurfaces in toric Fano varieties, a set of generators is identified in~\cite{She15}, and we can adjoin any finite collection of Lagrangian branes to this collection. This means that the split-derived category  $D^{\pi}\mathcal{F}(M)$ obtained from $\mathcal{F}(M)$ will be independent of the choice of $\mathbb{L}$ provided that it contains a suitable generating set. In particular, this construction should suffice for many applications in homological mirror symmetry. We remark that the proofs of our main results essentially only use \textit{homological} data, so we do not lose information by passing to the derived category.
\end{remark}

\begin{remark}
Recall that the proofs of homological mirror symmetry in~\cite{She15, SS21, GHHPS24} all involve computations in the \textit{relative Fukaya category} $\mathcal{F}(X,D)$, as constructed in~\cite{PS24}. Let $(X,\omega)$ be a closed symplectic manifold and $D\subset X$ a simple normal crossings divisor Poinca{\'e} dual to $D$. The objects of the relative Fukaya category are the same as the objects of the Fukaya category of compact (exact) Lagrangians $\mathcal{F}(X\setminus D)$, but the $A_{\infty}$-operations on $\mathcal{F}(X,D)$ count pseudoholomorphic disks in $X$ weighted by their intersection number with $D$. Moreover, it is natural to require that the objects of $\mathcal{F}(X\setminus D)$ only carry $\mathbb{C}^*$-local systems, since the Fukaya category of exact Lagrangians is naturally defined over $\mathbb{C}$, not just the Novikov field. A suitable analogue of the $A_{\infty}$ category of~\cite[Definition 1.6]{PS24}, which has coefficients in a universal Novikov field, can be obtained by defining $\mathcal{F}(X)$ as in Definition~\ref{fukcatclean} using a collection of objects that avoids $D$. Thus, in practice, we will not distinguish between the full and relative Fukaya categories.
\end{remark}

One benefit of this construction of the Fukaya category is that it allows us to define the open-closed map
\begin{align}
\mathcal{OC}_0\colon HH_*(\mathcal{F}(M))\to H^{*+n}(M;\Lambda_0)\label{ocfukcat}
\end{align}
in terms of the open-closed map for the $A_{\infty}$-algebra of an immersed Lagrangian. One would expect that this should be defined by pulling back differential forms on a Lagrangian immersion $L$ to an open-closed moduli space, and pushing the wedge product of these forms forward along the evaluation map at an interior output marked point $w_{\out}$ (using integration over the fiber). The techniques of~\cite{Fuk11} do not, however, imply that the evaluation maps $\evi_j$ at the interior marked points are submersions, but we can use this pairing to define the open-closed map whose chain-level outputs are differential currents. Given any sequence $\vec{a} = (a_1,\ldots,a_k)\in A^{k}$, relabel the boundary marked points on $\mathcal{M}_{k,1}(L;\vec{a};\beta)$ by $\vec{z} = (z_1,\ldots,z_k)$ and the interior marked point by $w_{\out}$ to get evaluation maps
\begin{align*}
\evb_i^{\vec{a};\beta}\colon\mathcal{M}_{k,1}(L;\vec{a};\beta)\to L_{a_i}
\end{align*}
for all $i = 1,\ldots,k$, and
\begin{align*}
\evi^{\vec{a};\beta}\colon\mathcal{M}_{k,1}(L;\vec{a};\beta)\to M \, .
\end{align*}
Let $\mathcal{D}^*(M;\mathbb{C})$ denote the space of differential currents on $M$ with coefficients in $\mathbb{C}$. For a sequence $\alpha = (\alpha_1,\ldots,\alpha_k)$, where $k\geq1$ and $\alpha_i\in\Omega^*(L_{a_i})$ for all $i = 1,\ldots,k$, we define the differential current
\begin{align*}
\mathfrak{p}_{k}^{\vec{a};\beta}(\alpha)\coloneqq(\evi^{\vec{a};\beta})_*\left(\bigwedge_{i=1}^k (\evb_i^{\vec{a};\beta})^*\alpha_i\right)\in\mathcal{D}^*(M;\mathbb{C}) \, .
\end{align*}
Here the pushforward is only a current because we cannot assume that $\evi^{\vec{a};\beta}$ is a submersion in any sense per Remark~\ref{kstructstabilization}. As before, we can extend this to a map
\begin{align*}
\mathfrak{p}_{k}\colon(CF^*(L;\nabla)[1])^{\otimes k}&\to\mathcal{D}^*(M;\Lambda) \\
\mathfrak{p}_{k} &\coloneqq\sum_{\vec{a}\in A^k}\sum_{\beta\in H_2(M,L)}\hol_{\nabla}(\partial\beta)\mathfrak{p}_{k}^{\vec{a};\beta} Q^{\omega(\beta)} \, .
\end{align*}
Examining the boundary strata of $\mathcal{M}_{k,1}(\beta)$ (cf.~\cite[Lemma 2.14]{T23}) shows that this induces a map
\begin{align}
\mathcal{OC}_0\colon HH_*(CF^*(L,\nabla))\to H^{*+n}(M;\Lambda_0)
\end{align}
on the Hochschild homology of $CF^*(L,\nabla)$ called the open-closed map. Letting $L_{\mathfrak{D}}$ denote the Lagrangian immersion of Definition~\ref{fukcatclean}, we obtain~\eqref{ocfukcat}.

\subsection{Energy spectral sequence}
Applying the algebraic results of~\cite[Ch. 6]{FOOOI} to the current setting, we can construct a spectral sequence from the energy filtration on $CF^*(L)$. Note that the description of the $E_2$-page of the spectral sequence simplifies as compared to~\cite[Theorem D]{FOOOI} because we have equipped $L$ with a grading (and the grading on our Novikov ring is trivial).
\begin{proposition}\label{spectralseq}
Let $\nabla$ be a rank one $U_{\Lambda}$-local system on $L$, and let $b$ be a bounding cochain for $(L,\nabla)$. Then for any sufficiently small positive real number $\epsilon_0>0$, there is a spectral sequence for which
\begin{itemize}
\item[(i)]
\begin{align*}
E_2^{p,q} = H^p(\overline{CF}^*(L))\otimes(Q^{q\epsilon_0}\Lambda_0/Q^{(q+1)\epsilon_0}\Lambda_0)
\end{align*}
where $H^p(\overline{CF}^*(L))$ is a complex vector space given by taking the degree $p$ cohomology of $\overline{CF}^*(L)$ equipped with the de Rham differential; and
\item[(ii)] there is a filtration $F^{\bullet}HF^*(L,\nabla,b)$ on the Floer cohomology of $L$ such that
\begin{align*}
E^{p,q}_{\infty}\cong F^q HF^p(L,\nabla,b)/F^{q+1} HF^p(L,\nabla,b) \, .
\end{align*}
\end{itemize}
\end{proposition}
\begin{proof}
The existence of an appropriate constant $\epsilon_0>0$ follows because $CF^*(L,\nabla)$ is $G_L$-gapped. By~\cite[Theorem 5.4.2]{FOOOI}, one can construct a canonical model for $CF^*(L,\nabla)$, meaning that it is weakly finite in the sense of~\cite[Definition 6.3.27]{FOOOI}. The result then follows from~\cite[Theorem 6.3.28]{FOOOI}.
\end{proof}
\subsection{Pseudo-isotopies}\label{pseudoisotopiessection}
Studying the behavior of the open Gromov--Witten potential under changes of the almost complex structure requires an explicit geometric model of the cylinder object $CF^*(L,\nabla)\times[0,1]$ in the category of gapped filtered $A_{\infty}$-algebras. Consider the graded $\mathbb{C}$-vector space
\begin{align*}
\overline{CF}^*(L\times[0,1])\coloneqq\Omega^*([0,1]\times L_0)\oplus\bigoplus_{a\in A\setminus\lbrace 0\rbrace}\Omega^*([0,1]\times L_a;\Theta_a^{-})[\deg(L_a)]
\end{align*}
from which we obtain the $\Lambda_{+}$-module
\begin{align*}
CF^*([0,1]\times L)\coloneqq\overline{CF}^*([0,1]\times L)\widehat{\otimes}_{\mathbb{C}}\Lambda_{+} \, .
\end{align*}
The results of~\cite[\S{14}]{Fuk17} imply that for any path $\underline{J} = \lbrace J_t\rbrace_{t\in[0,1]}$ of $\omega$-compatible almost complex structures on $M$, one can construct a gapped filtered $A_{\infty}$-structure $CF^*(L\times[0,1])$. Consider the moduli spaces
\begin{align*}
\widetilde{\mathcal{M}}_{k+1}(L;\vec{a};\beta;\underline{J})\coloneqq\lbrace (t,[(u,\vec{a},\vec{z},\widetilde{\partial u})])\colon t\in[0,1] \, , \; [(u,\vec{a},\vec{z},\widetilde{\partial u})]\in\mathcal{M}_{k+1}(L;\vec{a};\beta;J_t)\rbrace 
\end{align*}
which are defined for any $k\geq-1$. There are natural evaluation maps
\begin{align*}
\widetilde{\evb}_i^{\vec{a};\beta}\colon\widetilde{\mathcal{M}}_{k+1}(L,\vec{a};\beta;\underline{J})\to[0,1]\times(L\times_{\iota} L) \, .
\end{align*}
By~\cite[Proposition 14.17]{Fuk17}, the evaluation maps $\widetilde{\evb}_0^{\vec{a};\beta}$ at the zeroth boundary marked points can be treated as smooth submersions, meaning that one can make sense of integration over the fiber with respect to these maps. For all $k\geq0$ with $(k,\beta)\neq(1,0)$ and any $\vec{a} = (a_0,\ldots,a_k)\in A^{k+1}$, we define degree $1$ linear maps
\begin{align*}
\widetilde{\mathfrak{m}}_{k;\vec{a};\beta}\colon(\overline{CF}^*([0,1]\times L)[1])^{\otimes k}&\to\overline{CF}^*([0,1]\times L)[1] \\
\widetilde{\mathfrak{m}}_{k;\vec{a};\beta}(\widetilde{\alpha}_1\otimes\cdots\otimes\widetilde{\alpha}_k)&\coloneqq(-1)^*(\widetilde{\evb}_0^{\vec{a};\beta})_*\left(\bigwedge_{i=1}^k(\widetilde{\evb}^{\vec{a};\beta}_i)^*\widetilde{\alpha}_i\right) \, .
\end{align*}
The sign $(-1)^*$ is defined the same way as in~\eqref{gappedsign}. By summing over all sequences $\vec{a}\in A^{k+1}$, we obtain operations $\widetilde{\mathfrak{m}}_{k,\beta}$. As before, set $\widetilde{\mathfrak{m}}_{0,0} = 0$ and $\widetilde{\mathfrak{m}}_{1,0}$. Now suppose that $L$ is equipped with a rank one $U_{\Lambda}$-local system. We extend these operations to $CF^*([0,1]\times L)[1]$ to obtain the $A_{\infty}$-structure maps
\begin{align*}
\widetilde{\mathfrak{m}}_k\colon(CF^*([0,1]\times L,\nabla)[1])^{\otimes k} &\to CF^*([0,1]\times L,\nabla)[1] \\
\widetilde{\mathfrak{m}}_k &\coloneqq \sum_{\beta}\hol_{\nabla}(\partial\beta)\widetilde{\mathfrak{m}}_{k;\beta} Q^{\omega(\beta)} \, .
\end{align*}
These operations define a pseudo-isotopy in the sense of~\cite[Definition 3.36]{Fuk17} by~\cite[Lemma 21.31]{FOOO20}.

There is also an analogue of the open-closed map on a pseudo-isotopy. These are defined using moduli spaces of $\underline{J}$-holomorphic curves with an interior marked point, where $k\geq1$ and $\vec{a}\in A^k$, defined by
\begin{align*}
\widetilde{\mathcal{M}}_{k,1}(L;\vec{a};\beta;\underline{J})\coloneqq\lbrace (t,[u])\colon t\in[0,1] \, , \; [u]\in\mathcal{M}_{k,1}(L;\vec{a};\beta;J_t)\rbrace \, .
\end{align*}
where we abuse notation and write $[u]$ for
\begin{align*}
[(u,\vec{a},\vec{z},w_{\out},\widetilde{\partial u})]\in\mathcal{M}_{k,1}(L;\vec{a};\beta;J_t)\, .
\end{align*}
We have relabeled the boundary marked points $\vec{z} = (z_1,\ldots,z_k)$ and the interior marked point $w_{\out}$. These moduli spaces carry evaluation maps
\begin{align*}
\widetilde{\evb}_i^{\vec{a};\beta}\colon\widetilde{\mathcal{M}}_{k,1}(L;\vec{a};\beta) &\to L_{a_i} \\
\widetilde{\evi}^{\vec{a};\beta}\colon\widetilde{\mathcal{M}}(L;\vec{a};\beta) &\to M \, .
\end{align*}
Since $\widetilde{\evi}$ cannot be treated as a submersion, we have linear operators $\widetilde{\mathfrak{p}}_k^{\vec{a};\beta}$, for all $\beta\in H_2(M,L)$, valued in differential currents $\mathcal{D}^*([0,1]\times M;\mathbb{C})$ defined by
\begin{align*}
\widetilde{\mathfrak{p}}_{k}^{\vec{a};\beta}(\widetilde{\alpha})\coloneqq(\widetilde{\evi}^{\vec{a};\beta})_*\left(\bigwedge_{i=1}^k (\widetilde{\evb}_i^{\vec{a};\beta})^*\widetilde{\alpha}_i\right)\in\mathcal{D}^*([0,1]\times M;\mathbb{C}) \, .
\end{align*}
for any sequence $\widetilde{\alpha} = (\widetilde{\alpha}_1,\ldots,\widetilde{\alpha}_k)$ with $\widetilde{\alpha}_i\in\Omega^*([0,1]\times L_{a_i})$ for all $i = 1,\ldots,k$. These extend to linear maps
\begin{align*}
\widetilde{\mathfrak{p}}_{k}\colon(CF^*([0,1]\times L;\nabla)[1])^{\otimes k}&\to\mathcal{D}^*([0,1]\times M;\Lambda) \\
\widetilde{\mathfrak{p}}_{k} &\coloneqq\sum_{\vec{a}\in A^k}\sum_{\beta\in H_2(M,L)}\hol_{\nabla}(\partial\beta)\widetilde{\mathfrak{p}}_{k}^{\vec{a};\beta} Q^{\omega(\beta)} \, .
\end{align*}
Similarly to the time-independent case (cf.~\cite[Lemma 2.18]{T23}), these operators can be assembled to obtain a map
\begin{align}
\widetilde{\mathcal{OC}}_0\colon HH_*(CF^*([0,1]\times L,\nabla))\to H^{*+n}([0,1]\times M;\Lambda_0) \, .
\end{align}

\subsection{Wall-crossing terms}\label{wallcrossingsubsection}
We will also need to consider counts of closed curves when studying the dependence of the open Gromov--Witten invariants on the almost complex structure $J$ used to define them. Let $\Sigma$ be a closed connected nodal Riemann surface of genus zero and let $J$ be an $\omega$-compatible almost complex structure on $M$. For any nonzero $\beta\in H_2(M;\mathbb{Z})$, consider the moduli space $\mathcal{M}_1^{cl}(\beta;J)$ of stable $J$-holomorphic maps $u\colon\Sigma\to M$ with one marked point, modulo reparametrizations. The elements of this moduli space are denoted $[(u,w)]$, where $w\in\Sigma$ denotes the marked point. Let $\evi^{\beta}\colon\mathcal{M}_1^{cl}(\beta;J)\to M$ denote the evaluation map defined by $\evi^{\beta}([(u,w)]) = u(w)$. The results of~\cite{Fuk10} imply that this moduli space can be equipped with a Kuranishi structure with respect to which $\evi^{\beta}$ can be treated as a smooth submersion. With this we define operators
\begin{align}
{\mathfrak{m}}_{\emptyset}^{\beta} &\coloneqq ({\evi}^{\beta})_*(1)\in\Omega^*([0,1]\times M) \nonumber \\
{\mathfrak{m}}_{\emptyset} &\coloneqq \sum_{\beta\in H_2(M;\mathbb{Z})}{\mathfrak{m}}_{\emptyset}^{\beta} Q^{\omega(\beta)} \in\Lambda_{+} \, . \label{closedoperators}
\end{align}
For a time-dependent $\omega$-compatible almost complex structure $\underline{J} = \lbrace J_t\rbrace_{t\in[0,1]}$, we define the moduli space
\begin{align}
\widetilde{\mathcal{M}}_1^{cl}(\beta;\underline{J})\coloneqq\lbrace (t,[(u,w)])\colon [(u,w)]\in\mathcal{M}_1^{cl}(\beta;J_t) \, , \; t\in[0,1]\rbrace \, . \label{ocmodpseudo}
\end{align}
There is a natural evaluation map
\begin{align*}
\widetilde{\evi}^{\beta}\colon\widetilde{\mathcal{M}}_1^{cl}(\beta;\underline{J}) &\to [0,1]\times M \\
\widetilde{\evi}^{\beta}(t,[u,w]) &\coloneqq (t,u(w)) \, .
\end{align*}
The results of~\cite{Fuk10} imply that this moduli space can be equipped with a Kuranishi structure with respect to which $\widetilde{\evi}^{\beta}$ can be treated as a smooth submersion. With this we define operators
\begin{align}
\widetilde{\mathfrak{m}}_{\emptyset}^{\beta} &\coloneqq (\widetilde{\evi}^{\beta})_*(1)\in\Omega^*([0,1]\times M) \nonumber \\
\widetilde{\mathfrak{m}}_{\emptyset} &\coloneqq \sum_{\beta\in H_2(M;\mathbb{Z})}\widetilde{\mathfrak{m}}_{\emptyset}^{\beta} Q^{\omega(\beta)} \in\Lambda_{+} \, . \label{pseudoclosedoperators}
\end{align}
For a Lagrangian immersion $\iota_L\colon L\to M$, we define a \textit{wall-crossing term}
\begin{align}
\widetilde{GW}(L)\coloneqq\int_L\iota_L^*\widetilde{\mathfrak{m}}_{\emptyset} \label{wallcrossingterm}
\end{align}
which can be interpreted as a count of closed curves in $M$ which intersect $L$ in a point.

\section{Assumptions on the cyclic open-closed map}\label{assumptionsappendix}
To construct the open Gromov--Witten invariants of a graded Lagrangian in a Calabi--Yau $3$-fold and relate them to the Fukaya category, we need to assume that the Fukaya category of~\cite{Fuk17} satisfies some additional algebraic properties. Analogues of these assumptions, in a different model for the Fukaya category, have been announced in~\cite{GPS15}. All of these assumptions require one to consider moduli spaces of disks with immersed Lagrangian boundary conditions and \textit{interior} marked points, which are absent from~\cite{Fuk17}.

\begin{assumption}\label{generationass}
The full Fukaya category $\mathcal{F}(X)$ of Definition~\ref{fukcatclean}, defined with a suitable choice of objects, satisfies all of the assumptions listed in~\cite[\S{2.5}]{SS21}.
\end{assumption}
These roughly say that $\mathcal{F}(X)$ should have enough algebraic structure to mimic the proof of Abouzaid's generation criterion~\cite{Abo10}. Such structures include closed-open and open-closed maps, about which we must make some additional assumptions for our treatment of (open) Gromov--Witten theory.

The closed-open operations are essentially given by the $\mathfrak{q}$-operators on $CF^*(L,\nabla)$. To relate the counts of disks $\mathfrak{m}_{-1}$ to the open-closed map, we will need to introduce operators defined using moduli spaces of disks with horocyclic constraints. These operations are usually used to show that the closed-open and open-closed maps give $HH_*(\mathcal{F}(M))$ the structure of a $QH^*(X)$-module, and thus verifying the following assumption would most likely be a byproduct of verifying Assumption~\ref{traceass} below.

Genus zero Gromov--Witten invariants are related to the Fukaya category in~\cite{GPS15} and~\cite{Hug24b} via a \textit{(negative) cyclic open-closed map}. A construction of a cyclic open-closed map (on the relative Fukaya category) satisfying the assumptions below was announced in~\cite{GPS15}. We only require that such a map has been constructed with values in the cohomology of $X$, as opposed to the de Rham complex.

For a review of Hocschild and cyclic homology, see e.g.~\cite{She20, Gan23, Han24b}. Let $\mathcal{A}$ be a (possibly curved) $A_{\infty}$-algebra, and let $(CH_{\bullet}(\mathcal{A}),b)$ denote its Hochschild chain complex, where $b$ is the cyclic bar differential. Also let $B$ denote the Connes operator on $CH_{\bullet}(\mathcal{A})$. These operators give $CH_{\bullet}(\mathcal{A})$ the structure of a strict $S^1$-complex (cf.~\cite[\S{2}]{Gan23}). When $\mathcal{A}$ is the curved $A_{\infty}$-algebra associated to a Lagrangian immersion equipped with a local systems, we require the existence of a cyclic open-closed map which intertwines this action with the trivial $S^1$-action on quantum cohomology. Versions of the cyclic open-closed map are used both to define the open Gromov--Witten invariants following~\cite{Han24b}, and to extract them from the Fukaya category.
\begin{assumption}\label{traceass}
Let $\mathcal{A}\coloneqq CF^*(L,\nabla)$ denote the (curved) $A_{\infty}$-algebra associated to a clean graded Lagrangian immersion in a Calabi--Yau $n$-fold. Then there is a sequence of maps $\lbrace\mathcal{OC}_m\rbrace_{m=0}^{\infty}$ of the form
\begin{align*}
\mathcal{OC}_m\colon CH_*(\mathcal{A})\to QH^{*+n-2m}(X)
\end{align*}
where $\mathcal{OC}_0$ induces the open-closed map of~\eqref{ocfukcat}. Additionally, these maps should satisfy
\begin{align*}
\mathcal{OC}_{m-1}\circ B + \mathcal{OC}_m\circ b = 0 \,.
\end{align*}
\end{assumption}
In principle, we could mimic the construction of~\cite{Gan23} or~\cite{Han24b} to verify Assumption~\ref{traceass}, though we prefer not to make any assumptions about how the cyclic open-closed map is defined at the chain level, given that one would also need to verify Assumption~\ref{connectionsass}. Let $u$ be a formal variable of degree $2$, and consider the positive cyclic chain complex
\begin{align*}
CC_{\bullet}^{+}(\mathcal{A}) = (CH_{\bullet}(\mathcal{A})\otimes_{\Lambda}\Lambda((u))/u\Lambda[[u]], b_{eq})
\end{align*}
where $b_{eq}\coloneqq b+uB$. The maps of Assumption~\ref{traceass} determine a chain map
\begin{align*}
\mathcal{OC}^{+}\colon CC_{\bullet}^{+}(\mathcal{A}) &\to QH^{\bullet+n}(X)\otimes_{\Lambda}\Lambda((u))/u\Lambda[[u]] \\
\mathcal{OC}^{+} &\coloneqq \sum_{m=0}^{\infty}\mathcal{OC}_m u^m \,.
\end{align*}
By projecting with the $u^0$-factor and integrating over $X$, we obtain a trace map
\begin{align}\label{cctrace}
\widetilde{tr}\colon CC_*^{+}(\mathcal{A})\to\Lambda[-n] \,.
\end{align}
The construction of the open Gromov--Witten invariants in~\cite{Han24b} only requires~\eqref{cctrace}.

One also has a negative cyclic chain complex 
\begin{align*}
CC_{\bullet}^{-}(\mathcal{A}) = (CH_{\bullet}(\mathcal{A})\widehat{\otimes}_{\Lambda}\Lambda[[u]], b_{eq}) \,.
\end{align*}
Let $HC_{\bullet}^{-}(\mathcal{F}(X))$ denote the negative cyclic homology of Fukaya category equipped with the \textit{Getzler--Gauss--Manin} connection $\nabla^{\mathrm{GGM}}$. Similarly, let 
\begin{align*}
QH^*(X;\Lambda)[[u]]\coloneqq QH^*(X;\Lambda)\widehat{\otimes}_{\Lambda}\Lambda[[u]]
\end{align*}
denote the quantum cohomology of $X$. The quantum connection is given in terms of the (small) quantum product $\omega\star(\cdot)$ by
\begin{align*}
\nabla^{\mathrm{QDE}} = Q\partial_Q(\cdot)+u^{-1}\omega\star(\cdot) \, .
\end{align*}
\begin{assumption}\label{connectionsass}
The negative cyclic open-closed map
\begin{align*}
\mathcal{OC}^{-}\colon HC_{\bullet}^{-}(\mathcal{F}(X))\to QH^{\bullet+n}(X)[[u]]
\end{align*}
induced by the maps in Assumption~\ref{traceass} respects connections, in the sense that
\begin{align*}
\mathcal{OC}^{-}\circ\nabla^{\mathrm{GGM}}_{Q\partial_Q} = \nabla^{\mathrm{QDE}}_{Q\partial_Q}\circ\mathcal{OC}^{-} \, .
\end{align*}
\end{assumption}

\begin{remark}
Let $\mathcal{A}$ be a curved $A_{\infty}$-algebra $\mathcal{A}$ for which $H^*(\mathcal{A},\mathfrak{m}_{1,0})$ is finite-dimensional. We say that a chain map $tr\colon CH_{*}(\mathcal{A})\to\Lambda[-n]$ is a weak proper Calabi--Yau structure if the composition
\begin{align}
H^*(\mathcal{A},\mathfrak{m}_{1,0})\otimes H^{n-*}(\mathcal{A},\mathfrak{m}_{1,0})\xrightarrow{\mathfrak{m}_{2,0}(\cdot,\cdot)} H^n(\mathcal{A},\mathfrak{m}_{1,0})\to HH_n(\mathcal{A})\xrightarrow{tr}\Lambda
\end{align}
is a perfect pairing. Restricting the trace $\widetilde{tr}$ of~\eqref{cctrace} to Hochschild chains yields a weak proper Calabi--Yau structure in the sense above, and is said to be a \textit{stronger} proper Calabi--Yau structure. We do not necessarily claim that this is the optimal definition of a proper Calabi--Yau structure on a curved $A_{\infty}$-algebra, but in the presence of a trace as in Assumption~\ref{traceass}, such a structure will exist.
\end{remark}

\begin{remark}[Forgetting interior marked points]\label{forgetinterior}
The verifications that the cyclic open-closed maps are chain maps in~\cite{Gan23} and~\cite{Han24b} both invoke forgetful maps of interior marked points, which is reasonable in our setting in view of Remark~\ref{kstructstabilization}. Assumptions~\ref{traceass} and~\ref{connectionsass} are both phrased in such a way that they would follow from the construction of a chain-level cyclic open-closed map valued in \textit{differential currents}, as in. Hugtenburg's construction~\cite{Hug24a} of the (cyclic) open-closed map for a single embedded Lagrangian is phrased in terms of Poincar{\'e} duality on $X$ essentially for this reason. By Definition~\ref{fukcatclean}, we can use this strategy to define cyclic open-closed maps on the entire Fukaya category provided that we can do so on the Floer cochain space of a single \textit{immersed} Lagrangian.
\end{remark}

\bibliographystyle{abbrv}
\bibliography{myref}

\end{document}